\documentclass[a4paper,oneside,11pt,reqno]{amsart}
\usepackage[margin=30mm]{geometry}
\usepackage[utf8]{inputenc}
\usepackage[english]{babel}
\usepackage{amsmath,amssymb,amsthm,amsfonts,amsopn}
\usepackage{stmaryrd}
\usepackage{dsfont} 
\usepackage{empheq}
\usepackage{enumerate,enumitem}

\usepackage[table, dvipsnames]{xcolor}
\usepackage{arydshln,nicematrix}

\usepackage{caption}
\usepackage{subcaption}
\usepackage{graphicx}
\usepackage{pgfplots}
    \pgfplotsset{compat=newest}

\usepackage{ifthen}
\usepackage{float}

\theoremstyle{plain}
\newtheorem{theorem}{Theorem}[section]
\newtheorem{lemma}[theorem]{Lemma}
\newtheorem{corollary}[theorem]{Corollary}
\newtheorem{proposition}[theorem]{Proposition}
\newtheorem{definition}[theorem]{Definition}
\newtheorem{conjecture}[theorem]{Conjecture}

\newcommand\xqed[1]{%
	\leavevmode\unskip\penalty9999 \hbox{}\nobreak\hfill\quad\hbox{#1}%
}
\newcommand\remarkend{\xqed{$\triangle$}}
\makeatletter
    \def\@endtheorem{\remarkend\endtrivlist\@endpefalse }
\makeatother
\theoremstyle{remark}

\newtheorem*{remark*}{Remark}
\makeatletter
    \def\@endtheorem{\endtrivlist\@endpefalse }
\makeatother

\renewcommand{\leq}{\leqslant}
\renewcommand{\geq}{\geqslant}
\renewcommand{\cosh}{\operatorname{ch}}
\renewcommand{\sinh}{\operatorname{sh}}
\renewcommand{\tanh}{\operatorname{th}}

\renewcommand{\Im}{\operatorname{Im}}
\newcommand{\rank}{\operatorname{rank}}

\let\div\relax
\DeclareMathOperator{\div}{div}

\newcommand{\N}{\mathbb{N}} 
\newcommand{\C}{\mathbb{C}} 
\newcommand{\Z}{\mathbb{Z}} 
\newcommand{\R}{\mathbb{R}} 

\renewcommand{\epsilon}{\varepsilon}
\newcommand{\ii}{\mathrm{i}\mkern1mu}
\newcommand{\bu}{{\bf u}}
\newcommand{\bx}{{\bf x}}

\newcommand{\be}{{\bf e}}
\newcommand{\bC}{{\bf C}}

\newcommand{\beps}{{\boldsymbol \varepsilon}}
\newcommand{\btau}{{\boldsymbol \tau}}

\newcommand{\bnabla}{{\boldsymbol \nabla}}
\newcommand{\bxi}{{\boldsymbol \xi}}

\newcommand{\norm}[1]{\left\vert\kern-0.25ex\left\vert #1 \right\vert\kern-0.25ex\right\vert}
\newcommand{\normNS}[1]{\vert\kern-0.25ex\vert #1 \vert\kern-0.25ex\vert}
\newcommand*{\transp}{^{\mkern-1.5mu\mathsf{T}}}
\newcommand{\di}{\,\mathrm{d}}

\definecolor{BleuSombre}{rgb}{0,0,0.6}
\definecolor{RougeSombre}{rgb}{0.8,0,0}
\definecolor{VertSombre}{rgb}{0,0.6,0}

\usepackage[plainpages=false,colorlinks,linkcolor=BleuSombre, citecolor=RougeSombre,urlcolor=VertSombre,breaklinks]{hyperref} 
\usepackage[capitalize]{cleveref} 

\allowdisplaybreaks

\title[Inverse problem for Love waves in a layered, elastic half-space]{Inverse problem for Love waves\\in a layered, elastic half-space}
 
\author[M. V. de Hoop]{Maarten V. de Hoop}
	\address{Simons Chair in Computational and Applied Mathematics and Earth Science, Rice University, Houston, TX 77005}
	\email{\href{mdehoop@rice.edu}{mdehoop@rice.edu}}

\author[J. Garnier]{Josselin Garnier}
	\address{CMAP, CNRS, {\'E}cole polytechnique, Institut Polytechnique de Paris, 91120 Pa\-lai\-seau, France}
	\email{\href{josselin.garnier@polytechnique.edu}{josselin.garnier@polytechnique.edu}}

\author[A. Iantchenko]{Alexei Iantchenko}
	\address{Department of Materials and Applied Mathematics, Faculty of Technology and Society, Malm{\"o} University, Malm{\"o}, Sweden}
	\email{\href{ai@mau.se}{ai@mau.se}}

\author[J. Ricaud]{Julien Ricaud${}^\ddagger$}\thanks{${}^\ddagger$Corresponding author.}
	\address{CMAP, CNRS, {\'E}cole polytechnique, Institut Polytechnique de Paris, 91120 Pa\-lai\-seau, France}
	\email[Corresponding author]{\href{julien.ricaud@polytechnique.edu}{julien.ricaud@polytechnique.edu}}


\begin{document}

\begin{abstract}
    In this paper we study Love waves in a layered, elastic half-space. 
    We first address the direct problem and we characterize the existence of Love waves through the dispersion relation. 
    We then address the inverse problem and we show how to recover the parameters of the elastic medium from the empirical knowledge of the frequency--wavenumber couples of the Love waves.
\end{abstract}

\maketitle
\tableofcontents

\section{Introduction}\label{Section_introduction}
The paper is motivated by applications in seismology. Surface wave tomography has been used for a long time in global seismology to image crustal and upper mantle structures. It consists in extracting the dispersion curves of the surface waves (i.e., their frequency-dependent velocities). From those curves the three-dimensional map of the parameters of the elastic medium can be deduced tomographically. This last step is the topic of our paper.

Surface wave tomography was first used with natural seismic events~\cite{TraWoo-96,EksTroLar-97,BosDzi-99}.
It has recently attracted attention because it was shown that it can be used with low-frequency seismic ambient noise~\cite{ShaCamSteRit-05,LinMosRit-08,NisMonKaw-09} or both types of data (ambient-noise and earthquakes)~\cite{KasElSBosMeiRosBelCriWei-18}.
Indeed, surface waves can be easily extracted from ambient noise signals~\cite{ShaCam-04,GarPap}, because they dominate the Green function between receivers located at the surface and because ambient seismic noise is mostly excited by superficial sources, such as oceanic microseisms, ocean swell, and atmospheric disturbances~\cite{RhiRom-04}.
Finally, the use of coda wave interferometry, i.e., the analysis of the cross correlations of the tails of seismographs generated by earthquakes and that correspond to multiply scattered waves, has recently opened new ways to extract the dispersion curves~\cite{deHGarSol-22}.

Most inversion methods assume high-frequency asymptotics~\cite{deHIanvdHZha-20a} while the recent applications using ambient noise provide rich low-frequency information. That is why we would like to investigate the inverse problem associated with surface waves in a general framework.
In this work, we analyze the inverse problem associated with Love waves for a time-independent, isotropic, stratified half-space, homogeneous in the $(x,y)$-plane. We show how to recover the parameters of the elastic medium from the empirical knowledge of the dispersion relation. That is, from the empirical knowledge of the frequency--wavenumber couples of the Love waves.

The discontinuity that we will assume on the media and our goal to obtain results for all frequencies yield us to use tools from complex analysis and from analytic perturbation theory~\cite{Kato}. Moreover, the discontinuity assumption makes standard formulae for Weyl's law unavailable to us: we establish them by direct computations and a careful analysis.

We consider the space $\R^2\times[0,+\infty)$ and assume that the relevant quantities are constant on layers of the form $\R^2\times[H_j,H_{j+1})$. More precisely, we consider a medium composed of $n+1$~layers, $n\geq1$, such that the shear modulus $\mu>0$ and the density $\rho>0$ of the medium are constant inside each layer\footnote{with the convention, for $n=1$, that $\llbracket 2, 1 \rrbracket=\emptyset$.}:
\begin{equation}
    (\mu(z), \rho(z)) =
    \left\{
        \begin{aligned}
             &(\mu_1, \rho_1)\,, \quad &&\textrm{if }\!\!\!\! &0 &\leq z < H_2\,, \\
            &(\mu_j, \rho_j)\,, \quad &&\textrm{if }\!\!\!\! &H_j &\leq z < H_{j+1}\,, \quad \forall\, j \in \llbracket 2, n \rrbracket , \\
            &(\mu_{n+1}, \rho_{n+1})\,, \quad &&\textrm{if }\!\!\!\! &H_{n+1} &\leq z < +\infty\,,
        \end{aligned}
    \right.
\end{equation}
where we recall that $\llbracket p, n \rrbracket = [p, n]\cap\Z$.
Or, more concisely with $H_1:=0$ and $H_{n+2}:=+\infty$,
\begin{equation}
    \forall\, j \in \llbracket 1, n+1 \rrbracket\,, \quad (\mu(z), \rho(z)) = (\mu_j, \rho_j) \qquad \textrm{on }\, [H_j, H_{j+1})\,.
\end{equation}

Within this setup, we are interested in Love waves. That is, in frequency--wavenumber couples $(\omega,k)$ for which there exists $L^2((0,+\infty))$-solutions $\phi$ to the boundary value problem
\begin{equation}
\label{Problem_equations_intro}
    \left\{
    \begin{aligned}
        &-\left( \mu \phi' \right)'\!(z) + \left(\mu(z) k^2 - \rho(z) \omega^2 \right) \phi(z) = 0\,, \quad \text{on } [0,+\infty)\,, \\
        &\phi'(0) = 0\,,
    \end{aligned}
    \right.
\end{equation}
with continuity conditions resulting from the continuity of the displacement and of the shear and normal stress components:
$\phi \in \mathcal{C}([0,+\infty))$ and $\mu \phi' \in \mathcal{C}([0,+\infty))$.
Without loss of generality we restrict ourselves to $\phi$ real-valued: $\phi\in L^2((0,+\infty);\R)$.
See Appendix~\ref{Derivation_Love_waves_problem} for the derivation of this problem, as well as the continuity conditions, from the laws of physics.

Since $\mu$ and $\rho$ are positive, we define
\begin{equation}
        C(z) := \sqrt{\mu(z) / \rho(z)}>0 \quad \text{on } [0,+\infty) \quad \text{and} \quad C_j := \sqrt{\mu_j / \rho_j} \quad \text{for } j \in \llbracket 1, n+1 \rrbracket\,.
\end{equation}
We furthermore define
\begin{equation}
    C_\infty := C_{n+1} = \lim_{+\infty} C = C(z)\,, \, \, \forall\, z \geq H_{n+1}\,, \quad \text{ and } \quad C_0 := \min\limits_{[0,+\infty)} C\,.
\end{equation}

We emphasize that we do not assume a priori that $C$ is non-decreasing. The only assumption made on the values of $C$ is that $C_0 < C_\infty$ as, otherwise, there cannot be Love waves (see Lemma~\ref{Alt_Def_kell_nPlus1_layers}).

On each layer (indexed by $j$), $\mu=\mu_j$ and $\rho=\rho_j$ being positive constants, the eigenvalue equation becomes $\phi'' = ( k^2 - \omega^2 / C_j^2 ) \phi$. Consequently, for $j\in\llbracket 0, n+1 \rrbracket\cup\{\infty\}$, we define
\begin{equation}\label{Def_nu}
    \nu_j \equiv \nu_j(\omega, k) := C_j^{-1} \sqrt{C_j^2 k^2 - \omega^2} = \omega \sqrt{k^2 \omega^{-2} - C_j^{-2}} \quad \textrm{ with } \Im \nu_j\leq 0\,.
\end{equation}
On each layer the solutions are either of the form
\begin{equation}
    A_{j,+} e^{+\nu_j z} + A_{j,-} e^{-\nu_j z}
\end{equation}
or affine. The requirement of the solution being $L^2$ imposes that on the last layer (which has parameters $\nu_\infty=\nu_{n+1}$, $A_{\infty,+}=A_{n+1,+}$, and $A_{\infty,-}=A_{n+1,-}$) the solution is of the former form with $\nu_\infty>0$ and $A_{\infty,+}=0$. This means that for a Love wave $\phi$ to exist at~$(\omega,k)$, it must verify that $k$ is bounded away from zero by $k > \omega / C_\infty \geq0$ and that $\phi$ must vanish (exponentially) at infinity.

Finally, we define, for each layer $j\in\llbracket 1, n+1 \rrbracket$, its thickness
\begin{equation}
    T_j := H_{j+1} - H_j \in (0,+\infty]\,,
\end{equation}
as well as, for $j\in \llbracket 0, n+1 \rrbracket\cup\{+\infty\}$, the parameters independent of $\omega$
\begin{equation}\label{Def_nu_bar}
    \bar{\nu}_j \equiv \bar{\nu}_j(y) := \frac{\nu_j(\omega, \omega y)}{\omega} = \sqrt{y^2 - C_j^{-2}} \quad \textrm{ with } \Im \bar{\nu}_j\leq 0\,.
\end{equation}

\bigskip

\textbf{Main results.}
The boundary condition at $z=0$, the $L^2$-restriction, and the continuity conditions determine, for each $\omega>0$, the finite set of values of $k$ for which a Love wave exists at the parameters $(\omega, k)$. Summarizing the above, we consider the problem of finding $0\not\equiv\phi\equiv\phi_{\omega,k}\in L^2((0,+\infty))$ such that
\begin{equation}\label{Problem_equations_new}
    \left\{
    \begin{aligned}
        &-\frac{\di}{\di z}\left( \mu \frac{\di}{\di z} \phi \right) = \mu \omega^2 \left( 1/C^2 - k^2/\omega^2 \right) \phi\,, \quad \text{on } [0,+\infty)\,, \qquad k/\omega > 1/C_\infty\,,\\
        &\phi \in \mathcal{C}([0,+\infty)) \textrm{ with } \lim\limits_{+\infty} \phi = 0\,, \quad \textrm{ and } \quad \mu \phi' \in \mathcal{C}([0,+\infty)) \textrm{ with } \phi'(0) = 0\,.
    \end{aligned}
    \right.
\end{equation}
In the rest of the paper, we will say that ``a Love wave \emph{exists at}~$(\omega, k)$'', whenever there exists an $L^2$-solution~$\phi_{\omega,k}\not\equiv0$ to~\eqref{Problem_equations_new} for the couple~$(\omega, k)$.

The goal of this paper is to recover the profiles of the shear modulus $\mu>0$ and the density $\rho>0$ of the medium, or at least their ratio, as well as the values $H_j$'s, from the experimental knowledge of the couples $(\omega, k)$ at which a Love wave exists.

Looking at these $k$'s as functions of $\omega$, we will show that they form branches $\omega\mapsto k_\ell(\omega)$, $\ell\geq1$, and our first main result is the following.
\begin{theorem}[Regularity and monotonicity of the branches $k_\ell$]\label{Thm_monotonicity_nPlus1_layers_intro}
    Let $n\geq1$. For any $\ell\geq1$, there exists $\omega_\ell\geq0$ such that the function
    \begin{align*}
        (\omega_\ell, +\infty) &\to (1/C_\infty, 1/C_0) \\
        \omega &\mapsto k_\ell(\omega)/\omega
    \end{align*}
     is analytic, bijective, and increasing.
\end{theorem}
The precise definition of $k_\ell$'s will be given later.
Graphically, this can be seen in the numerical simulations in Figure~\ref{Fig_simu_1_6Plus1_layers_intro}, where each colored curve corresponds to one $\ell$ and the $\omega_\ell$'s are the values of $\omega$ at which the curve ``starts'' (with value $1/C_\infty$).

The rest of our main results are concerned with recovering the parameters of the medium. We first have the following immediate consequence of Theorem~\ref{Thm_monotonicity_nPlus1_layers_intro}.
\begin{corollary}[Recovering $C_0$ and $C_\infty$]\label{Cor_recover_C0_Cinfty_nPlus1_layers}
    Let $n\geq1$. With the notations of Theorem~\ref{Thm_monotonicity_nPlus1_layers_intro}, for all $\ell\geq 1$, we have
    \begin{equation}
        \frac{1}{C_0} = \sup\limits_{\omega>\omega_\ell} \frac{k_\ell(\omega)}{\omega} = \lim\limits_{\omega\to+\infty} \frac{k_\ell(\omega)}{\omega} \quad \textrm{ and } \quad \frac{1}{C_\infty} = \inf\limits_{\omega>\omega_\ell} \frac{k_\ell(\omega)}{\omega} = \lim\limits_{\omega\to\omega_\ell} \frac{k_\ell(\omega)}{\omega}\,.
    \end{equation} 
\end{corollary}

Our second main result concerns Weyl's law and is a complete result for $n=1,2$ but a partial one for $n\geq3$, in which case we conjecture the complete result based on a formal application of Weyl's law.
These results are concerned with the asymptotics, for any $\omega>0$ and $y\in(\omega/C_\infty, \omega/C_0)$, of the number $N(\omega, y)$ of branches $k_\ell(\omega)/\omega$ that are above or equal to $y$.
\begin{definition}\label{Def_number_branches_above_nPlus1_layers}
    Let $n\geq1$. Let $\omega>0$ and $y\in(1/C_\infty, 1/C_0)$. Define   
    \begin{equation} 
        N(\omega, y) := \# \left\{ \ell\geq1 \,:\, \frac{k_\ell(\omega)}{\omega} \geq y \right\} = \max \left\{ \ell\geq1 \,:\, \frac{k_\ell(\omega)}{\omega} \geq y > \frac{k_{\ell+1}(\omega)}{\omega} \right\}.
    \end{equation}
\end{definition}
In this definition, we take the convention $k_\ell(\omega)=-\infty$ if $k_\ell$ is undefined at $\omega$. Note that, due to the monotonicity of the $\frac{k_\ell(\omega)}{\omega}$'s (Theorem~\ref{Thm_monotonicity_nPlus1_layers_intro}), $\omega\mapsto N(\omega, y)$ is nondecreasing for any fixed $y\in(1/C_\infty, 1/C_0)$.

In order to state concisely our result and conjecture, we reorder the $C_j$'s as well as the associated parameters.
\begin{definition}\label{Def_ordered_Cs_nPlus1_layers}
    Let $n\geq1$. Define $\{\widetilde{C}_j\}_{1\leq j\leq n+1}$ as the nondecreasing reordering of the sequence $\{C_j\}_{1\leq j\leq n+1}$.
\end{definition}    
There exists a permutation $\sigma$ of $\llbracket 1, n+1 \rrbracket$ s.t.\ $\widetilde{C}_j = C_{\sigma(j)}$, and we define the sequences $\{\tilde{\nu}_j\}_{1\leq j\leq n+1}$ and $\{\widetilde{T}_j\}_{1\leq j\leq n+1}$ by $\tilde{\nu}_j = \bar{\nu}_{\sigma(j)}$ and $\widetilde{T}_j = T_{\sigma(j)}$ for $j\in\llbracket 1, n+1 \rrbracket$.

Notice that $C_{n+1} = C_\infty > \widetilde{C}_1=C_0$ for $n\geq1$ and $\widetilde{C}_1 = C_1$ for $n=1$.

We are now able to state our main result on Weyl's law in our setting (for which we recall that standard formulae are not available since $\mu$ and $\rho$ are discontinuous).
\begin{proposition}\label{Prop_asymptotics_N}
    Let $n=1$. Then, for $y\in[1/C_\infty, 1/C_0)$, as $\omega$ goes to $+\infty$, we have
    \begin{equation}
        N(\omega, y) \sim \frac{\omega}{\pi} |\tilde{\nu}_1(y)| \widetilde{T}_1\,.
    \end{equation}   
    Let $n=2$. Then, for $y\in[1/C_\infty, 1/C_0)$, as $\omega$ goes to $+\infty$, we have
    \begin{equation}
        \left\{
        \begin{aligned}
            &N(\omega, y) \sim \frac{\omega}{\pi} |\tilde{\nu}_1(y)| \widetilde{T}_1\,, \qquad &\text{if } y\in[1/\widetilde{C}_2, 1/C_0)\,, \\
            &N(\omega, y) \sim \frac{\omega}{\pi} \left( |\tilde{\nu}_1(y)| \widetilde{T}_1 + |\tilde{\nu}_2(y)| \widetilde{T}_2 \right), \qquad &\text{if } y\in[1/C_\infty, 1/\widetilde{C}_2)\,.
        \end{aligned}
        \right.
    \end{equation}
    Let $n\geq3$ and assume $C_0 < \widetilde{C}_2$. Then, for $y\in[1/\widetilde{C}_2, 1/C_0)$, as $\omega$ goes to $+\infty$, we have
    \begin{equation}
        N(\omega, y) \sim \frac{\omega}{\pi} |\tilde{\nu}_1(y)| \widetilde{T}_1\,.
    \end{equation}
\end{proposition}
In the case $n\geq3$, we conjecture the following natural extension to the whole interval $[1/C_\infty, 1/C_0)$ and without the assumption $C_0 < \widetilde{C}_2$ (that is, allowing several of the $C_j$'s to be equal to $C_0:=\min_j C_j$).
\begin{conjecture}
    Let $n\geq3$ and $y\in[1/C_\infty, 1/C_0)$. Then, as $\omega$ goes to $+\infty$ we have
    \begin{equation}
        N(\omega, y) \sim \frac{\omega}{\pi} \sum_{p=1}^{j} |\tilde{\nu}_p(y)| \widetilde{T}_p\,, \qquad \text{if } y\in[1/\widetilde{C}_{j+1}, 1/\widetilde{C}_j)\,.
    \end{equation}
\end{conjecture}

Under the assumption that $C_{n+1}$ is the largest value taken by the function~$C$, i.e., $C_j < C_\infty = C_{n+1}$ for all $\llbracket 1, n \rrbracket$, and that the~$C_j$'s are pairwise distinct, these asymptotics allow to fully determine $C$ as well as the~$T_j$'s (hence the~$H_j$'s). If we only assume that the~$C_j$'s are pairwise distinct, then all the values $C_j<C_\infty$ can be recovered, as well as the associated~$T_j$'s. Finally, if the ``pairwise distinct'' assumption is lifted, one can still recover the values~$C_j < C_\infty$ but only the sums of the thicknesses of the layers sharing the value~$C_j$. 

Indeed, the values $C_j$ can be extracted from empirical data (dispersion curves of surface Love waves can be obtained from earthquakes signals or ambient noise signals as discussed in the introduction). They are the horizontal lines where the ``density'' of branches of frequency--wavenumber couples of the Love waves diverges as the frequency goes to infinity ---see Appendix~\ref{Appendix_simulations} for simulated versions of such data. Then, evaluating the number $N$ at each $1/C_j$ when the frequency diverges yields the values of the $T_j$'s (hence of the $H_j$'s).
Finally, for $n=1$ and assuming that $\rho_1$ is known (hence $\mu_1=\rho_1 C_1^2$ too), we additionally determine $\mu_2$ and $\rho_2$.
In practice, least-squares or Bayesian inversion can be applied to noisy or perturbed dispersion curves to estimate the medium parameters in a robust way and to quantify the uncertainty of the estimation~\cite{MosTar-95,ShaRit-02,Tarantola-05}. Our work gives solid foundations to this approach by proving the existence and uniqueness of the least-squares minimum or Bayesian maximum a posteriori.

As a perspective of this work, the authors hope to address similar questions for Rayleigh waves. This could allow to recover all the Lam{\'e} and density parameters of the elastic medium.

\bigskip

\textbf{Organisation of the paper.}
We derive in Section~\ref{Section_derivation_k_ell} the dispersion relation defining (up to a constraint) the existence of Love waves. In Section~\ref{Section_monotonicity_existence_LoveWaves}, we prove Theorem~\ref{Thm_monotonicity_nPlus1_layers}, a detailed version of Theorem~\ref{Thm_monotonicity_nPlus1_layers_intro}. Doing so, we also prove in this section the case $n\geq3$ of Proposition~\ref{Prop_asymptotics_N}, see Lemma~\ref{Lemma_zeros_ftilde_any_y_nPlus1_layers}.

Because we are able to obtain stronger results for the cases of a simple ($n=1$) and of a double ($n=2$) square well, we then focus on these cases. Namely, in Section~\ref{Section_1plus1_layers} we study the simple square well for which all computations can be done explicitly. A direct proof (by implicit function theorem) of smoothness is given, during which we also obtain the explicit formulae of the $\omega_\ell$'s (see Proposition~\ref{Proposition_1plus1layer}), and we show that all the parameters of the medium can be recovered. Moreover, the proof of Weyl's law (Proposition~\ref{Prop_asymptotics_N}) in this case is completed at the end of this section (see Subsection~\ref{Section_asymptotics_N_for_n_equal_1}).
In Section~\ref{Section_2Plus1_layers}, we focus on the case of a double square well for which we prove a stronger version of Proposition~\ref{Prop_asymptotics_N} (see Propositions~\ref{Prop_zeros_ftilde_any_y_2Plus1_layers}--\ref{Prop_asymptotics_N_for_n_equal_2}).

Appendix~\ref{Appendix_simulations} presents numerical simulations. In Appendix~\ref{Derivation_Love_waves_problem} we derive the linear, elastic equation, the continuity conditions, and the boundary condition in~\eqref{Problem_equations_new}. Appendix~\ref{Appendix_nplus1_layers} gives two proofs that we postponed for the readability of the paper. Finally, Appendix~\ref{Appendix_1plus1_layers} presents two additional results for the simple square well.

\section{Characterization of Love waves: dispersion relation and first results}\label{Section_derivation_k_ell}
In this section, we characterize Love waves for the settings that we are considering. This characterization relies on the dispersion relation, which in our context was established in the literature as early as the celebrated work~\cite{Haskell-53} by Haskell, based on Thomson's work~\cite{Thomson-50} describing for the first time the transfer matrix method. Even though this relation is well-known, we detail here its derivation for several reasons. First, for the convenience of the reader and because Haskell's paper~\cite{Haskell-53} being focused on Rayleigh waves (like the one by Thomson), it gives little details on the computations in the case of Love waves. Second, because our derivation is slightly different: it is not per se based on the transfer matrix method even though these matrices appear in our work up to a simple transformation. Third, and more importantly for our results, because our derivation gives as a direct by-product the simplicity of the $k_\ell$’s, see Corollary~\ref{simplicity_k}, which is a key property in some of our later proofs.

\smallskip

Using the form of the solutions on each layer (see the introduction), the boundary condition at $z=0$, the $L^2$-restriction, and the fact that $\nu_{n+1}>0$, we obtain the form of a solution $\phi$:
\[
    \phi(z) =
    \left\{
        \begin{aligned}
            &2 \alpha_1(\omega) \cosh[\nu_1(\omega) z]\,,\\
            &\alpha_j(\omega) e^{-\nu_j(\omega) z} + \beta_j(\omega) e^{+\nu_j(\omega) z}\,,\\
            &\alpha_{n+1}(\omega) e^{-\nu_{n+1}(\omega) z}\,,
        \end{aligned}
        \qquad
        \begin{aligned}
            \textrm{if }\!\!\!\! &&0 &\leq z < H_2\,, \\
            \textrm{if }\!\!\!\! &&H_j &\leq z < H_{j+1}\,, \quad \forall\, j \in \llbracket 2, n \rrbracket\,, \\
            \textrm{if }\!\!\!\! &&H_{n+1} &\leq z < +\infty\,.
        \end{aligned}
    \right.
\]

For this introductory presentation, we assume $k(\omega) / \omega \neq C_j^{-1}$ for all $j$, i.e., $\nu_j(\omega, k)\neq0$, but the remaining cases are treated in Proposition \ref{Prop_recursive_formula_Dn} below.
The frequency--wavenumber couples of the Love waves are the
pairs $(\omega,k)$ for which non-trivial solutions $\phi$ exist.

Omitting the dependencies in $\omega$ for shortness, the $2n$ continuity conditions at the boundaries $\{H_j\}_{2\leq j\leq n+1}$ yield
\[
    \left\{
        \begin{aligned}
             2 \alpha_1 \cosh[ \nu_1 T_1 ] &= \beta_2 e^{+\nu_2 H_2} + \alpha_2 e^{-\nu_2 H_2}\,,\\
             2 \mu_1 \nu_1 \alpha_1 \sinh[ \nu_1 T_1 ] &= \mu_2 \nu_2 \left( \beta_2 e^{+\nu_2 H_2} - \alpha_2 e^{-\nu_2 H_2} \right),\\
             \beta_{j-1} e^{+\nu_{j-1} H_j} + \alpha_{j-1} e^{-\nu_{j-1} H_j} &= \beta_j e^{+\nu_j H_j} + \alpha_j e^{-\nu_j H_j}\,, &\forall\, j \in \llbracket 3, n \rrbracket\,,\\
             \mu_{j-1} \nu_{j-1} \left( \beta_{j-1} e^{+\nu_{j-1} H_j} - \alpha_{j-1} e^{-\nu_{j-1} H_j} \right) &= \mu_j \nu_j \left( \beta_j e^{+\nu_j H_j} - \alpha_j e^{-\nu_j H_j} \right), &\forall\, j \in \llbracket 3, n \rrbracket\,,\\
             \beta_n e^{+\nu_n H_{n+1}} + \alpha_n e^{-\nu_n H_{n+1}} &= \alpha_{n+1} e^{-\nu_{n+1} H_{n+1}}\,,\\
             \mu_n \nu_n \left( \beta_n e^{+\nu_n H_{n+1}} - \alpha_n e^{-\nu_n H_{n+1}} \right) &= - \mu_{n+1} \nu_{n+1} \alpha_{n+1} e^{-\nu_{n+1} H_{n+1}}\,.
        \end{aligned}
    \right.
\]
Denoting $A_j := \mu_j \nu_j$, non-trivial solutions $\phi$ exist if and only if this linear system  has non-zero solutions, which happens if and only if the determinant
\[
    \makebox[\textwidth]{
    \resizebox{0.98\paperwidth}{!}{
        $\begin{vmatrix}
            2 \cosh[ \nu_1 T_1 ] & - e^{-\nu_2 H_2} & -e^{+\nu_2 H_2} & 0 & 0 & \cdots & \cdots & 0 & 0 & 0 \\
            2 A_1 \sinh[ \nu_1 T_1 ] & +A_2 e^{-\nu_2 H_2} & -A_2 e^{+\nu_2 H_2} & 0 & 0 & \cdots & \cdots & 0 & 0 & 0 \\
            0 & +e^{-\nu_2 H_3} & +e^{+\nu_2 H_3} & -e^{-\nu_3 H_3} & -e^{+\nu_3 H_3} & 0 & \cdots & 0 & 0 & 0 \\
            0 & -A_2 e^{-\nu_2 H_3} & +A_2 e^{+\nu_2 H_3} & +A_3 e^{-\nu_3 H_3} & -A_3 e^{+\nu_3 H_3} & 0 & \cdots & 0 & 0 & 0 \\
            0 & 0 & 0 & \hphantom{\ddots\ddots}\ddots & \hphantom{\ddots\ddots}\ddots & \ddots\hphantom{\ddots\ddots\ddots} &  & \vdots & \vdots & \vdots \\
            \vdots & \vdots & \vdots &  & \hphantom{\ddots\ddots}\ddots & \ddots\hphantom{\ddots\ddots\ddots} & \ddots\hphantom{\ddots\ddots\ddots\ddots} & 0 & 0 & 0 \\
            0 & 0 & 0 & \cdots & 0 & +e^{-\nu_{n-1} H_n} & +e^{+\nu_{n-1} H_n} & -e^{-\nu_n H_n} & -e^{+\nu_n H_n} & 0 \\
            0 & 0 & 0 & \cdots & 0 & -A_{n-1} e^{-\nu_{n-1} H_n} & +A_{n-1} e^{+\nu_{n-1} H_n} & +A_n e^{-\nu_n H_n} & -A_n e^{+\nu_n H_n} & 0 \\
            0 & 0 & 0 & \cdots & \cdots & 0 & 0 & +e^{-\nu_n H_{n+1}} & +e^{+\nu_n H_{n+1}} & -e^{-\nu_{n+1} H_{n+1}} \\
            0 & 0 & 0 & \cdots & \cdots & 0 & 0 & -A_n e^{-\nu_n H_{n+1}} & +A_n e^{+\nu_n H_{n+1}} & +A_{n+1} e^{-\nu_{n+1} H_{n+1}}
        \end{vmatrix},$
    }
    }
\]
denoted by $D_n$, is zero. Note that this determinant appears for instance in~\cite[(7)--(8)]{Knopoff-64}, even though under a slightly different form. For clarity, we can write it as
\begin{equation}\label{Formula_D_n}
    D_n = \det \mathbb{M}_n \qquad \text{with} \qquad \mathbb{M}_n :=
    \begin{pNiceMatrix}[margin=5pt]
        L_1^r & R_2 & \mathbb{O}_2 & \mathbb{O}_2 & \mathbb{O}_2 & 0 \\
        0 & L_2 & R_3 & \Ddots & \mathbb{O}_2 & \Vdots \\
        & \mathbb{O}_2 & \Ddots & \Ddots & \mathbb{O}_2 & \\
        \Vdots & \mathbb{O}_2 & \Ddots & L_{n-1} & R_n & 0 \\
        0 & \mathbb{O}_2 & \mathbb{O}_2 & \mathbb{O}_2 & L_n & R_{n+1}^l
    \end{pNiceMatrix},
\end{equation}
where
\begin{align*}
    &\hphantom{
        \left\{
        \vphantom{
            \begin{aligned}
                \begin{pmatrix}
                    +e^{-\nu_j H_{j+1}} & +e^{+\nu_j H_{j+1}} \\
                    -\mu_j \nu_j e^{-\nu_j H_{j+1}} & +\mu_j \nu_j e^{+\nu_j H_{j+1}}
                \end{pmatrix}\\
                \begin{pmatrix}
                    1 \\
                    0
                \end{pmatrix}
            \end{aligned}
        }
        \right.
    }
    \mathbb{O}_2:=
    \begin{pmatrix}
        0 & 0 \\
        0 & 0
    \end{pmatrix},\quad
    L_1^r:= 2 
    \begin{pmatrix}
        \cosh[ \nu_1 T_1 ] \\
        \mu_1 \nu_1 \sinh[ \nu_1 T_1 ]
    \end{pmatrix},
    \intertext{and}
    \forall\, j\geq2,\,
    &\left\{
    \begin{aligned}
        L_j&:=
        \begin{pmatrix}
            +e^{-\nu_j H_{j+1}} & +e^{+\nu_j H_{j+1}} \\
            -\mu_j \nu_j e^{-\nu_j H_{j+1}} & +\mu_j \nu_j e^{+\nu_j H_{j+1}}
        \end{pmatrix}, \\
        R_j&:=
        \begin{pmatrix}
            -e^{-\nu_j H_j} & -e^{+\nu_j H_j} \\
            +\mu_j \nu_j e^{-\nu_j H_j} & -\mu_j \nu_j e^{+\nu_j H_j}
        \end{pmatrix},\\
        L_j^l&:= L_j 
        \begin{pmatrix}
            1 \\
            0
        \end{pmatrix},\quad
        L_j^r:= L_j 
        \begin{pmatrix}
            0 \\
            1
        \end{pmatrix},\quad
        R_j^l:= R_j 
        \begin{pmatrix}
            1 \\
            0
        \end{pmatrix},
        \quad\textrm{ and }\quad
        R_j^r:= R_j 
        \begin{pmatrix}
            0 \\
            1
        \end{pmatrix}.
    \end{aligned}
    \right.
\end{align*}

The first important remark is that the submatrix $\widetilde{\mathbb{M}}_n$ of $\mathbb{M}_n$, where we remove the first row and the last column, is a block (upper) triangular matrix with the blocks on the diagonal being $2 \mu_1 \nu_1 \sinh[ \nu_1 T_1 ]$ and the $L_j$'s. Since $\det L_j = 2 \mu_j \nu_j$, for $j\geq2$, we have
\begin{equation}
    \det \widetilde{\mathbb{M}}_n = 2^n \sinh[ \nu_1 T_1 ] \prod_{j=1}^n \mu_j \nu_j\,.
\end{equation}
Therefore, if the $\nu_j$'s are non-zero, then $\det \widetilde{\mathbb{M}}_n\neq0$ and $\rank \mathbb{M}_n \geq 2n-1$. Actually, we prove this to also hold when some $\nu_j$'s are zero and we show in the following proposition that $D_n$ can be computed recursively.
\begin{proposition}\label{Prop_recursive_formula_Dn}
    Let $n\in\N\setminus\{0\}$ and $D_n=\det \mathbb{M}_n$ be defined in~\eqref{Formula_D_n}. Then,
    \[
        \rank \mathbb{M}_n \geq 2n-1
    \]
    and, if $\nu_{j_i}(\omega)=0$ for $i\in\llbracket1, m\rrbracket$, $\nu_{k_i}(\omega)\neq0$ for $i\in\llbracket1,n-m\rrbracket$, and $\nu_{n+1}(\omega)>0$, then
    \begin{equation}\label{Formula_det_any_n}
        \frac{ (-1)^m e^{\nu_{n+1} H_{n+1}} }{ 2^{n-m} \prod\limits_{i=1}^{n-m} \mu_{k_i} \nu_{k_i} } \frac{D_n}{\prod\limits_{i=1}^m \mu_{j_i}} = \mu_{n+1} \nu_{n+1} P_n + Q_n \in\R\,,
    \end{equation}
    where the $P_n$'s and $Q_n$'s are defined recursively by $P_0=1$, $Q_0=0$, and 
    \begin{equation}\label{Formula_Pn_Qn}
        \begin{pmatrix}P_m \\ Q_m\end{pmatrix} =
        M_m
        \begin{pmatrix}P_{m-1} \\ Q_{m-1}\end{pmatrix} \quad \text{for all } m\in\llbracket1,n\rrbracket\,,
    \end{equation}
    where
    \begin{equation}\label{Def_Mm}
        M_m:=
        \left\{
        \begin{aligned}
            &\begin{pmatrix}
                \cosh[\nu_m T_m] & \sinh[\nu_m T_m] / (\mu_m \nu_m) \\
                \mu_m \nu_m \sinh[\nu_m T_m] & \cosh[\nu_m T_m]
            \end{pmatrix} &\text{if } \nu_m\neq0\,, \\
            &\begin{pmatrix}
                1 & T_m/\mu_m \\
                0 & 1
            \end{pmatrix} &\text{if } \nu_m=0\,.
        \end{aligned}
        \right.
    \end{equation}
\end{proposition}
Here, we used the convention $\prod_{j=k_1}^{k_2} a_j = 1$ if $k_1>k_2$.

This proposition leads us to define $f_n:(0,+\infty)\times\R\to\C$ by
\begin{equation}\label{Def_f_n}
    f_n := \mu_{n+1} \nu_{n+1} P_n + Q_n = \mu_\infty \nu_\infty P_n + Q_n\,,
\end{equation}
where we recall that $\nu_\infty=\nu_{n+1}$, since $C_\infty:=C_{n+1}$, and where we define $\mu_\infty:=\mu_{n+1}$, so that the dispersion relation for Love waves reads
\begin{equation}\label{Dispersion_relation_n}
    f_n(\omega, k) = 0\,.
\end{equation}

As explained, a Love wave existing at $(\omega,k)$ is equivalent to $D_n=0$ for this pair ---i.e., $(\omega,k)$ solves the dispersion relation $f_n(\omega, k) = 0$--- under the constraint $k > \omega/C_\infty$:
\begin{equation}\label{nPlus1_layers_relation_defining_Love_waves_preversion}
    f_n(\omega, k) = 0 \quad \text{ and } \quad k > \omega / C_\infty\,.
\end{equation}

\begin{remark*}
	Our strategy to derive the dispersion relation is different but somewhat related to the transfer matrix method, also known as propagator matrix method, which is well-known in geophysics~\cite{Thomson-50,Haskell-53,GilBac-66,Kennett-83,BucBen-96,Kennett-83-Edition2009}.
	Our matrices $M_m$ are, indeed, closely related to the transfer matrices, derived by Haskell~\cite{Haskell-53} in our context:
	\[
		a_m :=
		\begin{pmatrix}
			\cos[k r_{\beta_m} T_m] & i\frac{\sin[k r_{\beta_m} T_m]}{\mu_m r_{\beta_m}} \\
			i \mu_m r_{\beta_m} \sin[k r_{\beta_m} T_m] & \cos[k r_{\beta_m} T_m]
		\end{pmatrix}
	\]
	(we follow Haskell and Thomson notations ``$a_m$'' for the transfer matrices, which are nowadays often denoted $T_m$ in the literature). Indeed, our matrices $M_m$ defined by~\eqref{Def_Mm} in Proposition~\ref{Prop_recursive_formula_Dn} are, up to a simple transformation, exactly the transfer matrices $a_m$:
	\begin{equation}\label{relation_M_with_transfer_matrix}
		M_m = \begin{pmatrix} 1 & 0 \\ 0 & i k \end{pmatrix} a_m \begin{pmatrix} 1 & 0 \\ 0 & (i k)^{-1} \end{pmatrix}.
	\end{equation}
	
	Moreover, and of course, our dispersion relation is equivalent to that obtained by the transfer matrix method. Indeed, using that the $r_{\beta_m}$'s in Haskell's paper are related to our $\nu_j$'s by the relation $k r_{\beta_m} = i \nu_m$ when $\nu_m\in i\R_-$ and $k r_{\beta_m} = -i \nu_m$ when $\nu_m>0$, we~have
	\[
		a_m =
		\begin{pmatrix}
			\cosh[\nu_m T_m] & \frac{i k}{\mu_m \nu_m} \sinh[\nu_m T_m] \\
			\frac{\mu_m \nu_m}{i k} \sinh[\nu_m T_m] & \cosh[\nu_m T_m]
		\end{pmatrix}
	\]
	and consequently the relation $A_{21} = -\mu_n r_{\beta_n} A_{11}$ obtained by Haskell for Love waves through the transfer matrix method ---equation (9.9) in Haskell's paper---, can be written
	\[
		\mu_n \nu_n A_{11} \mp i k A_{21} = 0\,, \qquad \text{with ``+'' when } \nu_n>0\,.
	\]
	Remembering that for a Love wave to exist at $(\omega, k)$, the $\nu$ associated to the semi-infinite layer ($\nu_{n+1}$ in our paper but $\nu_n$ in Haskell's) must necessarily be positive, the identity determined by Haskell is therefore
	\[
		\mu_n \nu_n A_{11} + i k A_{21} = 0\,.
	\]
	Noticing now that Haskell labeled the layers from $1$ to $n$ while we labeled them from $1$ to $n+1$, this relation is the same identity as our dispersion relation~\eqref{Dispersion_relation_n}:
	\[
		\mu_{n+1} \nu_{n+1} P_n + Q_n = 0\,.
	\]
	Indeed, defining the $2\times2$ matrix $M := \prod_{m=1}^{n} M_m$ and using the relation~\eqref{relation_M_with_transfer_matrix} between $M_m$ and $a_m$, we have
	\[
		M = \begin{pmatrix} 1 & 0 \\ 0 & i k \end{pmatrix} A \begin{pmatrix} 1 & 0 \\ 0 & (i k)^{-1} \end{pmatrix} = \begin{pmatrix} A_{11} & (i k)^{-1} A_{12} \\ i k A_{21} & A_{22} \end{pmatrix},
	\]
	thence $P_n = M_{11} = A_{11}$ and $Q_n = M_{21} = i k A_{21}$ by~\eqref{Formula_Pn_Qn}.
\end{remark*}

Before turning to the proof of Proposition~\ref{Prop_recursive_formula_Dn}, let us continue with the definition of the $k_\ell(\omega)$'s appearing in Theorem~\ref{Thm_monotonicity_nPlus1_layers_intro}.
\begin{definition}[Definition of the $k_\ell$'s]\label{Def_kell_nPlus1_layers}
    Let $n\in\N\setminus\{0\}$ and $\omega>0$. The $k_\ell(\omega)$'s are the (decreasingly ordered) values $k\in\R$ for which $(\omega, k)$ solves the dispersion relation~\eqref{Dispersion_relation_n}.
\end{definition}
Notice that in Definition~\ref{Def_kell_nPlus1_layers}, we did not put a priori restrictions on $k\in\R$. This is because we actually have the folllowing.
\begin{lemma}\label{Alt_Def_kell_nPlus1_layers}
    Definition~\ref{Def_kell_nPlus1_layers} is equivalent to defining the $k_\ell(\omega)$'s as the (decreasingly ordered) values $k\in[\omega / C_\infty, \omega / C_0)$ for which $(\omega, k)$ solves the dispersion relation~\eqref{Dispersion_relation_n}.
\end{lemma}
Note that $f_n$ is real valued on $(0,+\infty)\times[\omega / C_\infty, \omega / C_0)$.
\begin{proof}
    On the one hand, $(\omega, k)$ being a solution to~\eqref{Dispersion_relation_n} implies $\nu_0(\omega, k)\in i\R_-\setminus\{0\}$. Indeed, we would otherwise have $\nu_i(\omega, k)$ for $j\in\llbracket1,n\rrbracket$ and $\nu_{n+1}(\omega, k)>0$. We claim that it implies $P_m\geq1$ and $Q_m\geq0$ for any $m\in\llbracket1,n\rrbracket$. This is because the diagonal coefficients of $M_m$, in~\eqref{Formula_Pn_Qn}, are then greater or equal to $1$ while the antidiagonal ones are nonnegative. Hence, since $P_0=1$ and $Q_0=0$, a straightforward induction gives the claim. We therefore obtain the contradiction, to $(\omega, k)$ being a zero, that
    \[
        f_n(\omega, k) = \mu_{n+1} \nu_{n+1}(\omega, k) P_n(\omega, k) + Q_n(\omega, k) \geq \mu_{n+1} \nu_{n+1}(\omega, k) > 0\,.
    \]

    On the another hand, $(\omega, k)$ being a solution to~\eqref{Dispersion_relation_n} implies $\nu_{n+1}(\omega, k)\geq0$. Indeed, we would otherwise have $\nu_j(\omega, k)\in i\R_-\setminus\{0\}$ for $j\in\llbracket1, n+1\rrbracket$ and, consequently,
    \[
        i\R \ni \mu_{n+1} \nu_{n+1} P_n = -Q_n \in \R\,.
    \]
    because $P_n, Q_n\in\R$, since $P_0$ and $Q_0$ are real and $M_m$ has real coefficients (even when the $\nu_m$'s are purely imaginary).
    Thus, since $\mu_{n+1}\nu_{n+1}\neq0$, we obtain $P_n=Q_n=0$. However, the matrices $M_m$ are all invertible, since $\det M_m=1$, contradicting~\eqref{Formula_Pn_Qn}:
    \[
        \begin{pmatrix} 0 \\ 0 \end{pmatrix} = \begin{pmatrix} P_n \\ Q_n \end{pmatrix} = M_n \cdots  M_1 \begin{pmatrix} 1 \\ 0 \end{pmatrix}. \qedhere
    \]
\end{proof}

The following proposition establishes the relation between the $k_\ell$'s and the Love waves.
\begin{proposition}[Characterization of Love waves]\label{nPlus1_layers_equivalence_kell_Love_waves}
    Let $n\in\N\setminus\{0\}$ and the $k_\ell(\omega)$'s be as in Definition~\ref{Def_kell_nPlus1_layers}. Then,
    \[
        \{ (\omega, k) : \text{a Love wave exists at } (\omega, k) \} = \{ (\omega, k_\ell(\omega)) : k_\ell(\omega)\neq\omega/C_\infty \}_{\omega>0,\, \ell\geq1}\,.
    \]
\end{proposition}
\begin{proof}
    Using~\eqref{nPlus1_layers_relation_defining_Love_waves_preversion} and Lemma~\ref{Alt_Def_kell_nPlus1_layers}, we obtain
    \[
        \begin{multlined}[b][0.9\textwidth]
            \{ (\omega, k) : \text{a Love wave exists at } (\omega, k) \} \\
            \begin{aligned}[t]
                &= \{ (\omega, k)\in(0,+\infty)\times(\omega/C_\infty, +\infty) : f_n(\omega, k)=0 \} \\
                &= \{ (\omega, k)\in(0,+\infty)\times\R : \exists\,\ell\geq1\,, k=k_\ell(\omega)\neq\omega/C_\infty \}
            \end{aligned} \\
            =: \{ (\omega, k_\ell(\omega)) : k_\ell(\omega)\neq\omega/C_\infty \}_{\omega>0,\, \ell\geq1}\,.
        \end{multlined}
        \qedhere
    \]
\end{proof}

\begin{remark*}
    The reader can notice the small difference between the definition of the $k_\ell(\omega)$'s and the existence of a Love wave at~$(\omega, k)$: the former allows $k_\ell(\omega)=\omega/C_\infty$, while the latter excludes $(\omega, \omega/C_\infty)$. Even though there are $\omega$'s for which $(\omega, \omega/C_\infty)$ is a zero of~$f_n$, there are no Love waves at these couples. Nevertheless, we allow them in the definition of~$k_\ell$ as it will be useful.
\end{remark*}

Finally, the characterization in Proposition~\ref{nPlus1_layers_equivalence_kell_Love_waves} together with Lemma~\ref{Alt_Def_kell_nPlus1_layers} implies that a Love wave existing at $(\omega,k)$ is equivalent to
\begin{equation}\label{nPlus1_layers_relation_defining_Love_waves}
    f_n(\omega, k) = 0 \quad \text{ and } \quad \omega / C_\infty < k < \omega / C_0\,.
\end{equation}
\begin{remark*}
    In particular $\nu_{n+1}=\nu_\infty>0$ and $\nu_0\in i\R_-$.
    Moreover, if $\{C_j\}_{1\leq j \leq n+1}$ is a strictly increasing sequence, then $C_1=C_0$ hence $\nu_1\in i\R_-$.
    Finally, the lower bound means that if there is a $C_j \geq C_\infty$, then the knowledge of the frequency--wavenumber couples of the Love waves will not allow to recover this value $C_j$.
\end{remark*}

As an immediate corollary of Proposition~\ref{Prop_recursive_formula_Dn}, and a key property in some of our proofs, we obtain the \emph{simplicity} of the $k_\ell$'s.
\begin{corollary}[Simplicity of the $k_\ell$'s]\label{simplicity_k}
    Let $n\in\N\setminus\{0\}$. If a Love wave exists at $(\omega, k)$, i.e., there exists an $L^2$-solution~$\phi_{\omega,k}\not\equiv0$ to~\eqref{Problem_equations_new} for the couple~$(\omega, k)$, then there are no other Love waves at $(\omega, k)$ that are linearly independent of $\phi_{\omega,k}$.
\end{corollary}

We now turn to the proof (by induction) of Proposition~\ref{Prop_recursive_formula_Dn}. To that end, for each $n$ we consider $D_n$ as a function of $\nu_{n+1}$: $D_n\equiv D_n(\nu_{n+1})$, and we define
\begin{equation}
    \bar{D}_n\equiv\bar{D}_n(\nu_{n+1}) := e^{\nu_{n+1} H_{n+1}} D_n(\nu_{n+1})
    \quad \text{and} \quad \widetilde{D}_n\equiv\widetilde{D}_n(\nu_{n+1}):=\frac{e^{\nu_{n+1} H_{n+1}}}{2^n \prod\limits_{j=2}^n \mu_j \nu_j} D_n\,.
\end{equation}
 
One of the key points in the proof is that, for any $n\geq2$, the one-to-last and two-to-last columns of $\mathbb{M}_n$ are exactly the same up to replacing $\nu_n$ by $-\nu_n$, since ``$R_n^l(\nu_n)=R_n^r(-\nu_n)$'' (see the definitions in~\eqref{Formula_D_n}). Consequently, expanding the determinant of $D_{n+1}$ will make appear both $\widetilde{D}_n(\nu_{n+1})$ and $\widetilde{D}_n(-\nu_{n+1})$ which depend very simply on $\pm\nu_{n+1}$, namely, only through the two factors $\mu_{n+1} \nu_{n+1}$ appearing in~\eqref{Formula_det_any_n}.

\begin{proof}[Proof of Proposition~\ref{Prop_recursive_formula_Dn}]
    We start by the result on the rank of $\mathbb{M}_n$. As explained earlier, inspecting~\eqref{Formula_D_n}, we see that the submatrix $\widetilde{\mathbb{M}}_n$ obtained by removing from $\mathbb{M}_n$ the last column as well as either the first or the second row is a block (upper) triangular matrix with first diagonal element $\tilde{L}_1$ being either $2\mu_1 \nu_1 \sinh[ \nu_1 T_1 ]$ or $2\cosh[ \nu_1 T_1 ]$, and the $L_j$'s, $2\leq j \leq n$, on the $n-1$ other diagonal blocks. Therefore,
    \[
        \det \widetilde{\mathbb{M}}_n = \tilde{L}_1 \prod_{j=2}^n \det L_j \,.
    \]
    Moreover, on one hand, for $2\leq j \leq n$,
    \[
        L_j = 
        \left\{
        \begin{aligned}
            &\begin{pmatrix}
                +e^{-\nu_j H_{j+1}} & +e^{+\nu_j H_{j+1}} \\
                -\mu_j \nu_j e^{-\nu_j H_{j+1}} & +\mu_j \nu_j e^{+\nu_j H_{j+1}}
            \end{pmatrix} \qquad &\text{if } \nu_j\neq0 \\
            &\begin{pmatrix}
                H_{j+1} & 1 \\
                \mu_j & 0
            \end{pmatrix} \qquad &\text{otherwise,} 
        \end{aligned}
        \right.
    \]
    where the formula for $\nu_j=0$ is due to the boundary conditions combined with the fact that the $L^2$-solution is then linear on the $j$-th layer. Consequently, still for $2\leq j \leq n$, $\det L_j = 2\mu_j\nu_j$ if $\nu_j\neq0$ and $\det L_j = -\mu_j$ if $\nu_j=0$. In particular, $\det L_j\neq0$ for $2\leq j \leq n$.
    On another hand, at least one of the values $2\mu_1 \nu_1 \sinh[ \nu_1 T_1 ]$ and $2\cosh[ \nu_1 T_1 ]$ is non-zero. Hence, we choose $\widetilde{\mathbb{M}}_n$ (i.e., the row of $\mathbb{M}_n$ that we remove to form $\widetilde{\mathbb{M}}_n$) in such a way that the number $\tilde{L}_1$ is non-zero.
    We have therefore constructed a submatrix $\widetilde{\mathbb{M}}_n\in\C_{2n-1,2n-1}$ of $\mathbb{M}_n$ with $\det \widetilde{\mathbb{M}}_n = \tilde{L}_1 \prod_{j=2}^n \det L_j \neq0$, hence $\rank \mathbb{M}_n\geq2n-1$.
    
    We now turn to the result on the determinant $D_n$. First, the fact that, for $n$ fixed, the formulae are real-valued is due to the fact that for a Love wave to exist (for a given $n$ fixed), it must hold that $\nu_{n+1}\in\R$ and that $\nu_j\in\R\cup i\R$ for $j\leq n$.
    
    We start by assuming that all $\nu_j$'s are non-zero.
    For $n=1$, we have
    \[
        \mathbb{M}_1 =
        \begin{pmatrix}
            2 \cosh[ \nu_1 T_1 ] & - e^{-\nu_2 H_2} \\
            2 \mu_1 \nu_1 \sinh[ \nu_1 T_1 ] & \mu_2 \nu_2 e^{-\nu_2 H_2}
        \end{pmatrix},
    \]
    hence
    $e^{\nu_2 H_2} D_1 / 2 = \mu_2 \nu_2 \cosh[ \nu_1 T_1 ] + \mu_1 \nu_1 \sinh[ \nu_1 T_1 ]$ and the claim~\eqref{Formula_det_any_n} is verified.
    Assume now that~\eqref{Formula_det_any_n} holds for some $n\geq1$. Then, using again $A_j = \mu_j \nu_j$ for shortness, from
    \begingroup
        \renewcommand*{\arraystretch}{1.1}
        \[
            D_{n+1} = \text{\tiny
            $\arraycolsep=0.3\arraycolsep\ensuremath{
		\begin{vNiceArray}{ccccDcc}[margin,custom-line = {letter=D, tikz=dashed}]
			\Block{5-4}<\LARGE>{\mathbb{M}_n} & & & & 0 & 0 \\[-5pt]
			& & & & \vdots & \vdots \\
			& & & & 0 & 0 \\
			& & & & -e^{+\nu_{n+1} H_{n+1}} & 0 \\
			& & & & -A_{n+1}e^{+\nu_{n+1} H_{n+1}} & 0 \\
			\hdashline
			0 & \cdots & 0 & +e^{-\nu_{n+1} H_{n+2}} & +e^{+\nu_{n+1} H_{n+2}} & -e^{-\nu_{n+2} H_{n+2} \vphantom{H_{n+2}^2}} \\
			0 & \cdots & 0 & -A_{n+1}e^{-\nu_{n+1} H_{n+2}} & +A_{n+1}e^{+\nu_{n+1} H_{n+2}} & +A_{n+2}e^{-\nu_{n+2} H_{n+2}}
		\end{vNiceArray}
            }$
            }\,,
        \]
        we obtain
        \begin{align*}
            \bar{D}_{n+1} &=
            A_{n+2} \text{\tiny
            	$\arraycolsep=0.3\arraycolsep\ensuremath{
               	\begin{vNiceArray}{ccccDc}[margin,custom-line = {letter=D, tikz=dashed}]
				\Block{5-4}<\LARGE>{\mathbb{M}_n} & & & & 0 \\[-5pt]
				& & & & \vdots \\
				& & & & 0 \\
				& & & & -e^{+\nu_{n+1} H_{n+1}} \\
				& & & & -A_{n+1}e^{+\nu_{n+1} H_{n+1}} \\
				\hdashline
				0 & \cdots & 0 & +e^{-\nu_{n+1} H_{n+2}} & +e^{+\nu_{n+1} H_{n+2} \vphantom{H_{n+2}^2} }
			\end{vNiceArray}
			}$
                }
                + \text{\tiny
			$\arraycolsep=0.3\arraycolsep\ensuremath{
			\begin{vNiceArray}{ccccDc}[margin,custom-line = {letter=D, tikz=dashed}]
				\Block{5-4}<\LARGE>{\mathbb{M}_n} & & & & 0 \\[-5pt]
				& & & & \vdots \\
				& & & & 0 \\
				& & & & -e^{+\nu_{n+1} H_{n+1}} \\
				& & & & -A_{n+1}e^{+\nu_{n+1} H_{n+1}} \\
				\hdashline
				0 & \cdots & 0 & -A_{n+1}e^{-\nu_{n+1} H_{n+2}} & +A_{n+1}e^{+\nu_{n+1} H_{n+2} \vphantom{H_{n+2}^2} }
			\end{vNiceArray}
			}$
		} \\
            &=
            \begin{multlined}[t][0.9\textwidth]
                A_{n+2} \left( e^{+\nu_{n+1} H_{n+2}} D_n(\nu_{n+1}) - e^{-\nu_{n+1} H_{n+2}} D_n(-\nu_{n+1}) \right)\\
                + A_{n+1} \left( e^{+\nu_{n+1} H_{n+2}} D_n(\nu_{n+1}) + e^{-\nu_{n+1} H_{n+2}} D_n(-\nu_{n+1}) \right)
            \end{multlined} \\
            &=
            \begin{multlined}[t][0.9\textwidth]
                A_{n+2} \left( e^{+\nu_{n+1} T_{n+1}} \bar{D}_n(\nu_{n+1}) - e^{-\nu_{n+1} T_{n+1}} \bar{D}_n(-\nu_{n+1}) \right)\\
                + A_{n+1} \left( e^{+\nu_{n+1} T_{n+1}} \bar{D}_n(\nu_{n+1}) + e^{-\nu_{n+1} T_{n+1}} \bar{D}_n(-\nu_{n+1}) \right),
            \end{multlined}
        \end{align*}
        where we used, for the one-to-last equality, that the last column of~$\mathbb{M}_n$ is the same as the column in the top-right block up to changing $\nu_{n+1}$ into $-\nu_{n+1}$.
    \endgroup
    Consequently,
    \begin{align*}
        \widetilde{D}_{n+1} &= \frac{\bar{D}_{n+1}}{2^{n+1} \prod\limits_{j=2}^{n+1} \mu_j \nu_j} \\
        &\begin{multlined}[t][0.9\textwidth]
            = \frac{A_{n+2}}{A_{n+1}} \Biggl( \frac{e^{+\nu_{n+1} T_{n+1}}}{2} \frac{\bar{D}_n(\nu_{n+1})}{2^n \prod\limits_{j=2}^n A_j} - \frac{e^{-\nu_{n+1} T_{n+1}}}{2} \frac{\bar{D}_n(-\nu_{n+1})}{2^n \prod\limits_{j=2}^n A_j} \Biggr)\\
            + \frac{e^{+\nu_{n+1} T_{n+1}}}{2} \frac{\bar{D}_n(\nu_{n+1})}{2^n \prod\limits_{j=2}^n A_j} + \frac{e^{-\nu_{n+1} T_{n+1}}}{2} \frac{\bar{D}_n(-\nu_{n+1})}{2^n \prod\limits_{j=2}^n A_j}
        \end{multlined} \\
        &= \left( 1 + \frac{A_{n+2}}{A_{n+1}} \right) \frac{e^{+\nu_{n+1} T_{n+1}}}{2} \widetilde{D}_n(\nu_{n+1}) + \left( 1 - \frac{A_{n+2}}{A_{n+1}} \right) \frac{e^{-\nu_{n+1} T_{n+1}}}{2} \widetilde{D}_n(-\nu_{n+1})\,,
    \end{align*}
    and we can now use the induction assumption
    \[
        \widetilde{D}_n(\pm\nu_{n+1}) = \left( \pm A_{n+1} P_{n-1} + Q_{n-1} \right) \cosh[\nu_n T_n] + \left( A_n^2 P_{n-1} \pm A_{n+1} Q_{n-1} \right) \frac{\sinh[\nu_n T_n]}{A_n}
    \]
    to obtain
    \begin{align*}
        \widetilde{D}_{n+1}
        &\begin{multlined}[t][0.9\textwidth]
            = \left( 1 + \frac{A_{n+2}}{A_{n+1}} \right) \frac{e^{+\nu_{n+1} T_{n+1}}}{2} \left( A_{n+1} P_{n-1} + Q_{n-1} \right) \cosh[\nu_n T_n] \\
            + \left( 1 - \frac{A_{n+2}}{A_{n+1}} \right) \frac{e^{-\nu_{n+1} T_{n+1}}}{2} \left( - A_{n+1} P_{n-1} + Q_{n-1} \right) \cosh[\nu_n T_n] \\
            + \left( 1 + \frac{A_{n+2}}{A_{n+1}} \right) \frac{e^{+\nu_{n+1} T_{n+1}}}{2} \left( A_n^2 P_{n-1} + A_{n+1} Q_{n-1} \right) \frac{\sinh[\nu_n T_n]}{A_n} \\
            + \left( 1 - \frac{A_{n+2}}{A_{n+1}} \right) \frac{e^{-\nu_{n+1} T_{n+1}}}{2} \left( A_n^2 P_{n-1} - A_{n+1} Q_{n-1} \right) \frac{\sinh[\nu_n T_n]}{A_n}
        \end{multlined} \\
        &=\begin{aligned}[t]
            &\begin{multlined}[t][0.87\textwidth]
                \left[ A_{n+2} \left( P_{n-1} \cosh[\nu_n T_n] + Q_{n-1} \frac{\sinh[\nu_n T_n]}{A_n} \right) \right. \\
                \left.+ \left(A_n^2 P_{n-1} \frac{\sinh[\nu_n T_n]}{A_n} + Q_{n-1} \cosh[\nu_n T_n] \right) \right] \cosh[\nu_{n+1} T_{n+1}]
            \end{multlined}\\
            &\begin{multlined}[t][0.87\textwidth]
                + \left[ A_{n+1}^2 \left( P_{n-1} \cosh[\nu_n T_n] + Q_{n-1} \frac{\sinh[\nu_n T_n]}{A_n} \right)  \right. \\
                \left. + A_{n+2} \left( A_n^2 P_{n-1} \frac{\sinh[\nu_n T_n]}{A_n} + Q_{n-1} \cosh[\nu_n T_n] \right) \right] \frac{\sinh[\nu_{n+1} T_{n+1}]}{A_{n+1}}
            \end{multlined}
        \end{aligned} \\
        &=\begin{aligned}[t]
            &\left( A_{n+2} P_n + Q_n \right) \cosh[\nu_{n+1} T_{n+1}] + \left( A_{n+1}^2 P_n + A_{n+2} Q_n \right) \frac{\sinh[\nu_{n+1} T_{n+1}]}{A_{n+1}}\,.
        \end{aligned}
    \end{align*}
    This concludes the proof in the case $\prod_{j=2}^n \nu_j\neq0$, since it is exactly
    \[
        \frac{e^{\nu_{n+2} H_{n+2}}}{2^{n+1} \prod\limits_{j=2}^{n+1} \mu_j \nu_j} D_{n+1}=\widetilde{D}_{n+1} = \mu_{n+2} \nu_{n+2} P_{n+1} + Q_{n+1}\,,
    \]
    where we used the definitions of $P_{n+1}$ and $Q_{n+1}$ given in~\eqref{Formula_Pn_Qn}. Namely,
    \[
        \left\{
        \begin{aligned}
            P_{n+1} &= P_n \cosh[\nu_{n+1} T_{n+1}] + Q_n \frac{\sinh[\nu_{n+1} T_{n+1}]}{\mu_{n+1} \nu_{n+1}}\\
            Q_{n+1} &= \mu_{n+1} \nu_{n+1} P_n \sinh[\nu_{n+1} T_{n+1}] + Q_n \cosh[\nu_{n+1} T_{n+1}]\,.
        \end{aligned}
        \right.
    \]
    
    We now prove the case ``$\nu_j=0$''. Roughly speaking, we prove that~\eqref{Formula_det_any_n} in such limit case is nothing else than passing to the limit $\nu_j\to0$ in the formula for $\prod_{j=2}^n \nu_j\neq0$.
    
    The r.h.s.\ of~\eqref{Formula_det_any_n} passes to the limit since the formula of $M_m$ in~\eqref{Formula_Pn_Qn} at $\nu_m=0$ is indeed the limit $\nu_m\to0$ of its formula for $\nu_m\neq0$.
    For the l.h.s., we will expand the determinant $D_n$ according to columns where $\nu_j$ appears for both the general case and the case $\nu_j=0$. Before doing so, we treat apart the case of $\nu_1$ since it appears only in the first column. In the general case it appears ---see~\eqref{Formula_D_n}--- through
    \[
        L_1^r:= 2 
        \begin{pmatrix}
            \cosh[ \nu_1 T_2 ] \\
            \mu_1 \nu_1 \sinh[ \nu_1 T_2 ]
        \end{pmatrix},
    \]
    while in the case $\nu_1=0$ the solution is linear on the first layer thence, by the boundary conditions, $L_1^r$ is replaced by $\begin{pmatrix} 2\\0\end{pmatrix}$. Since the latter is nothing than the limit, when $\nu_1\to0$ of the former, we of course have the claim:
    \[
        \lim\limits_{\nu_1\to0} \frac{e^{\nu_{n+1} H_{n+1}}}{2^n \prod\limits_{j=2}^n \mu_j \nu_j} D_n(\nu_1) = \frac{e^{\nu_{n+1} H_{n+1}}}{2^n \prod\limits_{j=2}^n \mu_j \nu_j} D_n(\nu_1=0)\,.
    \]
    
    For $j>1$, we first notice that $\nu_j$ appears only at the $(2j-2)$-th and $(2j-1)$-th column of $D_n$. Thus, omitting the dependency in $n$, we define $\widetilde{D}_p$ as the \emph{minor} where we remove the $(2j-2)$-th column and the $p$-th row, and $\widetilde{D}_{p}^{q}=\widetilde{D}_{q}^{p}$ as the \emph{minor} where we remove the $(2j-2)$-th and $(2j-1)$-th columns and the $p$-th and $q$-th rows.
    Moreover, we note that if $\nu_j=0$, then the solution is linear on the $j$-th layer thence, by the boundary conditions, the determinant is given by~\eqref{Formula_D_n} but with
    \[
        L_j:=
        \begin{pmatrix}
            +e^{-\nu_j H_{j+1}} & +e^{+\nu_j H_{j+1}} \\
            -\mu_j \nu_j e^{-\nu_j H_{j+1}} & +\mu_j \nu_j e^{+\nu_j H_{j+1}}
        \end{pmatrix}
        \quad \textrm{ and } \quad
        R_j:=
        \begin{pmatrix}
            -e^{-\nu_j H_j} & -e^{+\nu_j H_j} \\
            +\mu_j \nu_j e^{-\nu_j H_j} & -\mu_j \nu_j e^{+\nu_j H_j}
        \end{pmatrix}
    \]
    replaced by
    \[
        L_j(\nu_j=0):=
        \begin{pmatrix}
            H_{j+1} & 1 \\
            \mu_j & 0
        \end{pmatrix}
        \quad \textrm{ and } \quad
        R_j(\nu_j=0):=
        \begin{pmatrix}
            -H_j & -1 \\
            -\mu_j & 0
        \end{pmatrix}.
    \]
    
    We first expand $D_n$ when $\nu_j\neq0$ according to the $(2j-2)$-th and $(2j-1)$-th columns and obtain
    \begin{align*}
        D_n &= -\left(-e^{-\nu_j H_j}\right) \widetilde{D}_{2j-3} + \mu_j \nu_j e^{-\nu_j H_j} \widetilde{D}_{2j-2} - e^{-\nu_j H_{j+1}} \widetilde{D}_{2j-1} + \left(-\mu_j \nu_j e^{-\nu_j H_{j+1}} \right) \widetilde{D}_{2j} \\
        &=\begin{multlined}[t][0.92\textwidth]
            e^{-\nu_j H_j} \left[ -\left(-\mu_j \nu_j e^{+\nu_j H_j} \right) \widetilde{D}_{2j-3}^{2j-2} + e^{+\nu_j H_{j+1}} \widetilde{D}_{2j-3}^{2j-1} - \mu_j \nu_j e^{+\nu_j H_{j+1}} \widetilde{D}_{2j-3}^{2j} \right] \\
            + \mu_j \nu_j e^{-\nu_j H_j} \left[ -\left(- e^{+\nu_j H_j}\right) \widetilde{D}_{2j-2}^{2j-3} + e^{+\nu_j H_{j+1}} \widetilde{D}_{2j-2}^{2j-1} - \mu_j \nu_j e^{+\nu_j H_{j+1}} \widetilde{D}_{2j-2}^{2j} \right] \\
            - e^{-\nu_j H_{j+1}} \left[ -\left(- e^{+\nu_j H_j}\right) \widetilde{D}_{2j-1}^{2j-3} + \left(-\mu_j \nu_j e^{+\nu_j H_j} \right) \widetilde{D}_{2j-1}^{2j-2} - \mu_j \nu_j e^{+\nu_j H_{j+1}} \widetilde{D}_{2j-1}^{2j} \right] \\
            - \mu_j \nu_j e^{-\nu_j H_{j+1}} \left[ -\left(- e^{+\nu_j H_j}\right) \widetilde{D}_{2j}^{2j-3} + \left(-\mu_j \nu_j e^{+\nu_j H_j} \right) \widetilde{D}_{2j}^{2j-2} - e^{+\nu_j H_{j+1}} \widetilde{D}_{2j}^{2j-1} \right]
        \end{multlined}\\
        &=\begin{multlined}[t][0.92\textwidth]
            2\mu_j \nu_j \widetilde{D}_{2j-3}^{2j-2} + 2 \mu_j \nu_j \widetilde{D}_{2j-1}^{2j} + 2 \mu_j \nu_j \cosh\left[ \nu_j T_j \right] \widetilde{D}_{2j-2}^{2j-1} - 2 \mu_j \nu_j \cosh\left[ \nu_j T_j \right] \widetilde{D}_{2j-3}^{2j} \\
            + 2\sinh\left[ \nu_j T_j \right] \widetilde{D}_{2j-3}^{2j-1} - 2 \mu_j^2 \nu_j^2 \sinh\left[ \nu_j T_j \right] \widetilde{D}_{2j-2}^{2j}\,.
        \end{multlined}
    \end{align*}
    Thus,
    \begin{align*}
        \frac{D_n}{2 \mu_j \nu_j}
        &=\begin{multlined}[t][0.85\textwidth]
            \widetilde{D}_{2j-3}^{2j-2} + \widetilde{D}_{2j-1}^{2j} + \cosh\left[ \nu_j T_j \right] \widetilde{D}_{2j-2}^{2j-1} - \cosh\left[ \nu_j T_j \right] \widetilde{D}_{2j-3}^{2j} \\
            + \frac{\sinh\left[ \nu_j T_j \right]}{\mu_j \nu_j} \widetilde{D}_{2j-3}^{2j-1} - \mu_j \nu_j \sinh\left[ \nu_j T_j \right] \widetilde{D}_{2j-2}^{2j}
        \end{multlined} \\
        &\underset{\nu_j\to0}{\longrightarrow}
        \widetilde{D}_{2j-3}^{2j-2} - \widetilde{D}_{2j-3}^{2j} + \widetilde{D}_{2j-2}^{2j-1} + \widetilde{D}_{2j-1}^{2j} + \frac{T_j}{\mu_j} \widetilde{D}_{2j-3}^{2j-1}\,.
    \end{align*}
    
    We now expand $D_n$ when $\nu_j=0$ according to $(2j-2)$-th and $(2j-1)$-th columns:
    \begin{align*}
        D_n(\nu_j=0)&= -\left(-H_j\right) \widetilde{D}_{2j-3} + \left(-\mu_j\right) \widetilde{D}_{2j-2} - H_{j+1} \widetilde{D}_{2j-1} + \mu_j \widetilde{D}_{2j} \\
        &=\begin{multlined}[t][0.8\textwidth]
            H_j \left[ -0\times \widetilde{D}_{2j-3}^{2j-2} + 1\times \widetilde{D}_{2j-3}^{2j-1} - 0\times \widetilde{D}_{2j-3}^{2j} \right] \\
            - \mu_j \left[ -(-1)\times \widetilde{D}_{2j-2}^{2j-3} + 1\times \widetilde{D}_{2j-2}^{2j-1} - 0\times \widetilde{D}_{2j-2}^{2j} \right] \\
            - H_{j+1} \left[ -(- 1) \widetilde{D}_{2j-1}^{2j-3} + 0\times \widetilde{D}_{2j-1}^{2j-2} - 0\times \widetilde{D}_{2j-1}^{2j} \right] \\
            + \mu_j \left[ -(-1) \widetilde{D}_{2j}^{2j-3} + 0\times \widetilde{D}_{2j}^{2j-2} - 1\times \widetilde{D}_{2j}^{2j-1} \right]
        \end{multlined}\\
        &= - T_j \widetilde{D}_{2j-1}^{2j-3} - \mu_j \left[ \widetilde{D}_{2j-3}^{2j-2} + \widetilde{D}_{2j-2}^{2j-1} - \widetilde{D}_{2j-3}^{2j} + \widetilde{D}_{2j-1}^{2j} \right].
    \end{align*}
    Thus, this concludes the proof of~\eqref{Formula_det_any_n} since we have
    \[
        -\frac{D_n(\nu_j=0)}{\mu_j} = \widetilde{D}_{2j-3}^{2j-2} - \widetilde{D}_{2j-3}^{2j} + \widetilde{D}_{2j-2}^{2j-1} + \widetilde{D}_{2j-1}^{2j} + \frac{T_j}{\mu_j} \widetilde{D}_{2j-3}^{2j-1} = \lim\limits_{\nu_j\to0} \frac{D_n}{2 \mu_j \nu_j}\,. \qedhere
    \]
\end{proof}

To conclude this section, and because they will be useful, we define the following functions and give some of their properties, the proofs of which are postponed to Appendix~\ref{Appendix_nplus1_layers}.
\begin{definition}\label{Def_f_tilde_P_tilde_Q_tilde_nPlus1_layers}
    Let $n\in\N\setminus\{0\}$. Let $P_m, Q_m, f_n:(0,+\infty)\times[\omega / C_\infty, \omega / C_0)\to\R$, $m\in\llbracket0,n\rrbracket$, be respectively defined by the formulae in Proposition~\ref{Prop_recursive_formula_Dn} and in~\eqref{Def_f_n}, and $\bar{\nu}_\infty$ be defined in~\eqref{Def_nu_bar}. Define on $[0,+\infty)\times[1/C_\infty, 1/C_0)$ the three (real-valued) functions $\tilde{P}_m$, $\tilde{Q}_m$, $m\in\llbracket0,n\rrbracket$, and $\tilde{f}_n$ by
    \begin{align}
        &\left\{
        \begin{aligned}
            &\tilde{P}_m(\omega, y) = P_m(\omega, \omega y) \hphantom{\omega^{-1}} &&\text{ on } (0,+\infty)\times[1/C_\infty, 1/C_0)\,, \\
            &\tilde{P}_m(0, y) = 1 \hphantom{\omega^{-1}} &&\text{ for } y\in[1/C_\infty, 1/C_0)\,,
        \end{aligned}
        \right.
        \label{Def_P_tilde_nPlus1_layers} \\
        &\left\{
        \begin{aligned}
            &\tilde{Q}_m(\omega, y) = \omega^{-1} Q_m(\omega, \omega y) &&\text{ on } (0,+\infty)\times[1/C_\infty, 1/C_0)\,, \\
            &\tilde{Q}_m(0, y) = 0 &&\text{ for } y\in[1/C_\infty, 1/C_0)\,,
        \end{aligned}
        \right.
        \label{Def_Q_tilde_nPlus1_layers}
        \intertext{and}
        &\left\{
        \begin{aligned}
            &\tilde{f}_n(\omega, y) = \omega^{-1} f_n(\omega, \omega y) &&\text{ on } (0,+\infty)\times[1/C_\infty, 1/C_0)\,, \\
            &\tilde{f}_n(0, y) = \mu_\infty \bar{\nu}_\infty(y) &&\text{ for } y\in[1/C_\infty, 1/C_0)\,,
        \end{aligned}
        \right.
        \label{Def_f_tilde_nPlus1_layers}
    \end{align}
    i.e.,
    \begin{equation}\label{Alt_Def_f_tilde_nPlus1_layers}
        \tilde{f}_n(\omega, y) = \mu_\infty \bar{\nu}_\infty(y) \tilde{P}_n(\omega,y) + \tilde{Q}_n(\omega,y)\,.
    \end{equation}
\end{definition}

\begin{lemma}\label{Continuity_tilde_functions_nPlus1_layers}
    The functions (of two variables) $\tilde{P}_m$, $\tilde{Q}_m$, $m\in\llbracket0,n\rrbracket$, and $\tilde{f}_n$ are continuous.
\end{lemma}

\begin{lemma}\label{Equivalence_zeros_f_n_and_tilde_f_n}
    $\ker \tilde{f}_n = \{(0,1/C_\infty)\} \cup \left\{ (\omega,y): (\omega,\omega y) \in \ker f_n \right\}$.
\end{lemma}

\section{Regularity and monotonicity of branches of wavenumbers}\label{Section_monotonicity_existence_LoveWaves}
The aim here is to prove regularity in $\omega$ of the branches $k_\ell$ ---and of the associated functions $\phi_{\ell,\omega}$--- and the monotonicity of the functions $\omega\mapsto k_\ell(\omega)/\omega$. Namely, the goal is to prove the following result.
\begin{theorem}\label{Thm_monotonicity_nPlus1_layers}
    Let $n\in\N\setminus\{0\}$ and the $k_\ell$'s be as in Definition~\ref{Def_kell_nPlus1_layers}. Assume moreover $C_0<\widetilde{C}_2$ if $n\geq3$.
    Then, for any $\ell\geq1$ there exists $\omega_\ell\geq0$ such that the function
    \begin{align*}
        (\omega_\ell, +\infty) &\to (1/C_\infty, 1/C_0) \\
        \omega &\mapsto k_\ell(\omega)/\omega
    \end{align*}
     is analytic, bijective, and increasing, and that $k_\ell(\omega_\ell)=\omega_\ell/C_\infty$ if $\omega_\ell>0$.
    
    Moreover, a Love wave exists at~$(\omega, k)$ if and only if
    \[
        (\omega, k) \in \{(\omega, k_\ell(\omega)) : \omega>\omega_\ell \}_{\ell\geq1}\,.
    \]
\end{theorem}
The strategy is to first prove that the $k_\ell$'s are in finite number for any fixed $\omega$. Then, to deduce from it their regularity by analytic perturbation theory and finally to prove that the derivative of $k_\ell(\omega)/\omega$ is positive.

We believe the restriction $C_0<\widetilde{C}_2$ if $n\geq3$ to be purely technical and, in any case, it is only needed for the proof of Lemma~\ref{Lemma_zeros_ftilde_any_y_nPlus1_layers}.

First, we have that at any fixed $\omega$ the $k_\ell(\omega)$'s are in finite number.
\begin{proposition}\label{Lemma_finite_number_nPlus1_layers}
    Let $n\in\N\setminus\{0\}$, $\omega>0$, and the $k_\ell(\omega)$'s be as in Definition~\ref{Def_kell_nPlus1_layers}. Then, there is a finite number of $k_\ell(\omega)$'s.
\end{proposition}
The idea of the proof is to extend holomorphically $f_n$
to (part of) the complex plane and to prove that one of the $k_\ell$ being an accumulation points of the kernel of $y\mapsto f_n(\omega, y)$ would yield a contradiction on this extension.
\begin{proof}
    We fix $\omega>0$ and, for shortness, we omit the dependence in $\omega$ of the $k_\ell$'s in this proof.
    The $k_\ell$'s being bounded (Lemma~\ref{Alt_Def_kell_nPlus1_layers}), them being in finite number is equivalent to them being isolated.
    
    Assume on the contrary that there exists a $k_\ell$, denoted $\omega k_\star$, that is an accumulation point of the kernel of $y\mapsto f_n(\omega, y)$, i.e., of the set $\{k_\ell\}_{j\geq1}$. Since the $k_\ell$'s are simple (Corollary~\ref{simplicity_k}) ---that is, $\ell\neq\ell' \Rightarrow k_\ell\neq k_{\ell'}$---, the accumulation point $\omega k_\star$ is in particular simple\footnote{Finite multiplicity would actually be enough.} and there exists a subsequence (denoted the same) such that $k_\ell\to \omega k_\star$ as $\ell\to+\infty$ but $k_\ell\neq \omega k_\star$ for all $\ell$.
    
    We now define $C_-<C_+$ such that $\{k_\ell\}_{j\geq1} \cap [\omega/C_+, \omega/C_-]$ also admits $\omega k_\star$ as an accumulation point, that $1/C_j \not\in (1/C_+, 1/C_-)$ for all $j$, and that $C_0<C_-<C_+<C_\infty$. To define them properly, we distinguish two cases: either $k_\star \neq 1/C_j$ for all $j$ or $k_\star = 1/C_j$ for some $j$, in which case we have $C_\infty\neq C_j\neq C_0$ by~\eqref{nPlus1_layers_relation_defining_Love_waves}. See Figure~\ref{Fig_Def_Cminus_Cplus_nplus1} for a sketch of the definition. On the one hand, when $k_\star \neq 1/C_j$ for all $j$, then the required properties are satisfied by $C_- := k_\star - \epsilon < C_+ := k_\star + \epsilon$ for $\epsilon>0$ small enough. On the other hand, when there exists a $C_j$ such that $k_\star= 1/C_j$, we distinguish two subcases. Either $\{k_\ell\}_{j\geq1} \cap [\omega k_\star, \omega k_\star + 1)$ admits $\omega k_\star$ as an accumulation point, in which case we define $C_+ := C_j = 1 / k_\star$ and $C_-$ as the largest $C_j$'s such that $k_\star< 1/C_j$, with the special case that, if this results in $C_-=C_0$ ---hence $k_\star < 1/C_0$---, then we replace the value of $C_-$ by $C_-=2/(k_\star + 1/C_0)$ in order to ensure $C_0<C_-<C_+<C_\infty$. Or $\omega k_\star$ is not an accumulation point of $\{k_\ell\}_{j\geq1} \cap [\omega k_\star, \omega k_\star + 1)$, then it is one of $\{k_\ell\}_{j\geq1} \cap (\omega k_\star - 1, \omega k_\star]$ and we define $C_- := C_j = 1 / k_\star$ and $C_+$ as the smallest $C_j$'s such that $k_\star> 1/C_j$, with the special case that, if this results in $C_+=C_\infty$ ---hence $k_\star > 1/C_\infty$---, then we replace the value of $C_+$ by $C_+=2/(k_\star + 1/C_\infty)$ in order to ensure $C_0<C_-<C_+<C_\infty$.
    \begin{figure}[!htp]
        \centering
        \includegraphics[width=\textwidth]{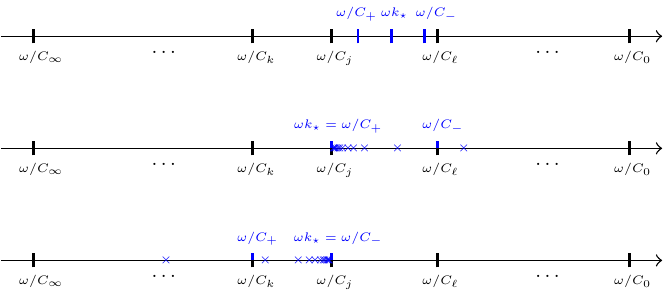}
        \caption{Sketch of the definition of $C_\pm$. Top: $k_\star \neq 1/C_j$ for all $j$.\\Center \& bottom: $k_\star = 1/C_j$ for some $j$, the two subcases.}
        \label{Fig_Def_Cminus_Cplus_nplus1}
    \end{figure}
    
    Next, we consider $f_n$ restricted to the interval $[\omega/C_+, \omega/C_-]$, which is a subset of $(\omega / C_\infty, \omega / C_0)$ due to the careful definitions of $C_\pm$. That is, working with $k_\ell/\omega$ instead of~$k_\ell$ itself, we consider the function $h_n: [1/C_+, 1/C_-] \to \R$ defined by $h_n(y) := f_n(\omega, \omega y)$. Recalling the definition~\eqref{Def_nu_bar} of $\bar\nu_j$, every $\nu_j$ appearing in $h_n$ is of the form $\nu_j(\omega, \omega y)=\omega\bar\nu_j(y)$.
    The fact that we restrict $h_n$ to $[1/C_+, 1/C_-]$ implies that each $\bar\nu_j$ is either real or purely imaginary on the whole interval $[1/C_+, 1/C_-]$.
    
    We now define $\Omega\subset\C$ such that $\C\setminus\R\subset\Omega$, that there exists a neighborhood of $k_\star$ inside~$\Omega$, and that the extension of $h_n$ (hence keeping the definitions of the $\bar\nu_j$'s fixed to their definition on $[1/C_+, 1/C_-]$) to $\Omega$ is holomorphic. We choose
    \[
        \Omega:=\C\setminus((-\infty, 1/C_+ - \delta]\cup[ 1/C_- + \delta, +\infty)) = (\C\setminus\R) \cup(1/C_+ - \delta, 1/C_- + \delta)\,,
    \]
    with $\delta=0$ if there are no $j$'s such that $k_\star=1/C_j$ and with $0<\delta<1/C_+-1/C_\infty$ otherwise.
    The key point for our argument is that when $k_\star=1/C_j$, even though the corresponding ``fixed'' $\bar\nu_j$ is not holomorphic on any complex (open) neighborhood of $k_\star$ (because $\bar\nu_j$ is, on $\Omega$, either the function $\sqrt{\cdot - C_j^{-2}}$ or the function $-i\sqrt{C_j^{-2} - \cdot}$), the function $h_n$ on the other hand is holomorphic on $\Omega$ in both cases as a sum-product of the ``fixed'' functions $\bar\nu_\infty$, $\cosh[\omega\bar\nu_j T_j]$, $\sinh[\omega\bar\nu_j T_j]/(\omega\bar\nu_j)$, and $\omega\bar\nu_j \sinh[\omega\bar\nu_j T_j]$, which are all holomorphic on $\Omega$: for $\bar\nu_\infty$, it is because the square root function is holomorphic on $\C\setminus(-\infty,0]$ and because $1/C_+ - \delta>1/C_\infty$; for $\cosh[\omega\bar\nu_j T_j]$ and $\sinh[\omega\bar\nu_j T_j]/(\omega\bar\nu_j)$ because of the well-known properties that $z\mapsto \cosh[\sqrt{z}]$ and $z\mapsto \sinh[\sqrt{z}] / \sqrt{z}$ are holomorphic on $\C$; and for $\omega\bar\nu_j \sinh[\omega\bar\nu_j T_j] = (\omega\bar\nu_j)^2 \sinh[\omega\bar\nu_j T_j]/(\omega\bar\nu_j)$ as the product of two holomorphic functions on $\C$.
    
    Given that $\Omega$ is a non-empty, connected, open subset of the complex plane, that $h_n$ is holomorphic on $\Omega$, and that the set $\{z\in\Omega: h_n(z)=0\}$ contains an accumulation point ---namely, $k_\star$---, the Identity Theorem implies that $h_n\equiv0$ on $\Omega$.
    
    We are now left with proving that $h_n$ cannot be trivial everywhere on~$\Omega$. Take $z\in\C$ such that $z^2=C_\infty^{-2} + i\omega^{-2}\eta^2$, $\eta>0$. Then, $z\in\Omega$ and
    \[
        \nu_j(\omega, \omega z) = \omega \bar\nu_j(z) = \sqrt{i\eta^2} = \eta e^{i\frac{\pi}{4}}
    \]
    for the $j$'s such that $C_j=C_\infty$, while for the other $j$'s we have
    \[
        \nu_j(\omega, \omega z) = \omega \bar\nu_j(z) =
        \left\{
            \begin{aligned}
                &\omega \sqrt{z^2-C_j^{-2}}\\
                \text{or}&\\
                -i&\omega \sqrt{C_j^{-2}-z^2}
            \end{aligned}
        \right.
        =
        \left\{
            \begin{aligned}
                &\eta \sqrt{+i + \frac{C_\infty^{-2}-C_j^{-2}}{(\eta/\omega)^2} }\\
                \text{or}&\\
                -i&\eta \sqrt{-i - \frac{C_\infty^{-2}-C_j^{-2}}{(\eta/\omega)^2}}
            \end{aligned}
        \right.
        \underset{\eta\to+\infty}{\sim}
        \left\{
            \begin{aligned}
                +&\eta e^{i\frac{\pi}{4}}\\
                \text{or}&\\
                -&\eta e^{i\frac{\pi}{4}}
            \end{aligned}
        \right..
    \]
    In particular, defining $\epsilon_j\in\{\pm1\}$ such that $\nu_j(\omega, \omega z) = \omega \bar\nu_j(z) \sim \epsilon_j \eta e^{i\frac{\pi}{4}}$, we have
    \begin{align*}
        |\cosh[ \omega \bar\nu_j(z) T_j ]| \sim \left| \cosh\!\left[ \epsilon_j \eta T_j e^{i\frac{\pi}{4}} \right] \right| &= \left| \cos[ \eta T_j/\sqrt{2} ] \cosh[ \eta T_j/\sqrt{2} ] +i \sin[ \eta T_j/\sqrt{2} ] \sinh[ \eta T_j/\sqrt{2} ] \right| \\
        &= \sqrt{ \cosh^2[ \eta T_j/\sqrt{2} ] - \sin^2[ \eta T_j/\sqrt{2} ] } > 0
    \end{align*}
    and, as $\eta\to+\infty$,
    \[
        \omega \bar\nu_j(z) \tanh[ \omega \bar\nu_j(z) T_j ] \sim +\eta e^{i\frac{\pi}{4}}\,.
    \]
    Therefore, for these $z$, and when $\eta\to+\infty$, we have the contradiction to $h_n(z)=0$ that
    \[
        \frac{h_n(z)}{ \prod\limits_{j=1}^n \cosh[ \omega \bar\nu_j(z) T_j ] } \sim K_n \eta e^{i\frac{\pi}{4}} \neq 0\,,
    \]
    where
    \begin{align*}
        K_n := \eta^{-1} e^{-i\frac{\pi}{4}} \begin{pmatrix} \mu_{n+1} \eta e^{i\frac{\pi}{4}} \\ 1 \end{pmatrix} \cdot \left(\left[
        \prod\limits_{m=1}^n
        \begin{pmatrix}
            1 & (\mu_m \eta e^{i\frac{\pi}{4}})^{-1} \\
            \mu_m \eta e^{i\frac{\pi}{4}} & 1
        \end{pmatrix} \right]
        \begin{pmatrix} 1 \\ 0 \end{pmatrix} \right)
    \end{align*}
    is a positive constant depending only on the $\mu_j$'s, $j\in\llbracket1,n+1\rrbracket$. In particular, it is independent of $\eta$ and of $e^{i\frac{\pi}{4}}$, since an easy induction gives
    \[
        \prod\limits_{m=1}^n \begin{pmatrix}
            1 & (\mu_m \eta e^{i\frac{\pi}{4}})^{-1} \\
            \mu_m \eta e^{i\frac{\pi}{4}} & 1
        \end{pmatrix}
        =
        \begin{pmatrix}
            K_n^{11} & K_n^{12} (\eta e^{i\frac{\pi}{4}})^{-1} \\
            K_n^{21} \eta e^{i\frac{\pi}{4}} & K_n^{22}
        \end{pmatrix},
    \]
    where the $K_n^{ij}$'s are independent of $\eta$ and of $e^{i\frac{\pi}{4}}$.
\end{proof}
\begin{remark*}
    For instance, for $n=2$, $K_2 = \mu_1 + \mu_2 + \mu_3 + \mu_1 \mu_3 / \mu_2 > 0$.
\end{remark*}

We now state our regularity result.
\begin{proposition}\label{Prop_regularity_nPlus1_layers}
    Let $n\in\N\setminus\{0\}$. The branches $\omega \mapsto (k(\omega),\phi(\omega))$ satisfying~\eqref{Problem_equations_new} exist on an open interval, and both components $k$ and $\phi$ are analytic. Moreover, the branches of eigenvalues do not cross.
\end{proposition}
For this result, we use the fact that, for the $k_\ell$'s restricted to $(\omega/C_\infty, \omega/C_0)$, $k_\ell^2$ are (pseudo-)eigenvalues of the problem in~\eqref{Problem_equations_intro} and we apply analytic perturbation theory. This theory being nowadays standard, for shortness we only give a sketch of the prove.
\begin{proof}[Sketch of the proof]
    First, by Proposition~\ref{nPlus1_layers_equivalence_kell_Love_waves}, for any fixed $\omega$ we identify 
    \[
        \{ k(\omega) : \exists\, \phi_\omega\in L^2\,, \, (k(\omega),\phi_\omega) \text{ satisfies~\eqref{Problem_equations_new}} \} = \{k_\ell(\omega) : k_\ell(\omega)\neq \omega/C_\infty \}_{\ell\geq1}\,.
    \]
    Consequently, for the rest of this sketch we write $(k_\ell(\omega),\phi_{\ell,\omega})$ instead of $(k(\omega),\phi(\omega))$.
    
    A point to be careful about is that we do not have a (proper) eigenvalue problem, but a generalized one, due to the function $\mu$ multiplying the (pseudo-)eigenvalue $k_\ell^2$. However, writing~\eqref{Problem_equations_new} as $T u=\lambda A u$ with $T= -\partial_z \mu \partial_z - \omega^2 \rho$, $\lambda = k_\ell^2$, and $A = -\mu$, we can apply the arguments of~\cite[Chapter 7 §6]{Kato} to bring ourselves back to the standard analytic perturbation theory.

    Now, by Corollary~\ref{simplicity_k} and Proposition~\ref{Lemma_finite_number_nPlus1_layers}, the eigenvalues are simple and isolated. Therefore, by analytic perturbation theory (see, e.g., \cite[Theorem XII.8 (Kato--Rellich theorem)]{ReeSim4} or the first sections of~\cite[Chapter 7]{Kato}), the branches $\omega \mapsto (k_\ell(\omega)^2,\phi_{\ell,\omega})$ exist on an open interval, both components $\omega \mapsto k_\ell(\omega)^2$ and $\omega \mapsto \phi_{\ell,\omega}$ are analytic, and $k_\ell(\omega)^2$ are simple and isolated. In particular, we have no crossing of branches of eigenvalues.
    
    We then conclude by noticing that, since $k_\ell(\omega) > \omega/C_\infty>0$ by Lemma~\ref{Alt_Def_kell_nPlus1_layers}, the analyticity of $k_\ell^2>0$ implies the one of $k_\ell=\sqrt{k_\ell^2}$.
\end{proof}

With all these results, we can now prove Theorem~\ref{Thm_monotonicity_nPlus1_layers}.
\begin{proof}[Proof of Theorem~\ref{Thm_monotonicity_nPlus1_layers}]
    The second part of the statement is a direct consequence of the first one combined with Proposition~\ref{nPlus1_layers_equivalence_kell_Love_waves}.

    The analyticity has been proved in Proposition~\ref{Prop_regularity_nPlus1_layers}. The proof of the rest of the first part is split into two steps. First, we prove the strict monotonicity of $k_\ell(\omega)/\omega$ on any open interval where $k_\ell$ exists, using the ``(pseudo-)eigenvalue'' property of the $k_\ell$. Second, using their ``zeros of a function'' property, we deduce the bijectivity.
    
    \textbf{Step 1.} We start again by identifying, for any $\omega$,
    \[
        \{k_\ell(\omega) : k_\ell(\omega)\neq \omega/C_\infty \}_{\ell\geq1} = \{ k(\omega) : \exists\, \phi(\omega)\,, \, (k(\omega),\phi(\omega)) \text{ satisfies~\eqref{Problem_equations_new}} \}
    \]
    by Proposition~\ref{nPlus1_layers_equivalence_kell_Love_waves}. We thus write $(k_\ell(\omega),\phi_{\ell,\omega})$ instead of $(k(\omega),\phi(\omega))$.
    
    By definition of $\phi_{\ell,\omega}$, we have $-\partial_z \mu \partial_z \phi_{\ell,\omega} - \omega^2 \rho \phi_{\ell,\omega} = -k_\ell^2(\omega) \mu \phi_{\ell,\omega}$. Hence,
    \begin{equation}\label{Proof_prop_monotonicity_nplus1_ineq_norm_eigenequation}
        -k_\ell^2(\omega) \norm{\sqrt{\mu} \phi_{\ell,\omega}}_2^2 = \norm{ \sqrt{\mu} \phi_{\ell,\omega}' }_2^2 - \omega^2 \norm{ \sqrt{\rho} \phi_{\ell,\omega} }_2^2 > - \omega^2 \norm{ \sqrt{\rho} \phi_{\ell,\omega} }_2^2\,,
    \end{equation}
    with a strict inequality as otherwise we would have $\mu_- \normNS{ \phi_{\ell,\omega}' }_2^2 \leq \normNS{ \sqrt{\mu} \phi_{\ell,\omega}' }_2^2 = 0$ hence $\phi_{\ell,\omega}$ constant, since $\mu_->0$, and the conditions on $\phi_{\ell,\omega}$ in~\eqref{Problem_equations_new} would yield $\phi_{\ell,\omega}=0$, a contradiction to its definition.
    
    By Proposition~\ref{Prop_regularity_nPlus1_layers}, we can differentiate both sides of the equality in~\eqref{Proof_prop_monotonicity_nplus1_ineq_norm_eigenequation} w.r.t.~$\omega$:
    \[
        -2 k_\ell(\omega) \partial_\omega k_\ell(\omega) \norm{\sqrt{\mu} \phi_{\ell,\omega}}_2^2 = - 2 \omega \norm{ \sqrt{\rho} \phi_{\ell,\omega} }_2^2\,,
    \]
    where we used again the eigenvalue equation to cancel terms.
    Consequently, and since $k_\ell(\omega) \normNS{\sqrt{\mu} \phi_{\ell,\omega}}_2^2>0$ and $\omega \normNS{ \sqrt{\rho} \phi_{\ell,\omega} }_2^2>0$,
    \[
        \partial_\omega k_\ell(\omega) = \frac{ \omega \norm{ \sqrt{\rho} \phi_{\ell,\omega} }_2^2 }{ k_\ell(\omega) \norm{\sqrt{\mu} \phi_{\ell,\omega}}_2^2 }>0\,.
    \]
    
    We now compute the derivative of $k_\ell(\omega)/\omega$ and obtain
    \[
        \partial_\omega \left(\frac{k_\ell(\omega)}{\omega}\right) = \frac{k_\ell(\omega)}{\omega^2} \left[ \frac{\omega \partial_\omega k_\ell(\omega)}{k_\ell(\omega)} - 1 \right] = \frac{k_\ell(\omega)}{\omega^2} \left[ \frac{ \omega^2 \norm{ \sqrt{\rho} \phi_{\ell,\omega} }_2^2 }{ k_\ell^2(\omega) \norm{\sqrt{\mu} \phi_{\ell,\omega}}_2^2 } - 1 \right] > 0
    \]
    by the inequality in~\eqref{Proof_prop_monotonicity_nplus1_ineq_norm_eigenequation}. This proves that, for any $\ell$, $\omega \mapsto k_\ell(\omega)/\omega$ is strictly increasing on any (open) interval where $k_\ell$ exists.
    
    \textbf{Step 2.}
    Consider an open interval on which $k_\ell$ exists by Proposition~\ref{Prop_regularity_nPlus1_layers}, and let $(\omega_\ell^-,\omega_\ell^+)\subset(0,+\infty)$ be its largest (open) superset on which $k_\ell$ exists.
    Still by Proposition~\ref{Prop_regularity_nPlus1_layers}, $k_\ell$ is continuous on $(\omega_\ell^-,\omega_\ell^+)$, hence $k_\ell(\omega)/\omega$ too, and, by Step 1, $k_\ell(\omega)/\omega$ is strictly increasing on~$(\omega_\ell^-,\omega_\ell^+)$. Moreover, being bounded from below and from above ---with values in $[1/C_\infty, 1/C_0)$, see Lemma~\ref{Alt_Def_kell_nPlus1_layers}---,
    it admits limits $L_\ell^-=\lim_{\omega\searrow\omega_\ell^-}k_\ell(\omega)/\omega$ and $L_\ell^+=\lim_{\omega\nearrow\omega_\ell^+}k_\ell(\omega)/\omega$.
    We now prove that $L_\ell^-=1/C_\infty$, $L_\ell^+=1/C_0$, and $\omega_\ell^+=+\infty$.
    
    First, suppose $\omega_\ell^+<+\infty$. By continuity of $\tilde{f}_n$, we would have
    \[
        \tilde{f}_n(\omega_\ell^+, L_\ell^+) = \tilde{f}_n\!\left( \lim_{\omega\nearrow\omega_\ell^+} (\omega,k_\ell(\omega)/\omega) \right) = \lim_{\omega\nearrow\omega_\ell^+} \tilde{f}_n (\omega, k_\ell(\omega)/\omega) = 0\,.
    \]
    If $L_\ell^+=1/C_0$, this contradicts the fact that $\tilde{f}_n$ has no zeros of the form $(\omega, \omega/C_0)$, and if $L_\ell^+<1/C_0$, this contradicts that $(\omega_\ell^-,\omega_\ell^+)$ is the largest open interval on which $k_\ell$ exists since it would actually exist on $(\omega_\ell^-,\omega_\ell^+]$ and, by Proposition~\ref{Prop_regularity_nPlus1_layers}, also on an open superset of $(\omega_\ell^-,\omega_\ell^+]$. Thus, we proved that $\omega_\ell^+=+\infty$.
    
    Second, suppose $L_\ell^->1/C_\infty$. If $\omega_\ell^->0$, then the continuity of $\tilde{f}_n$ gives similarly the contradiction that $k_\ell$ exists at $\omega_\ell^-$. If $\omega_\ell^-=0$, for any $y>1/C_\infty$, we notice that
    \[
        \tilde{f}_n(\omega, y)\underset{\omega\searrow0}{\sim} \mu_\infty \bar{\nu}_\infty(y) + O(\omega^2)\,,
    \]
    which is positive for $\omega$ small enough, contradicting that $k_\ell$ exists on $(\omega_\ell^-, \omega_\ell^- + \epsilon)=(0, \epsilon)$. Thus, we proved that $L_\ell^-=1/C_\infty$. A by product, as a direct consequence of $L_\ell^-=1/C_\infty$ and the confinuity of $\tilde{f}_n$, is that there exists $\omega_\ell=\omega_\ell^-\geq0$ such that $k_\ell(\omega_\ell)=\omega_\ell/C_\infty$ if $\omega_\ell>0$ and $k_\ell(\omega)\to0$ when $\omega\searrow0$ if $\omega_\ell=0$. Note that we do not have equality in the latter case only because the $k_\ell$'s have been defined only for $\omega>0$.
    
    Third, suppose that there exists a $p\geq1$ such that $L_p^+<1/C_0$, and consider the smallest of such $p$'s. Then, $L_\ell^+\leq L_p^+<1/C_0$ for all $\ell\geq p$, since the branches $k_\ell$ do not cross by Proposition~\ref{Prop_regularity_nPlus1_layers} and are continuous, and $L_\ell^+=1/C_0$ for $1\leq\ell<p$. Thus, on the one hand for $\ell\geq p$ the branches $k_\ell$ lie in $(1/C_\infty,L_p^+)$, while on the other hand for $\ell\geq p$ there exists $L\in(L_p^+,1/C_0)$ ---one can choose $L=(\max\{1/\widetilde{C}_2,L_p^+\}+ 1/C_0)/2$ for example, so that $L>1/\widetilde{C}_2$ as it will be needed--- such that for $\omega$ large enough the branches $k_\ell$ lie in $(L,1/C_0)$. This implies that for $\omega$ large enough none of the branches $k_\ell$, $\ell\geq1$, lies in $[L_p^+, L]$. That is, $\tilde{f}_n$ has no zeros $(\omega, y)$ with $y\in[L_p^+, L]$ and where $[L_p^+, L]\cap [1/\widetilde{C}_2,1/C_0)\neq\emptyset$, contradicting Proposition~\ref{Prop_infinitely_many_zeros_close_to_1overC0} stated below, and consequently proving that $L_\ell^+=1/C_0$ for all $\ell\geq1$.
    This concludes the proof of Theorem~\ref{Thm_monotonicity_nPlus1_layers}.
\end{proof}

\begin{proposition}\label{Prop_infinitely_many_zeros_close_to_1overC0}
    Let $n\in\N\setminus\{0\}$ and  $\tilde{f}_n$ be as in Definition~\ref{Def_f_tilde_P_tilde_Q_tilde_nPlus1_layers}. Assume moreover $C_0<\widetilde{C}_2$ if $n\geq3$. Then, for $y\in[1/\widetilde{C}_2, 1/C_0)$, the function $\omega\mapsto \tilde{f}_n(\omega,y)$ admits a sequence of zeros diverging to infinity.
\end{proposition}
The cases $n=1,2$ of this porposition will be proved later ---see Section~\ref{Section_asymptotics_N_for_n_equal_1} and Proposition~\ref{Prop_zeros_ftilde_any_y_2Plus1_layers}--- for all $y\in[1/C_\infty, 1/C_0)$. The proof for $n\geq3$ is a direct consequence of the following lemma, which also yields the result of Proposition~\ref{Prop_asymptotics_N} for $n\geq3$.
\begin{lemma}\label{Lemma_zeros_ftilde_any_y_nPlus1_layers}
    Let $n\geq3$. Assume $C_0 < \widetilde{C}_2$ and fix $y\in[1/\widetilde{C}_2, 1/C_0)$.
    \begin{enumerate}[label=(\roman*)]
        \item If $C_0 = C_1$, then on each interval
        \[
            \left[ \frac{p\pi}{|\bar{\nu}_1(y)| T_1}, \frac{(p+1)\pi}{|\bar{\nu}_1(y)| T_1} \right), \qquad p\geq0\,,
        \]
        $\omega\mapsto \tilde{f}_n(\omega,y)$ admits exactly one zero, which belongs to $\left[ \frac{p\pi}{|\bar{\nu}_1(y)| T_1}, \frac{(p+1/2)\pi}{|\bar{\nu}_1(y)| T_1} \right)$.
        
        \item Otherwise, i.e., if $C_0 = C_k$ for $k\in\llbracket 2, n \rrbracket$, then on each interval
        \[
            \left[ \frac{p\pi}{|\bar{\nu}_k(y)| T_k}, \frac{(p+1)\pi}{|\bar{\nu}_k(y)| T_k} \right), \qquad p\geq0\,,
        \]
        $\omega\mapsto \tilde{f}_n(\omega,y)$ admits, at least for large $p$'s, exactly one zero. These zeros \emph{generically} belong either all to $\left[ \frac{p\pi}{|\bar{\nu}_k(y)| T_k}, \frac{(p+1/2)\pi}{|\bar{\nu}_k(y)| T_k} \right)$ or all to $\left( \frac{(p+1/2)\pi}{|\bar{\nu}_k(y)| T_k}, \frac{(p+1)\pi}{|\bar{\nu}_k(y)| T_k} \right)$.
    \end{enumerate}
\end{lemma}
\begin{proof}
    We define $\bar{M}_m$ by $\bar{M}_m(0,y)=I_2$ for $y\in[1/C_\infty, 1/C_0)$ and, for any $\omega>0$, by
    \begin{equation}\label{Def_M_bar}
        \bar{M}_m(\omega,y):=
        \left\{
        \begin{aligned}
            &\begin{pmatrix}
                \cosh[\omega\bar{\nu}_m(y) T_m] & \frac{\sinh[\omega\bar{\nu}_m(y) T_m]}{\mu_m \bar{\nu}_m(y)} \\
                \mu_m \bar{\nu}_m(y) \sinh[\omega\bar{\nu}_m(y) T_m] & \cosh[\omega\bar{\nu}_m(y) T_m]
            \end{pmatrix} &\text{if } y\neq\frac{1}{C_m}\,, \\
            &\begin{pmatrix}
                1 & \frac{T_m}{\mu_m} \\
                0 & 1
            \end{pmatrix} &\text{if } y=\frac{1}{C_m}\,,
        \end{aligned}
        \right.
    \end{equation}
    for which we have on $(\omega,y)\in[0,+\infty)\times[1/C_\infty, 1/C_0)$ the identity
    \[
        \begin{pmatrix} \tilde{P}_m(\omega,y) \\ \tilde{Q}_m(\omega,y) \end{pmatrix} =
        \bar{M}_m(\omega,y)
        \begin{pmatrix} \tilde{P}_{m-1}(\omega,y) \\ \tilde{Q}_{m-1}(\omega,y) \end{pmatrix}, \qquad \forall\, m\in\llbracket1,n\rrbracket\,.
    \]

    For~\emph{(i)}, we have
    \[
        \begin{pmatrix} \tilde{P}_n \\ \tilde{Q}_n \end{pmatrix} =
        \bar{M}_n \cdots   \bar{M}_2
        \begin{pmatrix} \tilde{P}_1 \\ \tilde{Q}_1 \end{pmatrix} =
        \bar{M}_n \cdots   \bar{M}_2
        \begin{pmatrix} \cos[\omega |\bar{\nu}_1| T_1] \\ -\mu_1 |\bar{\nu}_1| \sin[\omega |\bar{\nu}_1| T_1] \end{pmatrix} ,
    \]
    with, since $C_0=:C_1 < \widetilde{C}_2 \leq C_\infty$ and $y\in[1/\widetilde{C}_2, 1/C_0)$, all the $\bar{M}_j$ for $j\in\llbracket2,n\rrbracket$ having nonnegative coefficients with positive diagonal coefficients (bounded below by $1$) on $[0,+\infty)\times[1/\widetilde{C}_2, 1/C_0)$, hence $\bar{M}_n \cdots \bar{M}_2$ too and we denote $m_{11},m_{22}\geq1$ and $m_{12},m_{21}\geq0$ the coefficients of $\bar{M}_n \cdots \bar{M}_2$.
    Moreover, $\nu_\infty(y)\geq0$.
    Thus,
    \[
        \tilde{f}_n = (\mu_\infty \nu_\infty m_{11} + m_{21}) \cos[\omega |\bar{\nu}_1| T_1] - (\mu_\infty \nu_\infty m_{12} + m_{22}) \mu_1 |\bar{\nu}_1| \sin[\omega |\bar{\nu}_1| T_1] ,
    \]
    where the coefficients of $\cos[\omega |\bar{\nu}_1| T_1]$ and of $\sin[\omega |\bar{\nu}_1| T_1]$ are respectively nonnegative and negative on $[0,+\infty)\times[1/\widetilde{C}_2, 1/C_0)$. Hence, $\tilde{f}_n(\cdot, y)$ has exactly one zero
    \[
        \omega_n(y) \in \left[ \frac{n\pi}{|\bar{\nu}_1(y)| T_1}, \frac{\left(n+\frac{1}{2}\right)\pi}{|\bar{\nu}_1(y)| T_1} \right) 
    \]
    in every $\left[ n\pi/(|\bar{\nu}_1(y)| T_1), (n+1)\pi/(|\bar{\nu}_1(y)| T_1) \right)$.
    
    In the case~\emph{(ii)}, we write $\tilde{f}_n$ as the scalar product
    \[
        \tilde{f}_n= \begin{pmatrix} \tilde{P}_n \\ \tilde{Q}_n \end{pmatrix} \cdot \begin{pmatrix} \mu_\infty \bar{\nu}_\infty \\ 1 \end{pmatrix}.
    \]
    With $k$ the unique integer in $\llbracket 2, n \rrbracket$ such that $C_k=C_0<\widetilde{C}_2 \leq C_\infty$, we have
    \[
        \tilde{f}_n = \bar{M}_k  \bar{M}_{k-1}  \cdots   \bar{M}_1 \begin{pmatrix} 1 \\ 0 \end{pmatrix} \cdot (\bar{M}_n)\transp \cdots  (\bar{M}_{k+1})\transp \begin{pmatrix} \mu_\infty \bar{\nu}_\infty \\ 1 \end{pmatrix}
    \]
    and we define the coefficients $m_{ij}$ and $\tilde{m}_{ij}$ through
    \[
        \bar{M}_{k-1}  \cdots \bar{M}_1 =
        \left( \prod\limits_{j=1}^{k-1} \cosh[\nu_j T_j] \right)
        \begin{pmatrix}
            m_{11} & m_{12} \\
            m_{21} & m_{22}
        \end{pmatrix}
    \]
    and
    \[
        (\bar{M}_n)\transp \cdots (\bar{M}_{k+1})\transp =
        \left( \prod\limits_{j=k+1}^n \cosh[\nu_j T_j] \right)
        \begin{pmatrix}
            \tilde{m}_{11} & \tilde{m}_{12} \\
            \tilde{m}_{21} & \tilde{m}_{22}
        \end{pmatrix}.
    \]
    A straightforward induction gives that the $m_{ij}$'s and the $\tilde{m}_{ij}$'s are multivariate polynomials in the variables $\tanh[\omega \bar{\nu}_j T_j]$, $j\neq k$, of order $k-1$ and $n-k$ (the numbers of matrices multiplied), and with $\bar{\nu}_j\geq0$ for $j\in\llbracket1,n+1\rrbracket\setminus\{k\}$.
    A first key remark, that can be established by induction, is that the coefficients of the polynomials are independent of $\omega$ and nonnegative.
    We therefore define, for clarity and shortness in the forthcoming computations, $\alpha$, $\beta$, $\gamma$, and $\delta$ as the multivariate polynomials in the variables $\tanh[\omega \bar{\nu}_j T_j]$, with coefficients independent of $\omega$, through
    \[
        \begin{pmatrix}\alpha \\ \beta \end{pmatrix} := \left( \prod\limits_{j=1}^{k-1} \cosh[\nu_j T_j] \right)^{-1}
        \bar{M}_{k-1} \cdots  \bar{M}_1 \begin{pmatrix} 1 \\ 0 \end{pmatrix} =
        \begin{pmatrix}
            m_{11} \\
            m_{21}
        \end{pmatrix}
    \]
    and
    \[
        \begin{pmatrix} \gamma \\ \delta \end{pmatrix} := \left( \prod\limits_{j=k+1}^n \cosh[\nu_j T_j] \right)^{-1}
        (\bar{M}_n)\transp   \cdots  (\bar{M}_{k+1})\transp \begin{pmatrix} \mu_\infty \bar{\nu}_\infty \\ 1 \end{pmatrix} = \begin{pmatrix} \mu_\infty \bar{\nu}_\infty \tilde{m}_{11} + \tilde{m}_{12} \\ \mu_\infty \bar{\nu}_\infty \tilde{m}_{21} + \tilde{m}_{22} \end{pmatrix}.
    \]
    Consequently,
    \begin{align*}
        \tilde{f}_n = 0 \quad
        \Leftrightarrow \quad & \frac{ \tilde{f}_n }{ \prod\limits_{j\neq k}\cosh[\nu_j T_j] } = 0 \\
        \Leftrightarrow \quad &\begin{pmatrix}
            \cos[\omega |\bar{\nu}_k| T_k] & \sin[\omega |\bar{\nu}_k| T_k] / (\mu_k |\bar{\nu}_k|) \\
            -\mu_k |\bar{\nu}_k| \sin[\omega |\bar{\nu}_k| T_k] & \cos[\omega |\bar{\nu}_k| T_k]
        \end{pmatrix}
        \begin{pmatrix}\alpha \\ \beta \end{pmatrix}
        \cdot \begin{pmatrix} \gamma \\ \delta \end{pmatrix} = 0 \\
        \Leftrightarrow \quad &(\alpha \gamma + \beta \delta) \cos[\omega |\bar{\nu}_k| T_k] + \left( \beta \gamma - \alpha \delta \mu_k^2 |\bar{\nu}_k|^2 \right) \frac{\sin[\omega |\bar{\nu}_k| T_k]}{\mu_k |\bar{\nu}_k|} = 0\,.
    \end{align*}
    Now, a second key remark, that can also be established by induction, is that the polynomials $\tilde{m}_{ii}$ and $\tilde{m}_{ii}$ on the diagonal are greater or equal to~$1$, and the ones on the anti-diagonal are nonnegative with $m_{21}$ (respectively $\tilde{m}_{12}$) having at least one positive coefficient if $y>1/C_j$ for one of the $j$'s in $\llbracket 1, k-1 \rrbracket$ (resp.\ in $\llbracket k+1, n \rrbracket$). This means on one hand that $\alpha=m_{11}\geq1$ and $\delta \geq \tilde{m}_{22} \geq1$, and another hand that either at least one of $\beta=m_{21}$ and $\gamma \geq \tilde{m}_{12}$ has a positive coefficient, or $y=1/\widetilde{C}_2=1/C_j$ for all $j \in \llbracket 1, n+1 \rrbracket \setminus\{k\}$.

    In the latter case, we actually have
    \[
        \tilde{f}_n=0 \quad
        \Leftrightarrow \quad - \mu_k |\bar{\nu}_k| \sin[\omega |\bar{\nu}_k| T_k] = 0\,,
    \]
    and the zeros are exactly the $p\pi / (|\bar{\nu}_k(y)| T_k)$, $p\geq0$.

    In the former case, since $\alpha$, $\beta$, $\gamma$, and $\delta$ are multivariate polynomials in the variables $\tanh[\omega \bar{\nu}_j T_j]$ with nonnegative coefficient, they are nondecreasing in $\omega$ and converge respectively to constants $\bar\alpha$, $\bar\beta$, $\bar\gamma$,
    and $\bar\delta$ when $\omega\to+\infty$ with $\bar\alpha,\bar\delta\geq1$ and, due to the key properties aforementioned about a positive coefficient, $\bar\beta,\bar\gamma > 0$ since $\beta = m_{12}>0$ and $\gamma\geq \tilde{m}_{12}$ are nondecrasing. Therefore, $\cos[\omega |\bar{\nu}_k| T_k]$ has a positive prefactor and the one of $\sin[\omega |\bar{\nu}_k| T_k]$ converges to $\bar\gamma \bar\beta/(\mu_k |\bar{\nu}_k|) - \bar\delta \bar\alpha \mu_k |\bar{\nu}_k|$. This limit is \emph{generically} nonzero, thence the coefficient of $\sin[\omega |\bar{\nu}_k| T_k]$ has a sign at least for $\omega$ large enough.
    \begin{remark*}
        Note that in the non-generic case $\bar\gamma(y) \bar\beta(y) = \mu_k^2 |\bar{\nu}_k(y)|^2 \bar\delta(y) \bar\alpha(y)$, the sign of the coefficient of $\sin[\omega |\bar{\nu}_k| T_k]$ is very likely constant for $\omega$ large enough (because the $\omega\to\tanh[\omega \bar{\nu}_j T_j]$ are increasing (and concave) hence what probably matters is the coefficient in front of $\tanh[\omega \bar{\nu}_j T_j]$ for the largest $\bar{\nu}_j T_j$. However, we are not able to prove it.
    \end{remark*}
    On the one hand, if $y$ is s.t.\ this limit is negative ---$\bar\gamma(y) \bar\beta(y) < \mu_k^2 |\bar{\nu}_k(y)|^2 \bar\delta(y) \bar\alpha(y)$---, then for $\omega$ large enough $\omega\mapsto\tilde{f}_n(\omega, y)$ has in every
    \[
        \left[ p\pi/(|\bar{\nu}_k(y)| T_k), (p+1)\pi/(|\bar{\nu}_k(y)| T_k) \right)
    \]
    exactly one zero $\omega_p\in \left( p\pi/(|\bar{\nu}_k(y)| T_k), (p+1/2)\pi/(|\bar{\nu}_k(y)| T_k) \right)$. On the other hand, if this limit is positive, then for $\omega$ large enough $\omega\mapsto\tilde{f}_n(\omega, y)$ has in every
    \[
        \left[ p\pi/(|\bar{\nu}_k(y)| T_k), (p+1)\pi/(|\bar{\nu}_k(y)| T_k) \right)
    \]
    exactly one zero $\omega_p\in \left( (p+1/2)\pi/(|\bar{\nu}_k(y)| T_k), (p+1)\pi/(|\bar{\nu}_k(y)| T_k) \right)$.

    This concludes the proof of Lemma~\ref{Lemma_zeros_ftilde_any_y_nPlus1_layers}.
\end{proof}

\section{The simple square well: direct computations}\label{Section_1plus1_layers}
In this section, we specify to the simple square well ($n=1$), i.e., the ``$1+1$~layers'' case:
\[
    (\mu(z), \rho(z)) =
    \left\{
        \begin{aligned}
            &(\mu_1, \rho_1)\,, \quad &&\textrm{if }\!\! \hphantom{H} 0 \leq z < H\,, \\
            &(\mu_2, \rho_2)\,, \quad &&\textrm{if }\!\! \hphantom{0} H \leq z < +\infty\,,
        \end{aligned}
    \right.
\]
where $\mu_j,\rho_j>0$, $j=1,2$, with $C_\infty = C_2 = \sqrt{\mu_2/\rho_2} > \sqrt{\mu_1/\rho_1} = C_1 = C_0$.

Our goal is to retrieve the values of $H, C_1, C_2$, and $\rho_2$ from the knowledge of the frequency--wavenumber couples of the Love waves and from the knowledge of $\rho_1$.

Applying Proposition~\ref{Prop_recursive_formula_Dn} to $n=1$ (or by direct computations), $f_1$ defined in~\eqref{Def_f_n} reads
\begin{equation}\label{Def_f_1}
    f_1(\omega, k) = \mu_2 \nu_2(\omega, k) \cosh[\nu_1(\omega, k) H ] + \mu_1 \nu_1(\omega, k) \sinh[ \nu_1(\omega, k) H ]\,.
\end{equation}

By Lemma~\ref{Alt_Def_kell_nPlus1_layers}, $\nu_1=\nu_0$ and $\nu_2=\nu_\infty$ defined in~\eqref{Def_nu} satisfy $\nu_1(\omega, k)\in i\R_-\setminus\{0\}$ and~$\nu_2(\omega, k)\geq0$ if~$(\omega, k)$ is a zero of $f_1$. Hence, for such $(\omega,k)$ ---thus $\omega>0$---, we have
\begin{equation}\label{Def_f_1_at_its_zeros}
    f_1(\omega, k) = \omega \left( \mu_2 \bar\nu_2(k/\omega) \cos[\omega |\bar\nu_1(k/\omega)| H ] - \mu_1 |\bar\nu_1(k/\omega)| \sin[ \omega |\bar\nu_1(k/\omega)| H ] \right).
\end{equation}

\subsection{Study of the wavenumbers}
We can give another form to the characterization of the $k_\ell$'s in Definition~\ref{Def_kell_nPlus1_layers}
(and Lemma~\ref{Alt_Def_kell_nPlus1_layers}): for $\omega$ fixed, the $k_\ell$'s are the $k$'s solution to
\[
    \tan \left[ H \omega \sqrt{ C_1^{-2} - k^2 \omega^{-2}} \right] = \frac{\mu_2}{\mu_1} \sqrt{ \frac{ k^2 \omega^{-2} - C_2^{-2} }{ C_1^{-2} - k^2 \omega^{-2} } } \quad \text{ and } \quad \frac{\omega}{C_2} \leq k < \frac{\omega}{C_1},
\]
where the l.h.s.\ of the equation makes sense since $f_1(\omega, k)=0$ implies $\cos[ |\nu_1(\omega, k)| H ]\neq0$.

Note that the above dispersion relation can be found for instance in~\cite[(9.10)]{Haskell-53}.

On one hand the function $\psi_1(x) := \sqrt{\frac{x - C_2^{-2}}{C_1^{-2} - x}}$ is continuous strictly increasing from $0$ to $+\infty$ on $[C_2^{-2}, C_1^{-2})$.
On another hand the function $\psi_2(x) := \tan \left[ H \omega \sqrt{C_1^{-2} - x} \right]$ satisfies (for $H\omega>0$) the following properties, where $n := \left\lfloor\frac{\omega H}{\pi} \frac{\sqrt{C_2^2-C_1^2}}{C_1 C_2} - \frac{3}{2} \right\rfloor + 1 \geq 1$.

If $C_2^{-2} < C_1^{-2} - \left(\frac{3}{2H} \frac{\pi}{\omega}\right)^2 \Leftrightarrow \omega > \frac32 \frac{\pi}{H} \frac{C_1 C_2}{\sqrt{C_2^2 - C_1^2}}$, then it is continuous and strictly decreasing
\begin{itemize}[leftmargin=5mm]
    \item from $\tan \left[ \omega H \frac{\sqrt{C_2^2 - C_1^2}}{C_1 C_2} \right]$ to $-\infty$ on $\left[ C_2^{-2}, C_1^{-2} - \left(\frac{2n+1}{2H\omega} \pi \right)^2 \right)$,
    \item from $+\infty$ to $-\infty$ on the intervals $\left( C_1^{-2} - \left(\frac{2\ell-1}{2H\omega} \pi \right)^2, C_1^{-2} - \left(\frac{2\ell-3}{2H\omega} \pi \right)^2 \right)$, $\ell\in\llbracket 2, n+1 \rrbracket$,
    \item from $+\infty$ to $0$ on $\left( C_1^{-2} - \left(\frac{1}{2H} \frac{\pi}{\omega}\right)^2, C_1^{-2} \right]$.
\end{itemize}
\begin{figure}[H]
    \centering
    \includegraphics[scale=0.7]{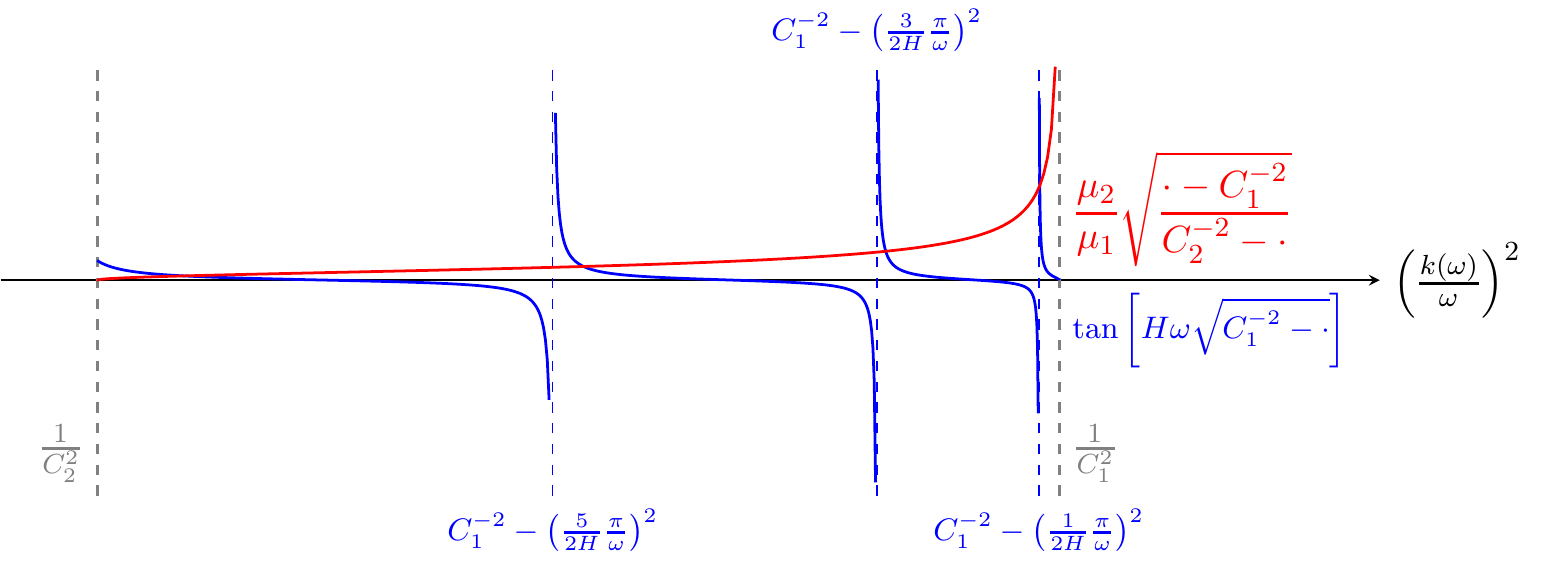}
    \caption{Functions $\psi_1$ (red) and $\psi_2$ (blue) when $C_2^{-2} < C_1^{-2} - \left(\frac{3}{2H} \frac{\pi}{\omega}\right)^2$.}
    \label{fig:12.5}
\end{figure}
\begin{figure}[H]
    \centering
    \includegraphics[scale=0.7]{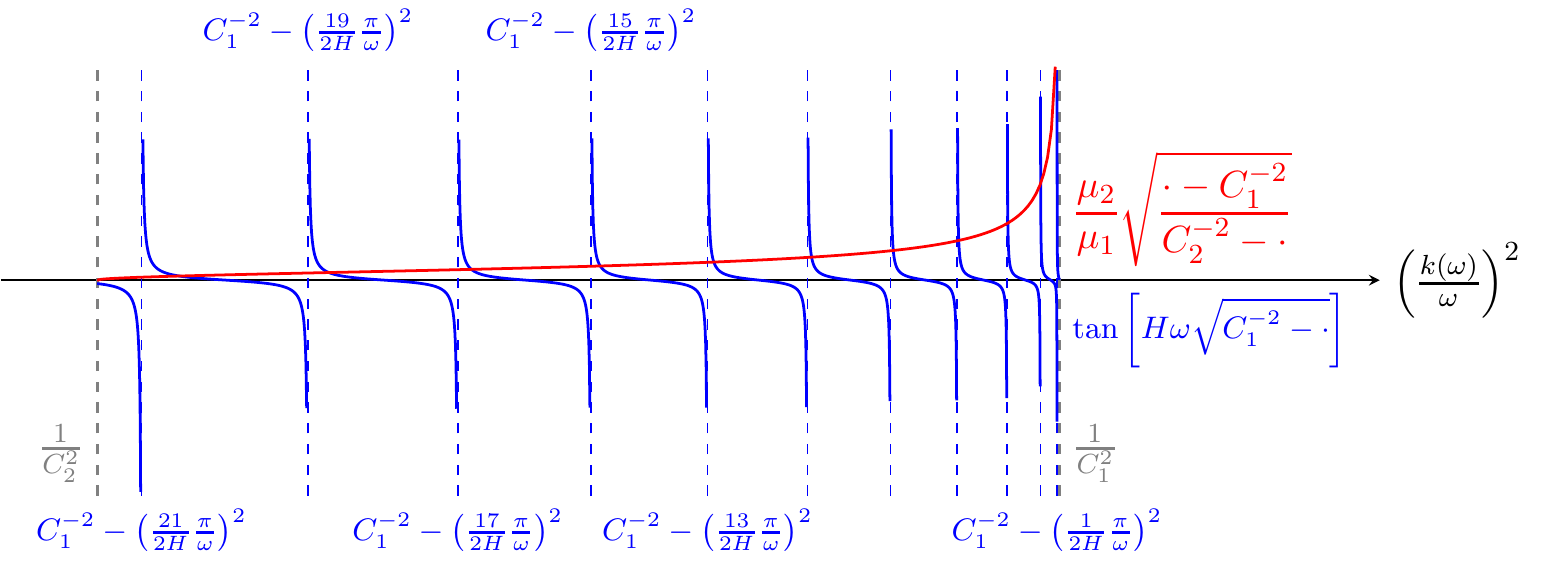}
    \caption{Functions $\psi_1$ (red) and $\psi_2$ (blue) when  $C_2^{-2} < C_1^{-2} - \left(\frac{3}{2H} \frac{\pi}{\omega}\right)^2$, but for an $\omega$ larger than in Figure~\ref{fig:12.5}.}
    \label{fig:39}
\end{figure}

If $C_1^{-2} - \left(\frac{3}{2H} \frac{\pi}{\omega}\right)^2 < C_2^{-2} < C_1^{-2} - \left(\frac{1}{2H} \frac{\pi}{\omega}\right)^2$, then it is continuous and strictly decreasing
\begin{itemize}[label=$\star$, leftmargin=2mm]
    \item from $\tan \left[ \omega H \frac{\sqrt{C_2^2 - C_1^2}}{C_1 C_2} \right]$ to $-\infty$ on $\left[ C_2^{-2}, C_1^{-2} - \left(\frac{1}{2H} \frac{\pi}{\omega}\right)^2 \right)$,
    \item from $+\infty$ to $0$ on $\left( C_1^{-2} - \left(\frac{1}{2H} \frac{\pi}{\omega}\right)^2, C_1^{-2} \right]$.
\end{itemize}
\begin{figure}[H]
    \centering
    \includegraphics[scale=0.7]{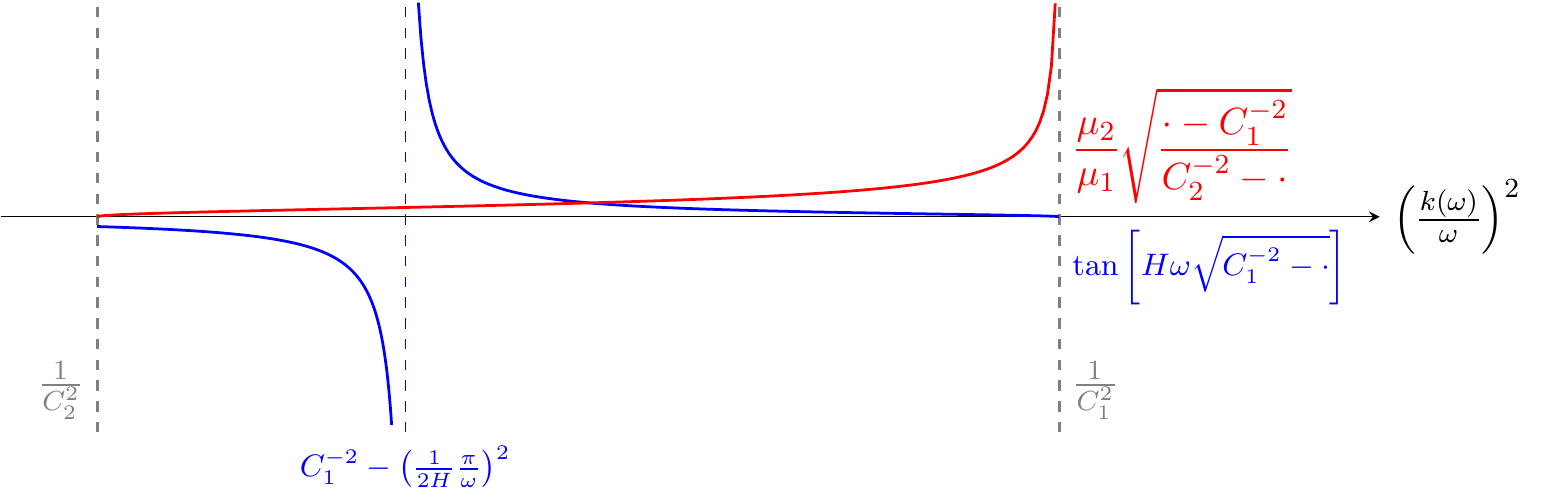}
    \caption{Functions $\psi_1$ (red) and $\psi_2$ (blue) when $C_1^{-2} - \left(\frac{3}{2H} \frac{\pi}{\omega}\right)^2 < C_2^{-2} < C_1^{-2} - \left(\frac{1}{2H} \frac{\pi}{\omega}\right)^2$.}
    \label{fig:2.2}
\end{figure}
        
If $C_1^{-2} - \left(\frac{1}{2H} \frac{\pi}{\omega}\right)^2 < C_2^{-2} \Leftrightarrow \omega < \frac{\pi}{2H} \frac{C_1 C_2}{\sqrt{C_2^2 - C_1^2}}$, then it is continuous ans strictly decreasing from $\tan \left[ \omega H \frac{\sqrt{C_2^2 - C_1^2}}{C_1 C_2} \right]>0$ to $0$ on its domain $\left[ C_2^{-2}, C_1^{-2} \right]$.
\begin{figure}[H]
    \centering
    \includegraphics[scale=0.7]{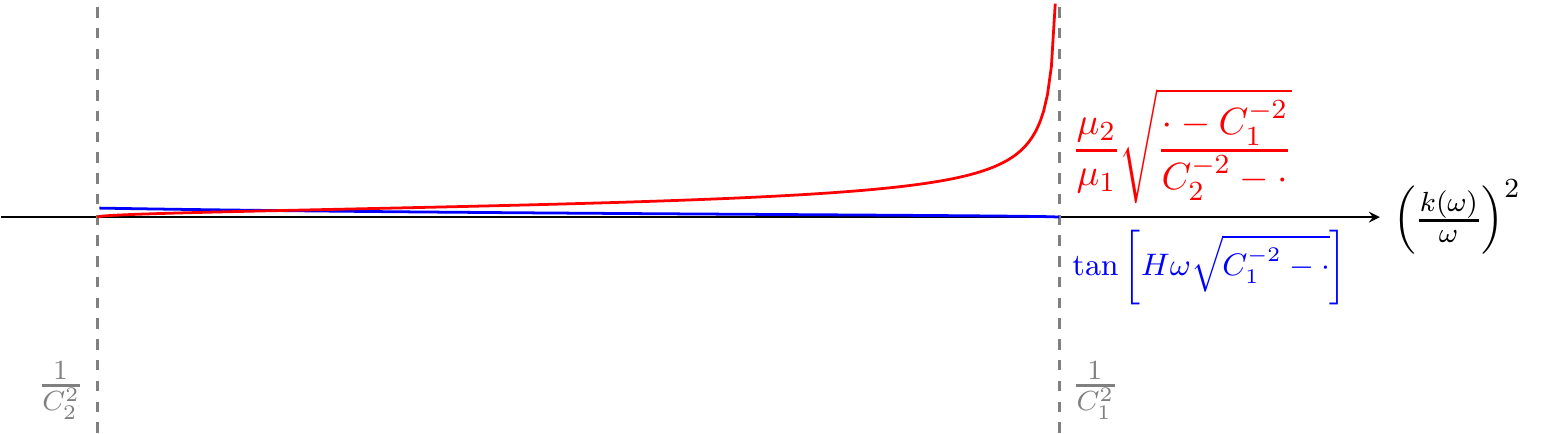}
    \caption{Functions $\psi_1$ (red) and $\psi_2$ (blue) when $C_1^{-2} - \left(\frac{1}{2H} \frac{\pi}{\omega}\right)^2 < C_2^{-2}$ ($\omega$ small).}
    \label{fig:1.4}
\end{figure}

The $k_\ell$'s are therefore implicitly defined as the points $x$'s at which the two functions intersect. For instance, in Figures~\ref{fig:1.4} and~\ref{fig:2.2} ---that is, $\omega$ small enough---, there is only one branch: $k_1$; in Figure~\ref{fig:12.5} ---that is, for a larger~$\omega$---, there are four branches: $k_1, k_2, k_3$, and $k_4$; and in Figure~\ref{fig:39} ---that is, for an even larger~$\omega$---, there are eleven branches (the intersection of the curves corresponding to $k_1$, on the right of the figure, is not visible).

\subsection{Regularity of the branches of wavenumbers}
The goal of this subsection is to prove the following.
\begin{proposition}\label{Proposition_1plus1layer}
    Let $n=1$. For any integer $\ell\geq1$, $\omega\mapsto k_\ell(\omega)/\omega$ is smooth, bijective, increasing from $(\omega_\ell,+\infty)$ to $(1/C_2, 1/C_1)$, where
    \[
        \omega_\ell := (\ell-1)\frac{C_1 C_2}{\sqrt{C_2^2 - C_1^2}} \frac{\pi}{H}\,.
    \]
\end{proposition}
The top-left simulation in Figure~\ref{Fig_simu_1_6Plus1_layers_intro} illustrates the bijectivity of $\omega\mapsto k_\ell(\omega)/\omega$.

We already know, by Proposition~\ref{Prop_regularity_nPlus1_layers}, that the function is even analytic. However, we give the proof of smoothness because it can be obtained ``by hand'' thanks to implicit function theorem (IFT), without analytic perturbation theory, and because we obtain the explicit formulae of the $\omega_\ell$'s.
\begin{proof}
    By the properties described above of the two functions $\psi_1$ and $\psi_2$ introduced earlier, and the definition of the $k_\ell$'s, we deduce all claimed results except for the values of the $\omega_\ell$'s, which is established at the end of this proof, and for the strict monoticity and the smoothness, which are now obtained by the IFT.
    
    Let us define $g:\R\times(0, \sqrt{1/C_1^2-1/C_2^2})\to\R$ by
    \[
        g(\omega,Y) := \mu_1 Y \sin[ H \omega Y ] - \mu_2 \sqrt{\frac{C_2^2-C_1^2}{C_1^2 C_2^2} - Y^2} \cos[ H \omega Y ]\,,
    \]
    which is continuous differentiable, as well as
    \[
        Y_\ell \equiv Y_\ell(\omega) := \sqrt{\frac{1}{C_1^2} - \left(\frac{k_\ell(\omega)}{\omega}\right)^2} \in \left(0, \sqrt{\frac{1}{C_1^2} - \frac{1}{C_2^2}} \right).
    \]
    \begin{remark*}
        This $g$ is nothing else than $-\tilde{f}_1$ in Definition~\ref{Def_f_tilde_P_tilde_Q_tilde_nPlus1_layers}, up to the domain and after the change of variable $Y=|\nu_1(y)|$.
    \end{remark*}
    For any $\omega_\star>\omega_\ell$, $k_\ell(\omega_\star)$ exists and we have $g(\omega_\star,Y_\ell(\omega_\star))=0$ by definition of the $k_\ell$'s.
    Moreover, defining for shortness
    \[
        s(Y):=\sqrt{\frac{C_2^2-C_1^2}{C_1^2 C_2^2} - Y^2} \in \left(0, \frac{\sqrt{C_2^2-C_1^2}}{C_1C_2}\right),
    \]
    we have
    \[
        \frac{1}{\mu_1}\frac{\di g}{\di Y}(\omega,Y) = \left[ 1 + \frac{\mu_2}{\mu_1} H \omega s(Y) \right] \sin[ H \omega Y ] + H \omega Y \left[ 1 + \frac{\mu_2}{\mu_1} \frac{1}{H \omega s(Y)} \right] \cos[ H \omega Y ]\,,
    \]
    thence $\frac{\di g}{\di Y}(\omega_\star,Y_\ell(\omega_\star))\neq0$ since $\sin[ H \omega_\star Y_\ell(\omega_\star) ]$ and $\cos[ H \omega_\star Y_\ell(\omega_\star) ]$ must have the same sign (and be non-zero) by~\eqref{Def_f_1_at_its_zeros}. Therefore, by the IFT there exists a neighborhood~$U$ of~$\{\omega_\star\}$ s.t.\ there exists a unique $\varphi\in C^1(U,\R)$ with $\varphi(\omega_\star)=Y_\ell(\omega_\star)$ and $g(\omega,\varphi(\omega))=0$ on~$U$. This $\varphi$ is (on~$U$) exactly $Y_\ell$ by definition of~$k_\ell$. Moreover, omitting for shortness the dependency of $Y_\ell$ in $\omega$, we have
    \begin{align*}
        Y_\ell'(\omega) = \varphi'(\omega) &= - \left(\frac{\di g}{\di Y_\ell}(\omega,\varphi(\omega))\right)^{-1} \frac{\di g}{\di \omega}(\omega,\varphi(\omega)) \\
        &= - H Y_\ell \frac{Y_\ell \cos[ H \omega Y_\ell ] + \frac{\mu_2}{\mu_1} s(Y_\ell) \sin[ H \omega Y_\ell ]}{\left[ 1 + \frac{\mu_2}{\mu_1} H \omega s(Y_\ell) \right] \sin[ H \omega Y_\ell ] + H \omega Y_\ell \left[ 1 + \frac{\mu_2}{\mu_1} \frac{1}{H \omega s(Y_\ell)} \right] \cos[ H \omega Y_\ell ]}
    \end{align*}
    on $U$, hence $Y_\ell'(\omega_\star) < 0$, for the same reason that the trigonometric functions share the same sign for any $(\omega,Y_\ell(\omega))$. Finally, given the definition of $Y_\ell$, we have
    \[
        \frac{k_\ell(\omega_\star)}{\omega} \partial_\omega \left(\frac{k_\ell(\omega_\star)}{\omega_\star}\right) = - Y_\ell(\omega_\star)Y_\ell'(\omega_\star) > 0\,.
    \]
    Since $k_\ell>0$ and since the above reasoning is true for any $\omega_\star>\omega_\ell$, it proves that $\omega\mapsto \left\{k_\ell(\omega)/\omega\right\}_\ell$ is a strictly increasing function (where it is defined).
    
    Moreover, since the denominator of $Y_\ell'$ does not vanish, by bootstrapping we obtain that $Y_\ell$ is smooth. Hence, from the definition of $Y_\ell$, we deduce that $\omega\mapsto \left\{k_\ell(\omega)/\omega\right\}_\ell$ is also smooth as claimed.
    
    Finally, due to the strict monotonicity, the $\omega_\ell$'s are necessarily the $\omega$'s such that $(\omega, \omega/C_2)$ is a zero of $f_1$. Since $\nu_2(\omega, \omega/C_2)=0$, the formula~\eqref{Def_f_1_at_its_zeros} of $f_1$ gives that the $\omega_\ell$'s are the (increasingly ordered) nonnegative solutions to $\omega |\bar\nu_1(1/C_2)| \sin[ \omega |\bar\nu_1(1/C_2)| H ] = 0$. That is,
    \[
        \omega_\ell = (\ell-1)\frac{C_1 C_2}{\sqrt{C_2^2 - C_1^2}} \frac{\pi}{H}\,. \qedhere
    \]
\end{proof}

\subsection{Recovering the parameters of the medium}
The properties of the two functions $\psi_1$ and $\psi_2$ introduced earlier imply, when $k_\ell(\omega)$ exists, that $k_\ell(\omega)^2/\omega^2$ belongs to
\[
    \left( C_1^{-2} - \left(\frac{2\ell-1}{2H\omega} \pi \right)^2, C_1^{-2} - \left(\frac{2\ell-3}{2H\omega} \pi \right)^2 \right) \cap \left[C_2^{-2},C_1^{-2}\right),
\]
for $\ell\geq2$, and that $k_1(\omega)^2/\omega^2 \in \left( C_1^{-2} - \left(\frac{1}{2H\omega} \pi \right)^2, C_1^{-2} \right)$. Consequently,
\begin{equation}\label{1Plus1_layers_determination_C1}
    \forall\, \ell\geq 1, \quad   C_1 = \lim\limits_{\omega\to+\infty} \frac{\omega}{k_\ell(\omega)} = \inf\limits_{\omega>0} \frac{\omega}{k_\ell(\omega)}  \,,
\end{equation}
and we (empirically) recover the value of $C_1$ from
\[
    \forall\, \ell\geq 1, \quad 1/C_1 = \sup\limits_{\omega>0} \frac{k_\ell(\omega)}{\omega} \,.
\]
Since we suppose $\rho_1$ to be known, the definition of $C_1$ gives $\mu_1$: $\mu_1 = \rho_1 C_1^2$.

Additionally, and by construction, we have
\[
    \forall\, \ell\geq 2\,, \quad \sqrt{\frac{1}{C_1^2} - \left( \frac{\ell-1}{\omega_\ell} \frac{\pi}{H} \right)^2} = \frac{1}{C_2} = \frac{k_\ell(\omega_\ell)}{\omega_\ell} = \lim\limits_{\omega\searrow\omega_1=0}\frac{k_1(\omega)}{\omega}
\]
and we consequently recover (empirically) the value of $C_2$ from
\begin{equation}\label{1Plus1_layers_determination_C2}
    \forall\, \ell\geq 1, \quad   C_2 = \lim\limits_{\omega\searrow\omega_\ell} \frac{\omega}{k_\ell(\omega)} = \sup\limits_{\omega>\omega_\ell} \frac{\omega}{k_\ell(\omega)} \,.
\end{equation}
\begin{figure}[H]
    \centering
    \includegraphics[scale=0.7]{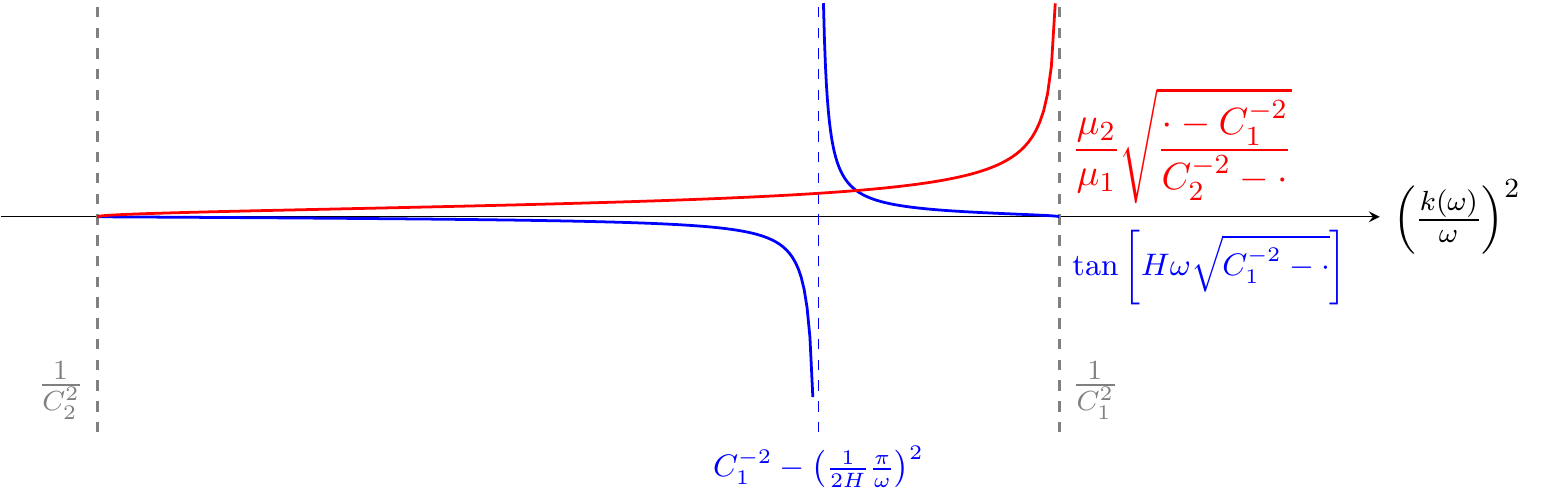}
    \caption{$\omega=\omega_2$.}
\end{figure}

Moreover, the knowledge of the maps $k_\ell(\omega)/\omega$'s allows us to (empirically) determine the $\omega_\ell$'s: they are the value below which $k_\ell$ ceases to exist.

With $C_1$, $C_2$, and the $\omega_\ell$'s recovered, the knowledge of two consecutive~$\omega_\ell$'s yields~$H$:
\begin{equation}\label{1Plus1_layers_determination_H}
     H = \frac{C_1 C_2}{\sqrt{C_2^2 - C_1^2}} \times \frac{\pi}{\omega_{\ell+1} - \omega_\ell}\,.
\end{equation}

Finally, to determine $\rho_2$ and, consequently, $\mu_2 = \rho_2 C_2^2$, we use the equation in~\eqref{nPlus1_layers_relation_defining_Love_waves}, which reads $\mu_1 |\nu_1(\omega, k)| \sin[ |\nu_1(\omega, k)| H ] = \mu_2 \nu_2(\omega, k) \cos[ |\nu_1(\omega, k)| H ]$ for $n=1$. Indeed, the knowledge of one couple~$(\omega, k_\ell(\omega))$, and of $\rho_1$, yields~$\rho_2$:
\begin{equation}\label{1Plus1_layers_determination_rho2}
     \rho_2 = \rho_1 \frac{C_1^2}{ C_2^2} \sqrt{\frac{C_1^{-2} - (k_\ell(\omega)/\omega)^2 }{ (k_\ell(\omega)/\omega)^2 - C_2^{-2} }} \tan \left[ H \omega \sqrt{ C_1^{-2} - (k_\ell(\omega)/\omega)^2} \right].
\end{equation}

\subsection{Proof of Weyl's law for the simple square well}\label{Section_asymptotics_N_for_n_equal_1}
We give here the proof of Proposition~\ref{Prop_asymptotics_N} in the case $n=1$ (but notice that we did not need it to recover the parameter of the problem). We want to prove that for any $y\in[1/C_\infty, 1/C_0)=[1/C_2, 1/C_1)$, as $\omega$ goes to $+\infty$, we have
\[
    N(\omega, y) \sim \frac{\omega}{\pi} |\tilde{\nu}_1(y)| \widetilde{T}_1 = \frac{\omega}{\pi} |\bar{\nu}_1(y)| H = \frac{\omega}{\pi} \sqrt{C_1^{-2} - y^2} H\,.
\]

By Definition~\ref{Def_number_branches_above_nPlus1_layers} of $N(\omega, y)$, we look for any fixed $y\in[1/C_2, 1/C_1)$, at the positive zeros of
\[
    \omega \mapsto \tilde{f}_1(\omega,y) = \mu_2 \bar\nu_2(y) \cos[\omega |\bar\nu_1(y)| H ] - \mu_1 |\bar\nu_1(y)| \sin[ \omega |\bar\nu_1(y)| H ] \,.
\]
Recall that $|\bar\nu_1(y)|>0$ for $y\in[1/C_2, 1/C_1)$. Note first that if $\omega$ is a zero of this function, then $\cos[ \omega |\bar\nu_1(y)| H ]\neq0$ as, otherwise, $\sin[ \omega |\bar\nu_1(y)| H ]=\pm1$ and we have the contradiction
\[
    \tilde{f}_1(\omega,y) = - \mu_1 |\bar\nu_1(y)| \sin[ \omega |\bar\nu_1(y)| H ]\neq0\,.
\]
Hence, $\tilde{f}_1(\omega,y)=0 \Leftrightarrow     \tan[ \omega |\bar\nu_1(y)| H ] = (\mu_2 \bar\nu_2(y) )/ (\mu_1 |\bar\nu_1(y)| ) \geq0$ and $\tilde{f}_1(\cdot,y)$ admits exactly one zero on each interval
\[
    \left[ \frac{p\pi}{|\bar{\nu}_1(y)| H}, \frac{(p+1)\pi}{|\bar{\nu}_1(y)| H} \right), \qquad p\geq0\,.
\]
Moreover, if $y>1/C_2$, then the $(p+1)$-th zero belongs to
\[
    \left( \frac{p\pi}{|\bar{\nu}_1(y)| H}, \frac{(p+1/2)\pi}{|\bar{\nu}_1(y)| H} \right), \qquad p\geq0\,,
\]
while if $y=1/C_2$, then the $p$-th positive zero $\omega_p$ is the $(p+1)$-th $\omega_\ell$:
\[
    \omega_{\ell=p+1} = \frac{p\pi}{|\bar{\nu}_1(y)| H} \,, \qquad p\geq1\,.
\]
The claim $N(\omega, y) \sim \omega |\bar{\nu}_1(y)| H / \pi$ for $y\in[1/C_2, 1/C_1)$ is therefore proved.

\section{The double square well}\label{Section_2Plus1_layers}
In this section, we specify to the double square well ($n=2$), i.e., the ``$2+1$~layers'' case:
\[
    (\mu(z), \rho(z)) =
    \left\{
        \begin{aligned}
            &(\mu_1, \rho_1)\,, \quad &&\textrm{if }\!\! \hphantom{H_2} 0 \leq z < H_2\,, \\
            &(\mu_2, \rho_2)\,, \quad &&\textrm{if }\!\! \hphantom{0} H_2 \leq z < H_3\,, \\
            &(\mu_3, \rho_3)\,, \quad &&\textrm{if }\!\! \hphantom{0} H_3 \leq z < +\infty\,,
        \end{aligned}
    \right.
\]
where $\mu_j,\rho_j>0$, $j=1,2,3$, with
\[
    C_\infty = C_3 = \sqrt{\mu_3/\rho_3} > \min\{\sqrt{\mu_1/\rho_1}, \sqrt{\mu_2/\rho_2}\} = \min\{C_1,C_2\} = C_0\,.
\]

Our goal is to retrieve the values of $H_2, H_3, C_1, C_2, C_3, \rho_2$ and $\rho_3$ from the knowledge of the frequency--wavenumber couples of the Love waves and from the knowledge of $\rho_1$.

Applying Proposition~\ref{Prop_recursive_formula_Dn} to $n=2$, $f_2$ defined in~\eqref{Def_f_n} reads
\begin{multline}\label{Def_f_2}
    f_2(\omega, k) := \left( \mu_3 \nu_3 \cosh[ \nu_1 T_1 ] + \mu_1 \nu_1 \sinh[ \nu_1 T_1 ] \right) \cosh[ \nu_2 T_2 ] \\
        + \left( \mu_2^2 \nu_2^2 \cosh[ \nu_1 T_1 ] + \mu_3 \nu_3 \mu_1 \nu_1 \sinh[ \nu_1 T_1 ] \right) \frac{\sinh[ \nu_2 T_2 ]}{\mu_2 \nu_2}\,,
\end{multline}
if $k\neq \omega/C_2$, and
\begin{equation}\label{Def_f_2_special_case}
    f_2(\omega, \omega/C_2) := \mu_3 \nu_3 \cosh[ \nu_1 T_1 ] + \mu_1 \nu_1 \sinh[ \nu_1 T_1 ] + \mu_3 \nu_3 \mu_1 \nu_1 \sinh[ \nu_1 T_1 ] \frac{T_2}{\mu_2}\,,
\end{equation}
where for shortness we omitted in the r.h.s.' the respective dependence in~$(\omega, k)$ and in~$(\omega, \omega/C_2)$ of the $\nu_j$'s.
We recall that $\nu_0, \nu_1, \nu_2 \in \R_+ \cup i\R_-$, and $\nu_3=\nu_\infty$ are defined in~\eqref{Def_nu}, with $\nu_3(\omega, k)>0$ for a Love wave to exists at~$(\omega, k)$.

We have the following stronger version of Proposition~\ref{Prop_infinitely_many_zeros_close_to_1overC0} for~$n=2$.
\begin{proposition}\label{Prop_zeros_ftilde_any_y_2Plus1_layers}
    Let $n=2$ and $\tilde{f}_2$ be as in Definition~\ref{Def_f_tilde_P_tilde_Q_tilde_nPlus1_layers} and fix $y\in[1/C_\infty, 1/C_0)$.
    \begin{enumerate}[label=(\roman*), leftmargin=1em]
        \item If $1/C_2 \leq y < 1/C_1$, then $\tilde{f}_2(\cdot,y)$ admits exactly one zero
        \[
            \omega_p(y) \in \left[ \frac{p\pi}{|\bar{\nu}_1(y)| T_1}, \frac{\left(p+\frac{1}{2}\right)\pi}{|\bar{\nu}_1(y)| T_1} \right)
            \text{ on each interval }
            \left[ \frac{p\pi}{|\bar{\nu}_1(y)| T_1}, \frac{(p+1)\pi}{|\bar{\nu}_1(y)| T_1} \right), \quad p\geq1\,.
        \]
        
        \item If $1/C_1 \leq y < 1/C_2$, then $\tilde{f}_2(\cdot,y)$ admits, at least for $p$ large enough, exactly one zero $\omega_p(y)$ on each interval
        \[
            \left[ \frac{p\pi}{|\bar{\nu}_2(y)| T_2}, \frac{(p+1)\pi}{|\bar{\nu}_2(y)| T_2} \right)
        \]
        with either $\omega_n(y)\in\left( p\pi/(|\bar{\nu}_2(y)| T_2), (p+1/2)\pi/(|\bar{\nu}_2(y)| T_2) \right]$ for all $p$ (large enough) or $\omega_n(y)\in\left( (p+1/2)\pi/(|\bar{\nu}_2(y)| T_2), (p+1)\pi/(|\bar{\nu}_2(y)| T_2) \right)$ for all $p$ (large enough).
        
        \item If $1/C_\infty\leq y<\min\{1/C_1, 1/C_2\}$, define
        \[
            \left\{
            \begin{aligned}
                m:={}&\min\{|\bar{\nu}_1(y)|T_1, |\bar{\nu}_2(y)|T_2\}\,, \\ M:={}&\max\{|\bar{\nu}_1(y)|T_1, |\bar{\nu}_2(y)|T_2\}\,.
            \end{aligned}
            \right.
        \]
        Then, there exists $\{\tilde\omega_p\}_{p\geq1}$ s.t.\ $\tilde\omega_{p+1}-\tilde\omega_p=\pi/m$, $\tilde\omega_1>0$, and $\tilde{f}_2(\cdot,y)$ admits on each interval $[\tilde\omega_p, \tilde\omega_{p+1})$ either $\lfloor M/m \rfloor + 1$ or $\lceil M/m \rceil + 1$ zeros.
    \end{enumerate}
\end{proposition}
Note that \emph{(i)} covers in particular the non-standard setting $C_1<C_\infty\leq C_2$, and that \emph{(iii)} covers in particular the degenerated setting $C_1=C_2< C_\infty$.
\begin{proof}
    \textbf{Case $y=1/C_2$.} Hence, necessarily $C_\infty = C_3 \geq C_2 > C_1 = C_0$ and we have
    \[
        \tilde{f}_2(\omega,y) = 0 \, \Leftrightarrow \, \mu_3 \bar{\nu}_3(y) \cos[ \omega |\bar{\nu}_1(y)| T_1 ] = \left( 1 + \omega \mu_3 \bar{\nu}_3(y) \frac{T_2}{\mu_2} \right) \mu_1 |\bar{\nu}_1(y)| \sin[ \omega |\bar{\nu}_1(y)| T_1 ]\,,
    \]
    for any $\omega>0$. Thus, if $C_2<C_\infty$, then in every $\left[ p\pi/(|\bar{\nu}_1(y)| T_1), (p+1)\pi/(|\bar{\nu}_1(y)| T_1) \right)$, $p\geq1$, $\tilde{f}_2(\cdot, y)$ has exactly one zero $\omega_p\in \left( p\pi/(|\bar{\nu}_1(y)| T_1), (p+1/2)\pi/(|\bar{\nu}_1(y)| T_1) \right)$, while if $C_2=C_\infty$, then the zeros of $\tilde{f}_2(\cdot, y)$ are the $\omega_p = p\pi/(|\bar{\nu}_1(y)| T_1)$, $p\geq1$.
    
    \medskip
    The special case $y=1/C_2$ of \emph{(i)} is therefore proved and, from now on, we assume $y\in[1/C_\infty, 1/C_0)\setminus\{1/C_2\}$. For any $\omega>0$, we have
    \begin{align*}
        \tilde{f}_2(\omega, y)
        &=\begin{multlined}[t][0.8\textwidth]
            \left( \mu_3 \bar{\nu}_3 \cosh[ \omega \bar{\nu}_1 T_1 ] + \mu_1 \bar{\nu}_1 \sinh[ \omega \bar{\nu}_1 T_1 ] \right) \cosh[ \omega \bar{\nu}_2 T_2 ] \\
            + \left( \mu_2^2 \bar{\nu}_2^2 \cosh[ \omega \bar{\nu}_1 T_1 ] + \mu_3 \bar{\nu}_3 \mu_1 \bar{\nu}_1 \sinh[ \omega \bar{\nu}_1 T_1 ] \right) \frac{\sinh[ \omega \bar{\nu}_2 T_2 ]}{\mu_2 \bar{\nu}_2}
        \end{multlined} \\
        &=\begin{multlined}[t][0.8\textwidth]
            \left( \mu_3 \bar{\nu}_3 \cosh[ \omega \bar{\nu}_2 T_2 ] + \mu_2 \bar{\nu}_2 \sinh[ \omega \bar{\nu}_2 T_2 ] \right) \cosh[ \omega \bar{\nu}_1 T_1 ] \\
            + \left( \cosh[ \omega \bar{\nu}_2 T_2 ] + \frac{\mu_3 \bar{\nu}_3 }{\mu_2 \bar{\nu}_2} \sinh[ \omega \bar{\nu}_2 T_2 ] \right) \mu_1 \bar{\nu}_1 \sinh[ \omega \bar{\nu}_1 T_1 ]\,.
        \end{multlined}
    \end{align*}
    
    We always have $\bar{\nu}_3\geq0$, but for $\bar{\nu}_1$ and $\bar{\nu}_2$, there are three distinct situations to consider: $1/C_2 < y < 1/C_1$, $1/C_1 \leq y < 1/C_2$, and $y<\min\{1/C_1, 1/C_2\}$.
    
    \medskip
    \textbf{Case $1/C_2 < y < 1/C_1$.}
    Then, $\bar{\nu}_1 = -i|\bar{\nu}_1|$, $\bar{\nu}_2>0$, and $\tilde{f}_2$ reads
    \begin{multline*}
        \tilde{f}_2(\omega, y) = \left( \mu_3 \bar{\nu}_3 \cosh[ \omega \bar{\nu}_2 T_2 ] + \mu_2 \bar{\nu}_2 \sinh[ \omega \bar{\nu}_2 T_2 ] \right) \cos[ \omega |\bar{\nu}_1| T_1 ] \\
        - \mu_1 |\bar{\nu}_1| \left( \cosh[ \omega \bar{\nu}_2 T_2 ] + \frac{\mu_3 \bar{\nu}_3}{\mu_2 \bar{\nu}_2} \sinh[ \omega \bar{\nu}_2 T_2 ] \right) \sin[ \omega |\bar{\nu}_1| T_1 ]\,,
    \end{multline*}
    where the factors of $\cos[ \omega |\bar{\nu}_1| T_1 ]$ and $\sin[ \omega |\bar{\nu}_1| T_1 ]$ are respectively positive and negative. Hence, in every $\left[ p\pi/(|\bar{\nu}_1(y)| T_1), (p+1)\pi/(|\bar{\nu}_1(y)| T_1) \right)$, $p\geq1$, $\tilde{f}_2(\cdot, y)$ has exactly one zero $\omega_p\in \left( p\pi/(|\bar{\nu}_1(y)| T_1), (p+1/2)\pi/(|\bar{\nu}_1(y)| T_1) \right)$. This concludes the proof of \emph{(i)}.
    
    \medskip
    \textbf{Case $1/C_1 \leq y < 1/C_2$} (proof of~\emph{(ii)}). Then, $\bar{\nu}_2 = -i|\bar{\nu}_2|$, and $\bar{\nu}_1\geq0$, and $\tilde{f}_2$ reads
    \begin{align*}
        \frac{\tilde{f}_2(\omega, y)}{\cosh[ \omega \bar{\nu}_1 T_1 ]} = &\left( \mu_3 \bar{\nu}_3 + \mu_1 \bar{\nu}_1 \tanh[ \omega \bar{\nu}_1 T_1 ] \right) \cos[ \omega |\bar{\nu}_2| T_2 ] \\
        &
        + \left( \mu_3 \bar{\nu}_3 \mu_1 \bar{\nu}_1 \tanh[ \omega \bar{\nu}_1 T_1 ] - \mu_2^2 |\bar{\nu}_2|^2 \right) \frac{\sin[ \omega |\bar{\nu}_2| T_2 ]}{\mu_2 |\bar{\nu}_2|}\,.
    \end{align*}
    \begin{itemize}[leftmargin=2em]
        \item If $y=1/C_1$, then $\bar{\nu}_1=0$ and the positive zeros of $\tilde{f}_2(\cdot, y)$ are the ones of $\tan[ \omega |\bar{\nu}_2| T_2 ] - \mu_3 \bar{\nu}_3/(\mu_2 |\bar{\nu}_2|)$. Hence, they are spaced exactly by $\pi/(|\bar{\nu}_2(y)| T_2)$ and belong to the intervals $\left( p\pi/(|\bar{\nu}_2(y)| T_2), (p+1/2)\pi/(|\bar{\nu}_2(y)| T_2) \right)$ if $1/C_3<1/C_1$, while they are the $\omega_p = (p+1/2)\pi/(|\bar{\nu}_2(y)| T_2)$, $p\geq1$, if $1/C_3=1/C_1$.

	\smallskip
        
        \item For $y>1/C_1$, on the one hand, if $y$ is s.t.\ $\mu_3 \bar{\nu}_3(y) \mu_1 \bar{\nu}_1(y) \leq \mu_2^2 |\bar{\nu}_2(y)|^2$ ---which allows $y=1/C_3$---, then in every $\left[ p\pi/(|\bar{\nu}_2(y)| T_2), (p+1)\pi/(|\bar{\nu}_2(y)| T_2) \right)$, $p\geq1$, $\tilde{f}_2(\cdot, y)$ has exactly one zero $\omega_p\in \left( p\pi/(|\bar{\nu}_2(y)| T_2), (p+1/2)\pi/(|\bar{\nu}_2(y)| T_2) \right]$, similarly to previously. On the other hand, if $\mu_3 \bar{\nu}_3(y) \mu_1 \bar{\nu}_1(y) > \mu_2^2 |\bar{\nu}_2(y)|^2$ ---which excludes $y=1/C_3$---, then for $\omega$ large enough, the factor of $\sin[ \omega |\bar{\nu}_2| T_2 ]$ is positive. Thus, for $p$ large enough, in every $\left[ p\pi/(|\bar{\nu}_2(y)| T_2), (p+1)\pi/(|\bar{\nu}_2(y)| T_2) \right)$ the function $\tilde{f}_2(\cdot, y)$ has exactly one zero $\omega_p\in \left( (p+1/2)\pi/(|\bar{\nu}_2(y)| T_2), (p+1)\pi/(|\bar{\nu}_2(y)| T_2) \right)$.
    \end{itemize}
    This concludes the proof of \emph{(ii)}.
    
    \medskip
    \textbf{Case $1/C_3\leq y<\min\{1/C_1, 1/C_2\}$} (proof of~\emph{(iii)}. Then, $\tilde{f}_2$ reads
    \begin{multline*}
        \tilde{f}_2(\omega, y) = \left( \mu_3 \bar{\nu}_3 \cos[ \omega |\bar{\nu}_2| T_2 ] - \mu_2 |\bar{\nu}_2| \sin[ \omega |\bar{\nu}_2| T_2 ] \right) \cos[ \omega |\bar{\nu}_1| T_1 ] \\
        - \mu_1 |\bar{\nu}_1| \left( \cos[ \omega |\bar{\nu}_2| T_2 ] + \frac{\mu_3 \bar{\nu}_3}{\mu_2 |\bar{\nu}_2|} \sin[ \omega |\bar{\nu}_2| T_2 ] \right) \sin[ \omega |\bar{\nu}_1| T_1 ]\,.
    \end{multline*}
    \begin{itemize}[leftmargin=2em]
        \item If $|\bar{\nu}_1|T_1=|\bar{\nu}_2|T_2=m=M$, it reduces to
        \[
            \tilde{f}_2(\omega, y) = \mu_3 \bar{\nu}_3 \cos^2[ \omega m ] - (\mu_1|\bar{\nu}_1|+\mu_2|\bar{\nu}_2|) \cos[ \omega m ] \sin[ \omega m ] - \mu_3 \bar{\nu}_3\frac{\mu_1|\bar{\nu}_1|}{\mu_2|\bar{\nu}_2|} \sin^2[ \omega m ]\,.
        \]
        
        If $y>1/C_3$, since the r.h.s.\ does not vanish when $\cos[ \omega m ]=0$, then the l.h.s.\ shares its zeros with
        \[
            \mu_3 \bar{\nu}_3 - (\mu_1|\bar{\nu}_1|+\mu_2|\bar{\nu}_2|) \tan[ \omega m ] - \mu_3 \bar{\nu}_3 \frac{\mu_1|\bar{\nu}_1|}{\mu_2|\bar{\nu}_2|} \tan^2[ \omega m ]\,,
        \]
        i.e., the $\omega$'s s.t.
        \[
            \tan[ \omega m ] =
            \left\{
                \begin{aligned}
                    - &\mu_2|\bar{\nu}_2| \frac{ \sqrt{ (\mu_1|\bar{\nu}_1|+\mu_2|\bar{\nu}_2|)^2 + 4 (\mu_3 \bar{\nu}_3)^2 \frac{\mu_1|\bar{\nu}_1|}{\mu_2|\bar{\nu}_2|} } + (\mu_1|\bar{\nu}_1|+\mu_2|\bar{\nu}_2|) }{ 2 \mu_3 \bar{\nu}_3 \mu_1|\bar{\nu}_1| } < 0 \\
                    \text{or}\,& \\
                    &\mu_2|\bar{\nu}_2| \frac{ \sqrt{ (\mu_1|\bar{\nu}_1|+\mu_2|\bar{\nu}_2|)^2 + 4 (\mu_3 \bar{\nu}_3)^2 \frac{\mu_1|\bar{\nu}_1|}{\mu_2|\bar{\nu}_2|} } - (\mu_1|\bar{\nu}_1|+\mu_2|\bar{\nu}_2|) }{ 2 \mu_3 \bar{\nu}_3 \mu_1|\bar{\nu}_1| } > 0\,.
                \end{aligned}
            \right.
        \]
        Hence, in every $\left[ p\pi/(2m), (p+1)\pi/(2m) \right)$, $p\geq1$, the function $\tilde{f}_2(\cdot, y)$ has exactly one zero $\omega_p$. That is, exactly  $2=\lfloor M/m \rfloor + 1=\lceil M/m \rceil + 1$ zeros in every interval $\left[ p\pi/m, (p+1)\pi/m \right)$ of length $\pi/m$.
        
        If $y=1/C_3$, then the positive zeros are the positive $\omega$'s s.t.\ $\cos[ \omega m ] \sin[ \omega m ]=0$, i.e., the elements of $\{n\pi/(2m)\}_{n\geq1}$, and the claim is also proved.

	\smallskip
            
        \item If $M=|\bar{\nu}_1|T_1>|\bar{\nu}_2|T_2=m$, then we consider the sequence $\tilde\omega_p$ of consecutive positive zeros of $\tan[ \omega m ] + \mu_2 |\bar{\nu}_2|/(\mu_3 \bar{\nu}_3)$ if $y>1/C_3$ and of $\cos[ \omega m ]$ if $y=1/C_3$. Considering on $(\tilde\omega_p, \tilde\omega_{p+1})$ the function
        \[
            \mu_1 |\bar{\nu}_1| \tan[ \omega M ] - \frac{ \mu_3 \bar{\nu}_3 - \mu_2 |\bar{\nu}_2| \tan[ \omega m ] }{ 1 + \frac{\mu_3 \bar{\nu}_3}{\mu_2 |\bar{\nu}_2|} \tan[ \omega m ] }\,,
        \]
        one can check the following.
        Either $\cos[ \tilde\omega_p M ]=0$ and the function has exactly $\lceil M/m \rceil$ zeros on $(\tilde\omega_p, \tilde\omega_{p+1})$, hence $\tilde{f}_2(\cdot,y)$ admits exactly $\lceil M/m \rceil + 1$ zeros on $[\tilde\omega_p, \tilde\omega_{p+1})$: $\tilde\omega_p$ plus these $\lceil M/m \rceil$ zeros in $(\tilde\omega_p, \tilde\omega_{p+1})$; or $\cos[ \tilde\omega_p M ]\neq0$ and $\cos[ \tilde\omega_p M ]$ has either $\lfloor M/m \rfloor$ or $\lceil M/m \rceil$ zeros in $(\tilde\omega_p, \tilde\omega_{p+1})$ hence the function has either $\lfloor M/m \rfloor + 1$ or $\lceil M/m \rceil + 1$ zeros on $(\tilde\omega_p, \tilde\omega_{p+1})$, and $\tilde{f}_2(\cdot,y)$ admits also either $\lfloor M/m \rfloor + 1$ or $\lceil M/m \rceil + 1$ zeros on $[\tilde\omega_p, \tilde\omega_{p+1})$.

	\smallskip
        
        \item If $m=|\bar{\nu}_1|T_1<|\bar{\nu}_2|T_2=M$, then we consider the sequence $\tilde\omega_p$ of consecutive positive zeros of $\tan[ \omega m ] + (\mu_2 |\bar{\nu}_2|)^2/(\mu_1 \bar{\nu}_1 \mu_3 \bar{\nu}_3)$ if $y>1/C_3$ and of $\cos[ \omega m ]$ if $y=1/C_3$. The same result is obtained working on
        \[
            \mu_2 |\bar{\nu}_2| \tan[ \omega M ] - \frac{ \mu_3 \bar{\nu}_3 - \mu_1 |\bar{\nu}_1| \tan[ \omega m ] }{ 1 + \frac{\mu_3 \bar{\nu}_3}{\mu_2 |\bar{\nu}_2|}\frac{\mu_1 \bar{\nu}_1}{\mu_2 |\bar{\nu}_2|} \tan[ \omega m ] }\,.
        \]
    \end{itemize}
    Summarizing, in the case $1/C_3\leq y<\min\{1/C_1, 1/C_2\}$, we have found intervals $[\tilde\omega_p, \tilde\omega_{p+1})$ of length $\pi/m$ partitioning $\R_+$ such that $\tilde{f}_2(\cdot,y)$ admits on each of them either $\lfloor M/m \rfloor + 1$ or $\lceil M/m \rceil + 1$ zeros.
    This concludes the proof of \emph{(iii)}.
\end{proof}

We can now prove Proposition~\ref{Prop_asymptotics_N} in the case $n=2$, that we recall in Proposition~\ref{Prop_asymptotics_N_for_n_equal_2} below. The proof is in the spirit of the one of Proposition~\ref{Prop_zeros_ftilde_any_y_2Plus1_layers} (and is actually an immediate consequence of it in its cases~\emph{(i)} and~\emph{(ii)}). This result will allow, thanks to Corollary~\ref{Cor_recover_C1_C2_2Plus1_layers} below, to recover $\widetilde{C}_2$ (and $C_0$) as well as $\widetilde{T}_1$ and $\widetilde{T}_2$.
\begin{proposition}\label{Prop_asymptotics_N_for_n_equal_2}
    Let $y\in[1/C_\infty, 1/C_0)$ be fixed.
    As $\omega$ goes to $+\infty$, we have
    \[
        \left\{
        \begin{aligned}
            &N(\omega, y) \sim \frac{\omega}{\pi} |\tilde{\nu}_1(y)| \widetilde{T}_1\,,\, &\text{if } y\in[1/\widetilde{C}_2, 1/C_0)\,, \\
            &N(\omega, y) \sim \frac{\omega}{\pi} \left( |\tilde{\nu}_1(y)| \widetilde{T}_1 + |\tilde{\nu}_2(y)| \widetilde{T}_2 \right),\, &\text{if } y\in[1/C_\infty, 1/\widetilde{C}_2)\,.
        \end{aligned}
        \right.
    \]
\end{proposition}
Notice that on the one hand if $C_1=C_2$, then $1/C_\infty< 1/\widetilde{C}_2 = 1/C_0$ and the first case is empty but not the second one, which then reads $N(\omega, y) \sim \omega |\tilde{\nu}_1(y)| ( \widetilde{T}_1 + \widetilde{T}_2 ) / \pi$, and that on the other hand if $1/C_\infty\geq 1/\widetilde{C}_2$, then $1/C_0>1/C_\infty\geq 1/\widetilde{C}_2$ and the second case is empty but not the first one.
\begin{proof}
    When $y\in[1/\widetilde{C}_2, 1/C_0)$, this is an immediate consequence of Proposition~\ref{Prop_zeros_ftilde_any_y_2Plus1_layers}\emph{(i)}--\emph{(ii)}.
    
    When $y\in[1/C_\infty, 1/\widetilde{C}_2)$, we follow the proof of Proposition~\ref{Prop_zeros_ftilde_any_y_2Plus1_layers}\emph{(iii)}.
    \begin{itemize}[leftmargin=2em]
        \item If $|\bar{\nu}_1|T_1=|\bar{\nu}_2|T_2=m=M$, the result is an immediate consequence of the proof of Proposition~\ref{Prop_zeros_ftilde_any_y_2Plus1_layers} \emph{(iii)} for that case, since the proof gives either, for $y>1/C_\infty$, the exact number of zeros ---two zeros in every interval $\left[ p\pi/m, (p+1)\pi/m \right)$ of length $\pi/m$--- or, for $y=1/C_\infty$, the exact location of all the zeros ---$\{n\pi/(2m)\}_{n\geq1}$.
            
        \item If $m=\min\{|\bar{\nu}_1|T_1, |\bar{\nu}_2|T_2 \} < \max\{|\bar{\nu}_1|T_1, |\bar{\nu}_2|T_2 \}=M$ ---here, we group together the second and third subcases of the proof of Proposition~\ref{Prop_zeros_ftilde_any_y_2Plus1_layers} \emph{(iii)}---, then the proof of Proposition~\ref{Prop_zeros_ftilde_any_y_2Plus1_layers} \emph{(iii)} shows that we have to study on $(\tilde\omega_p, \tilde\omega_{p+1})$ the function
    \[
        \tan[ \omega M ] - \frac{ \alpha - \beta \tan[ \omega m ] }{ 1 + \gamma \tan[ \omega m ] }
    \]
    either, if $y>1/C_\infty$, for $\alpha, \beta, \gamma > 0$ and where the $\tilde\omega_p$'s are the consecutive positive zeros of $1 + \gamma \tan[ \omega m ]$, or, if $y=1/C_\infty$, for $\beta>0 = \alpha = \gamma$ and where the $\tilde\omega_p$'s are the consecutive positive zeros of $\cos[ \omega m ]$. In both cases, the number of zeros in $(\tilde\omega_p, \tilde\omega_{p+1})$ of the studied function is one plus the number of zeros of $\cos[ \omega M ]$ in that interval of length~$\pi/m$. Moreover, $\tilde\omega_p$ is itself a positive zero of $\tilde{f}_2(\cdot,y)$ if and only if it is a zero of $\cos[ \omega M ]$. Thus, the average of the number of positive zeros of $\tilde{f}_2(\cdot,y)$ on the intervals $[\tilde\omega_p, \tilde\omega_{p+1})$ is $1+M/m$. Consequently, when $\omega\to+\infty$,
    \[
        N(\omega, y) \sim \frac{\omega}{\pi/m} \left( 1+\frac{M}{m} \right) = \frac{\omega}{\pi} (m + M) = \frac{\omega}{\pi} \left( |\tilde{\nu}_1(y)| \widetilde{T}_1 + |\tilde{\nu}_2(y)| \widetilde{T}_2 \right). \qedhere
    \]
    \end{itemize}
\end{proof}
\begin{corollary}\label{Cor_recover_C1_C2_2Plus1_layers}
    Let $n=2$ and $y\in[1/C_\infty, 1/C_0)$. Then,
    \[
        \lim\limits_{\omega\to+\infty} \pi \frac{N(\omega,y-\omega^{-1})-N(\omega,y)}{\sqrt{2\omega}} >0 \quad \Leftrightarrow \quad y \in \{1/\widetilde{C}_1, 1/\widetilde{C}_2\}\setminus\{1/C_\infty\}\,,
    \]
    and in such case, for $\widetilde{C}_j < C_\infty$, we have
    \[
        \pi \frac{N(\omega,1/\widetilde{C}_j-\omega^{-1})-N(\omega,1/\widetilde{C}_j)}{\sqrt{2\omega}} \underset{\omega\to+\infty}{\longrightarrow}
        \left\{
        \begin{aligned}
             &\widetilde{T}_j/\sqrt{\widetilde{C}_j} &\quad\text{if } C_1\neq C_2\,, \\
             &(\widetilde{T}_1+\widetilde{T}_2)/\sqrt{C_0} &\quad\text{if } C_1 = C_2\,.
        \end{aligned}
        \right.
    \]
\end{corollary}
This corollary means that the branches~$k_\ell(\omega)/\omega$ ``accumulate'' from below at the~$1/C_j$'s, in the sense that the number of branches below and close to~$1/C_j$ is diverging with~$\omega$, at speed~$\sqrt\omega$.
\begin{proof}
    First, at $y=1/C_\infty$, $N(\omega, y-\omega^{-1})=0$ for $\omega>0$, hence
    \[
        \pi \frac{N(\omega,y-\omega^{-1})-N(\omega,y)}{\sqrt{2\omega}} \leq 0\,.
    \]
    
    From now on, we suppose $y\in(1/C_\infty, 1/C_0)$. We have
    \[
        |\tilde{\nu}_j(y-\omega^{-1})| = \left\{
        \begin{aligned}
            &|\tilde{\nu}_j(y)| + \frac{y \omega^{-1}}{|\tilde{\nu}_j(y)|} + O\!\left(\omega^{-2}\right) &\text{if } y < \widetilde{C}_j^{-1}\,, \\
            &\sqrt{\frac{2}{\omega \widetilde{C}_j}} \left(1 - \frac{\omega^{-1}}{4\widetilde{C}_j} + O\!\left(\omega^{-2}\right)\right) &\text{if } y = \widetilde{C}_j^{-1}\,.
        \end{aligned}
        \right.
    \]
    
    Recall that at $y=1/C_0=1/\widetilde{C}_1$, we have $N(\omega,1/C_0)=0$. Therefore,
    \[
        \pi \frac{N\left(\omega,\frac{1}{C_0}-\frac{1}{\omega} \right)-N(\omega,\frac{1}{C_0})}{\sqrt{2\omega}} \sim
        \left\{
        \begin{aligned}
            &\sqrt{\frac{\omega}{2}} \left|\tilde{\nu}_1\!\left(\frac{1}{C_0}-\frac{1}{\omega}\right)\right| \widetilde{T}_1 \to \frac{\widetilde{T}_1}{\sqrt{C_0}} &\text{if } \widetilde{C}_2>C_0\,, \\
            &\sqrt{\frac{\omega}{2}} \left|\tilde{\nu}_1\!\left(\frac{1}{C_0}-\frac{1}{\omega}\right)\right| \left( \widetilde{T}_1 + \widetilde{T}_2 \right) \to \frac{\widetilde{T}_1 + \widetilde{T}_2}{\sqrt{C_0}} &\text{if } \widetilde{C}_2=C_0\,.
        \end{aligned}
        \right.
    \]
    
    At $y=1/\widetilde{C}_2<1/C_0$, we have
    \begin{multline*}
        \pi \frac{N(\omega,\widetilde{C}_2^{-1}-\omega^{-1})-N(\omega,\widetilde{C}_2^{-1})}{\sqrt{2\omega}} \\
        \sim \sqrt{\frac{\omega}{2}} \left( |\tilde{\nu}_1(\widetilde{C}_2^{-1}-\omega^{-1})| \widetilde{T}_1 + |\tilde{\nu}_2(\widetilde{C}_2^{-1}-\omega^{-1})| \widetilde{T}_2 - |\tilde{\nu}_1(\widetilde{C}_2^{-1})| \widetilde{T}_1 \right) \\
        = \frac{\widetilde{T}_2}{\sqrt{\widetilde{C}_2}} +  \frac{\omega^{-1/2}}{\sqrt{2}\widetilde{C}_2 |\tilde{\nu}_1(\widetilde{C}_2^{-1})|} \widetilde{T}_1 + O(\omega^{-1}) \to \frac{\widetilde{T}_2}{\sqrt{\widetilde{C}_2}} \,.
    \end{multline*}
    
    For $1/C_\infty<y<1/\widetilde{C}_2$, we have
    \begin{multline*}
        \pi \frac{N(\omega,y-\omega^{-1})-N(\omega,y)}{\sqrt{2\omega}} \\
        \sim \sqrt{\frac{\omega}{2}} \left( \left( |\tilde{\nu}_1(y-\omega^{-1})| - |\tilde{\nu}_1(y)| \right) \widetilde{T}_1 + \left( |\tilde{\nu}_2(y-\omega^{-1})| - |\tilde{\nu}_2(y)| \right) \widetilde{T}_2 \right) \\
        = \frac{y}{\sqrt{2\omega}} \left( \frac{\widetilde{T}_1}{|\tilde{\nu}_1(y)|} + \frac{\widetilde{T}_2}{|\tilde{\nu}_2(y)|} + O\!\left(\omega^{-1}\right) \right) \to 0\,.
    \end{multline*}
    
    For $1/\widetilde{C}_2<y<1/C_0$, we have
    \begin{multline*}
        \pi \frac{N(\omega,y-\omega^{-1})-N(\omega,y)}{\sqrt{2\omega}} \sim \sqrt{\frac{\omega}{2}} \left( |\tilde{\nu}_1(y-\omega^{-1})| - |\tilde{\nu}_1(y)| \right) \widetilde{T}_1 \\
        = \frac{y}{\sqrt{2\omega}} \left( \frac{\widetilde{T}_1}{|\tilde{\nu}_1(y)|} + O\!\left(\omega^{-1}\right) \right) \to 0\,.
    \end{multline*}
\end{proof}

In order to recover the values we are looking for, one has to plot the experimental data into a graph $(\omega, k(\omega)/\omega)$ ---see top-right figure in Figure~\ref{Fig_simu_1_6Plus1_layers_intro} for a simulated version of such data---, then to proceed as follow.

First, following Corollaries~\ref{Cor_recover_C0_Cinfty_nPlus1_layers} and~\ref{Cor_recover_C1_C2_2Plus1_layers}, one reads in the plot the three values~$1/\widetilde{C}_j$. Corollary~\ref{Cor_recover_C1_C2_2Plus1_layers} ensures that the levels $y$ of ``accumualtion'' of branches, when $\omega$ becomes large, that one reads on the plot, are the values $1/\widetilde{C}_2$ and $1/\widetilde{C}_1$ and only them: the plot contains no other levels of ``accumualtion''. In particular, if there is only one level of ``accumualtion'' (which is then necessarily at the top of the plot), then $C_1=C_2$.

Second, one retrieves $\widetilde{T}_1$ and $\widetilde{T}_2$ by evaluating the limits in Corollary~\ref{Cor_recover_C1_C2_2Plus1_layers}. In the special case $C_1=C_2$, we recover instead the sum $\widetilde{T}_1 + \widetilde{T}_2$.

Third, if we have three different values for the $C_j$'s (i.e., if $C_1\neq C_2$), we identify which layer is at the surface and which one is below it (and above the semi-infinite layer). To do so, we use (see the proof of Proposition~\ref{Prop_zeros_ftilde_any_y_2Plus1_layers}) that on the one hand if $C_1>C_2$, then the zeros at $y=1/\widetilde{C}_2=1/C_1$ are equidistant ---by $\pi/(|\bar{\nu}_2(1/C_1)| T_2)$--- while on the other hand if $C_1>C_2$, then the zeros at $y=1/\widetilde{C}_2=1/C_2$ are not equidistant (but their spacing tends to $\pi/(|\bar{\nu}_1(1/C_2)| T_1)$ from above).

Summarizing, we have identified $C_1$, $C_2$, $C_3$, $T_1$, and $T_2$ (or $T_1+T_2$ if $C_1=C_2<C_3$).

\bigskip

\noindent{\bf Acknowledgments.} J.R. thanks Mathieu Lewin for fruitful discussions. MVdH gratefully acknowledges support from the Simons Foundation under the MATH + X program, the National Science Foundation under grant DMS-1815143, and the corporate members of the Geo-Mathematical Imaging Group at Rice University. J.G. and J.R. acknowledge support from the Agence Nationale de la Recherche under Grant No. ANR-19-CE46-0007 (project ICCI). The authors thank the referees for the suggestions that helped improving the manuscript.

\newpage
\section*{Appendix}
\addtocontents{toc}{\protect\setcounter{tocdepth}{1}}
\renewcommand{\thesubsection}{\Alph{subsection}}
\makeatletter
	\renewcommand{\thetheorem}{\thesubsection.\arabic{theorem}}
	\@addtoreset{theorem}{subsection}
	
	\renewcommand{\theequation}{\thesubsection.\arabic{equation}}
	\@addtoreset{equation}{subsection}
	
	\renewcommand{\thefigure}{\thesubsection}
	\@addtoreset{figure}{subsection}
\makeatother

\subsection{Numerical simulations}\label{Appendix_simulations}
We present here numerical simulations of $k_\ell(\omega)/\omega$.

\vspace{-0.2cm}
\begin{figure}[!htp]
    \centering
    \captionsetup{width=1.2\textwidth,justification=centering,font=small}
    \captionsetup[subfigure]{labelformat=empty,skip=0pt,font=tiny}
    \makebox[\textwidth][c]{%
        \subcaptionbox{\underline{$(1+1)$-layers}\\[2pt]
        $(C_1, C_2) = (1000,10000) \qquad \qquad H_2 = 100$}
{\includegraphics[width=0.59\textwidth]{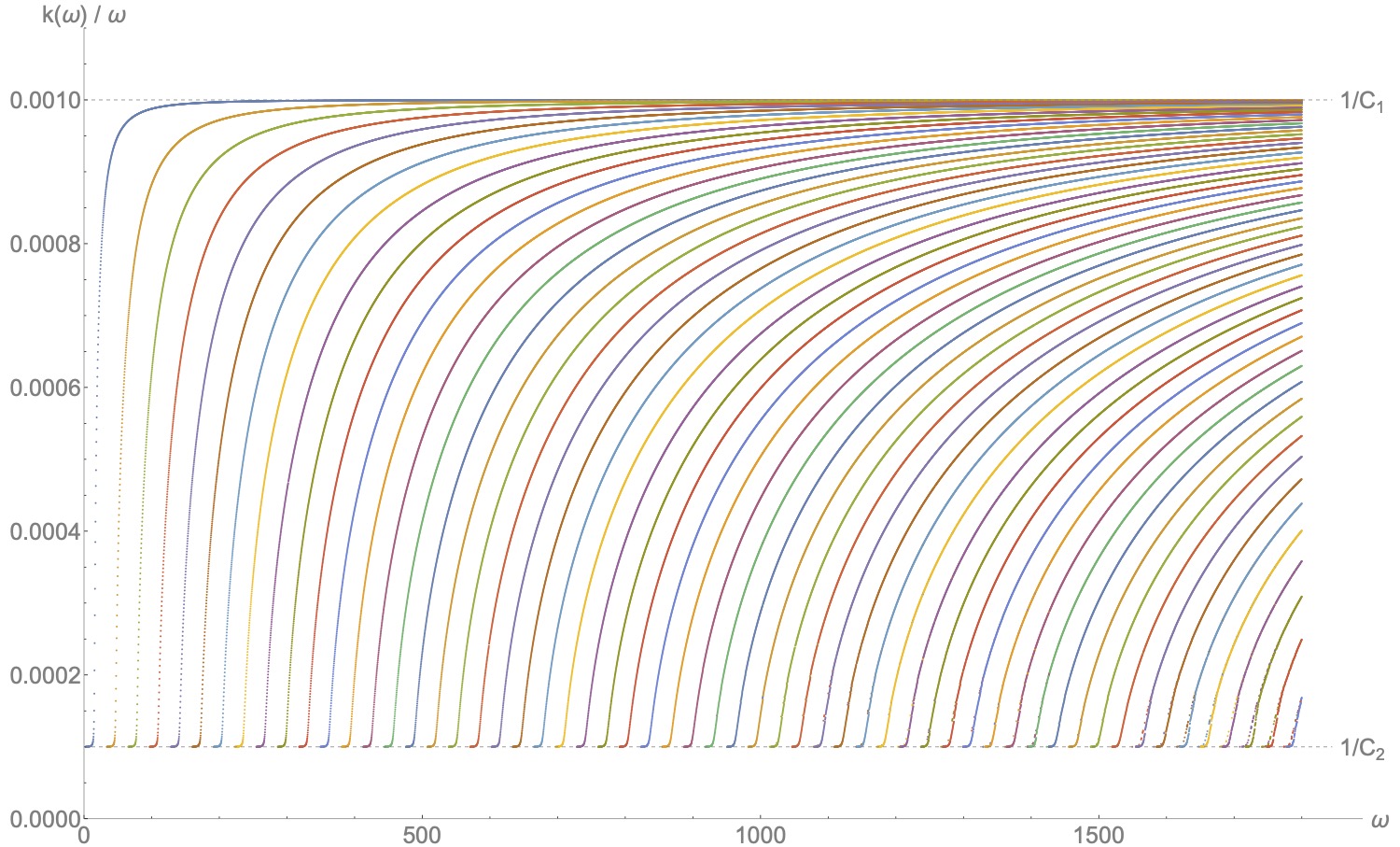}}
        \hspace{1em}
        \subcaptionbox{\underline{$(2+1)$-layers}\\[2pt]
        $(C_1, C_2, C_3) = (1000,1818,10000) \qquad \qquad (H_2, H_3) = (100, 200)$}
{\includegraphics[width=0.59\textwidth]{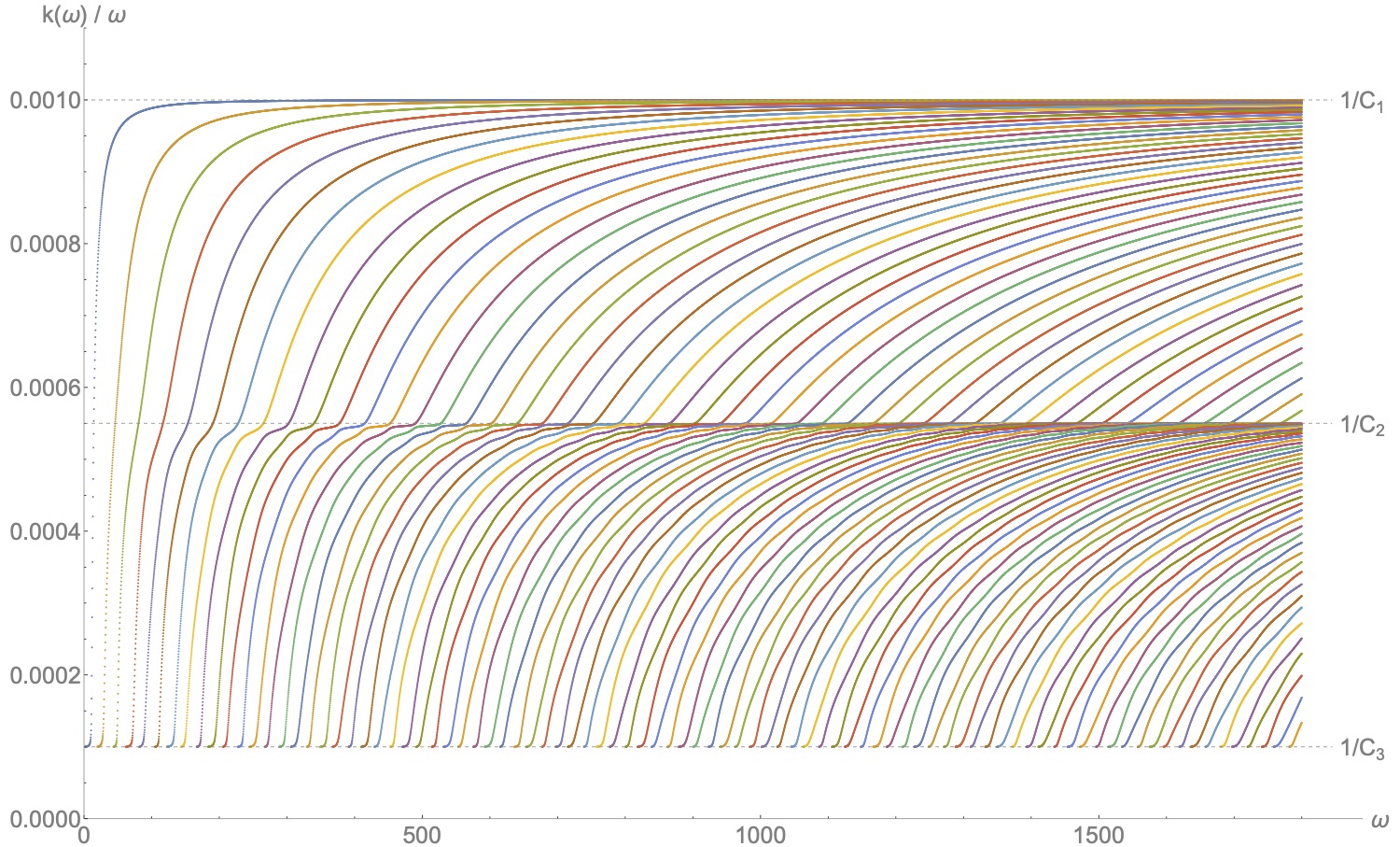}}%
    }
        
    \makebox[\textwidth][c]{%
        \subcaptionbox{\underline{$(3+1)$-layers}\\[2pt]
        $(C_1, C_2, C_3, C_4) = (1000,1429,2500,10000)$\\
        $(H_2, H_3, H_4) = (100, 200, 300)$}
{\includegraphics[width=0.59\textwidth]{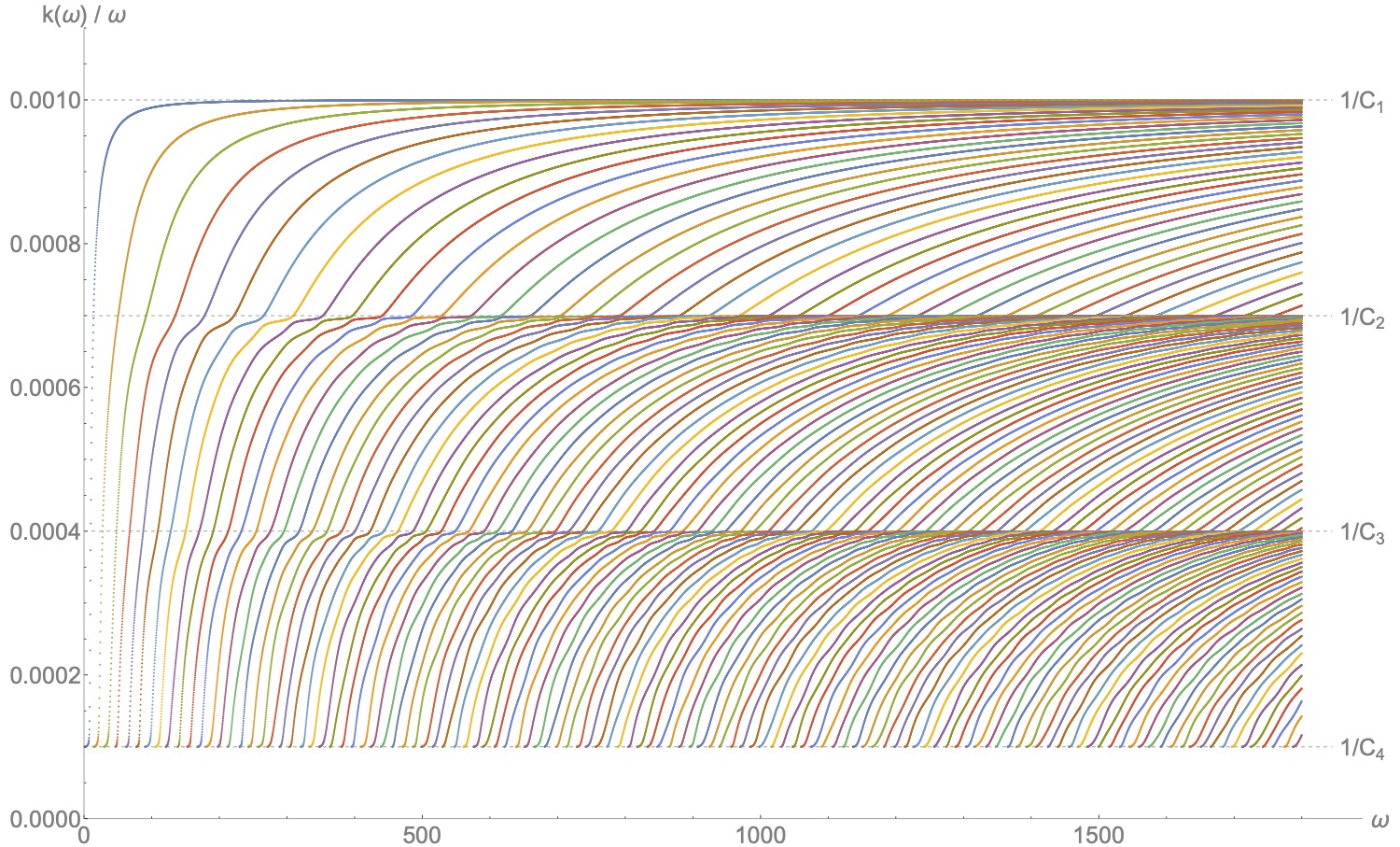}}%
        \hspace{1em}
        \subcaptionbox{\underline{$(4+1)$-layers}\\[2pt]
        $(C_1, C_2, C_3, C_4, C_5) = (1000,1290,1818,3077,10000)$\\
        $(H_2, H_3, H_4, H_5)=(100, 200, 300, 400)$}
{\includegraphics[width=0.59\textwidth]{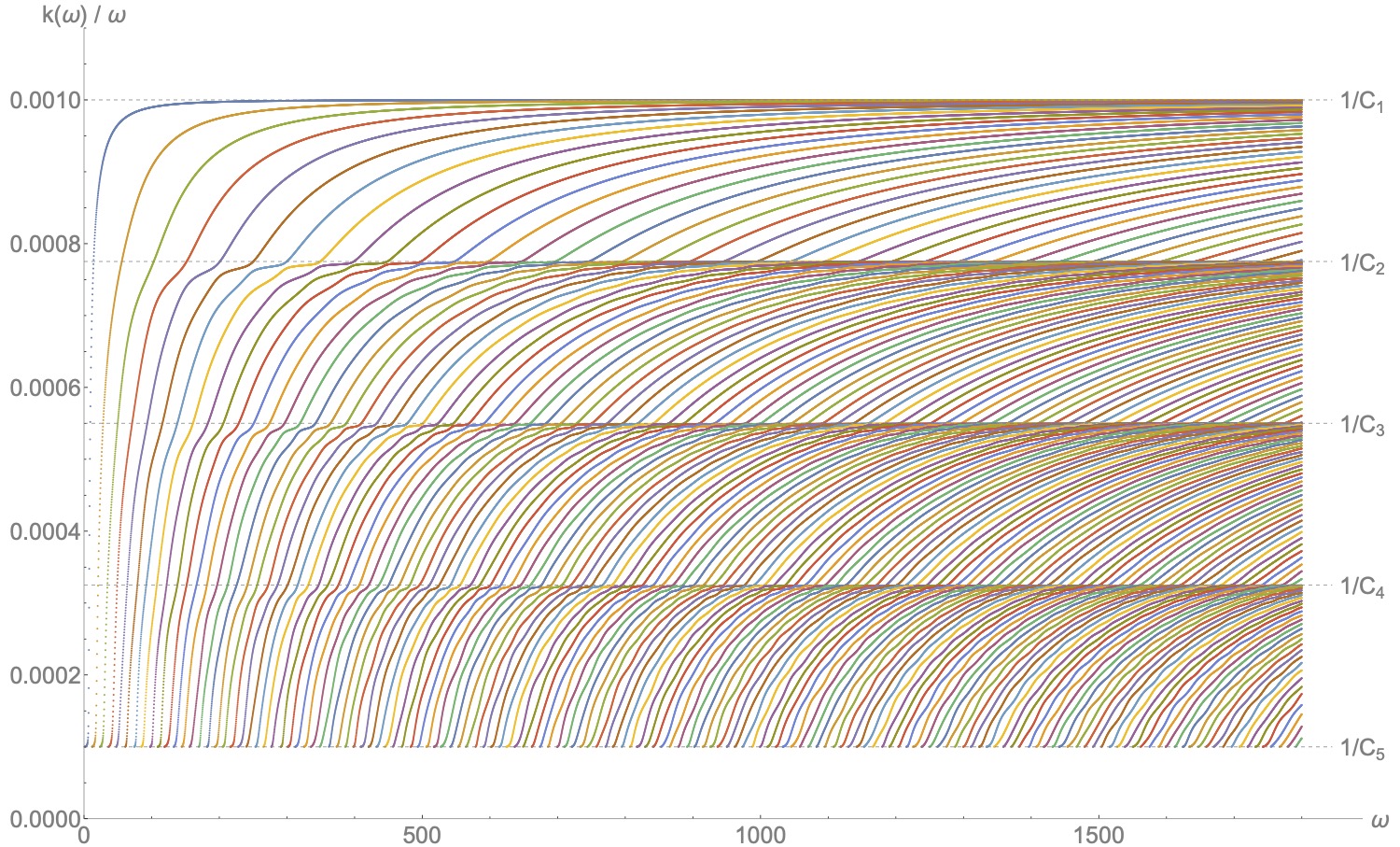}}%
    }
        
    \makebox[\textwidth][c]{%
        \subcaptionbox{\underline{$(5+1)$-layers}\\[2pt]
        $(C_1, C_2, C_3, C_4, C_5, C_6) = (1000,1220,1562,2174,3571,10000)$\\
        $(H_2, H_3, H_4, H_5, H_6)=(100, 200, 300, 400, 500)$}
{\includegraphics[width=0.59\textwidth]{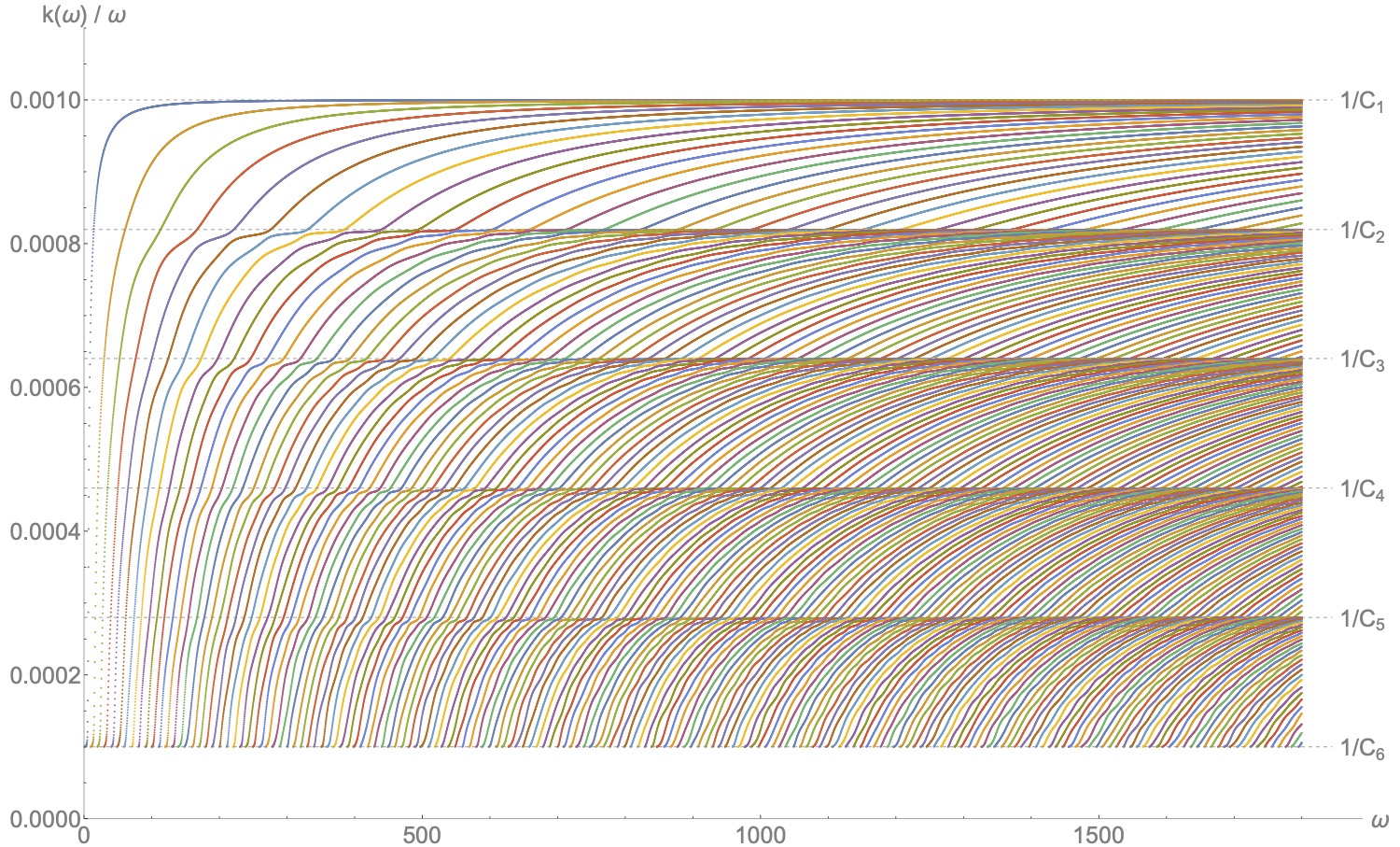}}%
        \hspace{1em}
        \subcaptionbox{\underline{$(6+1)$-layers}\\[2pt]
        $(C_1, C_2, C_3, C_4, C_5, C_6, C_7) = (1000,1176,1429,1818,2500,4000,10000)$\\
        $(H_2, H_3, H_4, H_5, H_6, H_7)=(100, 200, 300, 400, 500, 600)$}
{\includegraphics[width=0.59\textwidth]{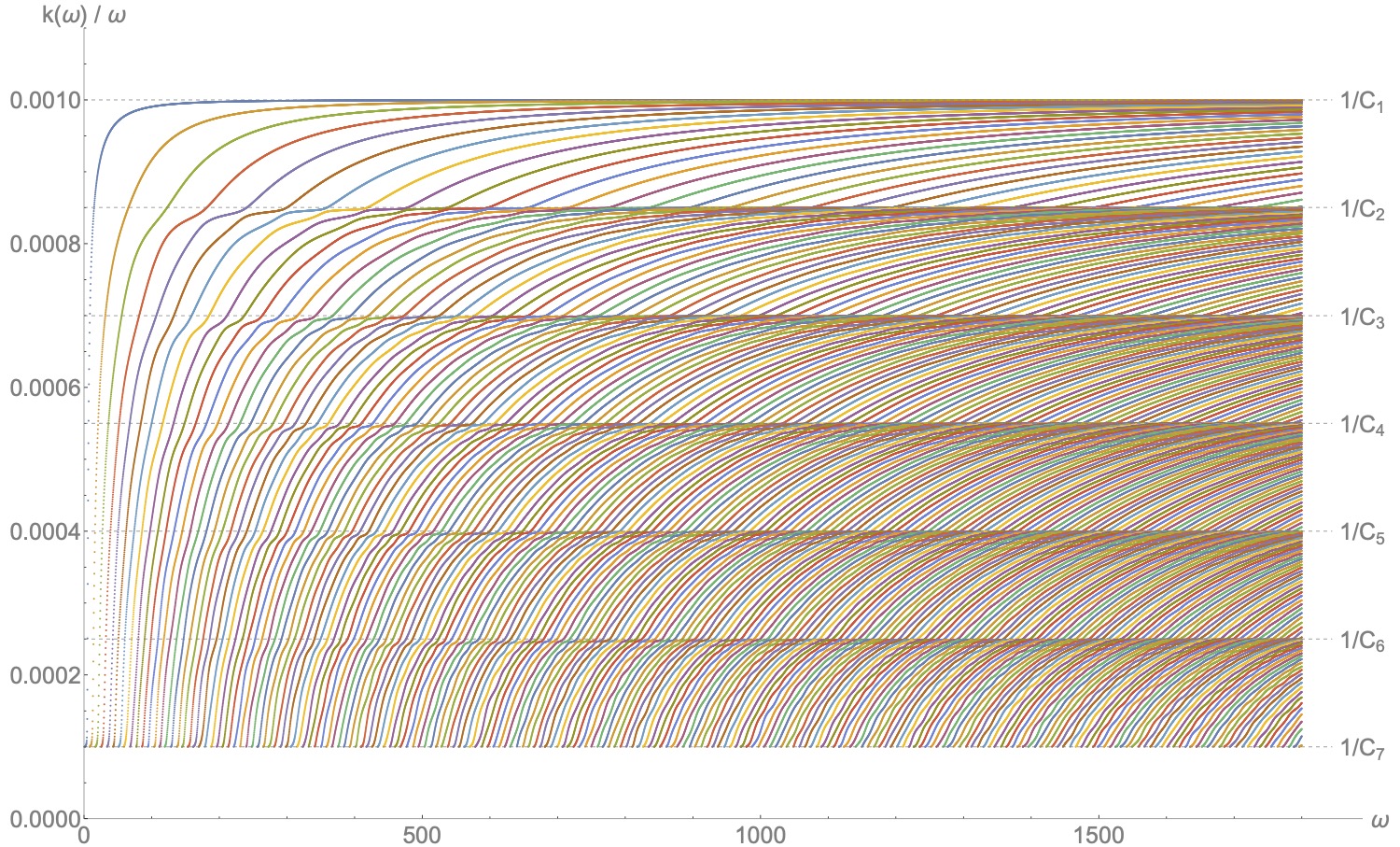}}%
    }

    \caption{Numerical computations of $\omega\mapsto \left\{k_\ell(\omega)/\omega\right\}_\ell$, from $1+1$ to $6+1$~layers.\\
    For $\omega\in(0,1800]$, increments of $0.25$, and $\rho_i=1$, $i\geq1$.
    }
    \label{Fig_simu_1_6Plus1_layers_intro}
\end{figure}

\vspace{-0.2cm}
Note that, at least for the case $(2+1)$-layers, i.e., $n=2$, similar simulations can be found in the literature, even though with less details. See e.g., in~\cite[Fig~1(a)]{BucBen-96}.
\begin{remark*}
    Contrarily to what the above figures could mislead to think, the first branch does not (necessarily) ``starts'' at~$(0,1/C_{n+1})$. As shown in Section~\ref{Section_1plus1_layers}, this is true for $n=1$, but this is generically false for $n\geq2$, for which it starts at~$(\omega_1,1/C_{n+1})$ with $\omega_1\geq0$.
\end{remark*}

\subsection{Derivation of the Love waves boundary value problem}\label{Derivation_Love_waves_problem}
We consider and present first the linear elastic wave equation (without source term) in $\R^2\times [0,+\infty)$, see e.g.,~\cite{AkiRic}. In coordinates $(\bx,z)$, $\bx=(x_1,x_2) \in \R^2$, $z=x_3\in [0,+\infty)$, we consider solutions $\bu=(u_1,u_2,u_3)$ satisfying stress-free (Neumann) boundary condition at the surface $z=0$,
\begin{equation}\label{div_equation}
    \rho \partial_{tt} \bu = \div \btau(\bu) \qquad \text{and} \qquad \btau(\bu)\bigr|_{z=0}\cdot \be_3 = 0 \,,
\end{equation}
where $\rho(\bx,z)$ is the mass density, $\bu(\bx,z,t) = (u_1(\bx,z,t), u_2(\bx,z,t), u_3(\bx,z,t))$ the displacement vector, and $\btau(\bu) = (\tau_{ij})_{1\leq i,j \leq 3}$ the Cauchy stress tensor given by Hookes’ law
\[
    \btau(\bu) = \bC \beps(\bu)\,,
\]
in which $\bC(\bx,z) = (C_{i j k \ell}(\bx,z))_{1\leq i,j,k,\ell \leq 3}$ is the stiffness tensor and $\beps = (\varepsilon_{ij})_{1\leq i,j \leq 3}$ the infinitesimal strain tensor, which is given by the strain–displacement equation
\[
    \beps(\bu) = \frac{\bnabla\bu + {\bnabla\bu}\transp}2\,.
\]
Equivalently, in terms matrices' coefficients we have
\[
    \tau_{ij}(\bu) = \sum\limits_{k,\ell=1}^3 C_{i j k \ell} \varepsilon_{k\ell}(\bu) \quad \text{ with } \quad \varepsilon_{k\ell}(\bu) = \frac{\partial_{x_k} u_\ell + \partial_{x_\ell} u_k}2\,.
\]
Note that, in particular, $\beps$ is symmetric: $\varepsilon_{ij}=\varepsilon_{ji}$ for all $i,j\in \llbracket 1, 3 \rrbracket$.

Our first but physically natural assumption is that $\bC$ is symmetric:
\[
    \bC_{i j k \ell} = \bC_{j i k \ell}=\bC_{k \ell i j}\,, \qquad \text{for all } i,j\in \llbracket 1, 3 \rrbracket\,.
\]
The $i$-th component of the elastic wave equation~\eqref{div_equation} therefore reads
\begin{equation}\label{div_equation_C_symm_ith_component}
    \rho \partial_{tt} u_i = \left( \div \btau(\bu) \right)_i = \sum_{j=1}^3 \partial_{x_j} \sum_{k,\ell=1}^3 C_{i j k \ell} \frac{\partial_{x_k} u_\ell + \partial_{x_\ell} u_k}{2} = \sum_{j, k,\ell=1}^3 \partial_{x_j} C_{i j k \ell} \partial_{x_k} u_\ell
\end{equation}
and the one of the stress-free boundary condition reads
\[
    \sum\limits_{k,\ell=1}^3 C_{i 3 k \ell} \partial_{x_k} u_\ell\biggr|_{z=0} = 0\,.
\]

We now introduce our main assumption. Namely, we assume that the medium is a time-independent, stratified medium that is homogeneous in the $(x,y)$-plane: $\rho$ and $\bC$ depend only on the variable $z$. This allows us to write the elastic wave equation~\eqref{div_equation_C_symm_ith_component} as
\[
    \rho \partial_{tt} u_i = \sum_{\ell=1}^3 \left[ \partial_z C_{i 3 3 \ell} \partial_z + \sum_{j=1}^2 C_{i j 3 \ell} \partial_{x_j} \partial_z + \sum_{k=1}^2 \partial_z C_{i 3 k \ell} \partial_{x_k} + \sum_{j, k=1}^2 C_{i j k \ell} \partial_{x_j} \partial_{x_k} \right] u_\ell\,.
\]

Introducing the time and $(x,y)$-space Fourier transform
\[
    \hat{u}_i(z) \equiv \hat{u}_i(\bxi, z, \omega) := \int_{\R^2} \int_{\R} u_i(\bx, z, t) e^{\ii \omega t} e^{\ii \bxi\cdot\bx} \di t \di\bx
\]
(and assuming that everything is well-defined), we formally obtain
\begin{equation}\label{div_equation_C_symm_only_z_ith_component_Fourier}
    -\omega^2 \rho \hat{u}_i = \sum_{\ell=1}^3 \left[ \partial_z C_{i 3 3 \ell} \partial_z + \sum_{j=1}^2 \ii \xi_j C_{i j 3 \ell} \partial_z + \sum_{k=1}^2 \ii \xi_k \partial_z C_{i 3 k \ell} \partial_{x_k} - \sum_{j, k=1}^2 \xi_k \xi_j C_{i j k \ell} \right] \hat{u}_\ell
\end{equation}
and the one of the stress-free boundary condition reads
\begin{equation}\label{div_Neumann_boundary_C_symm_only_z_ith_component_Fourier}
    \sum\limits_{\ell=1}^3 \left\{ \ii \sum\limits_{k=1}^2 C_{i 3 k \ell} \xi_k \hat{u}_\ell + C_{i 3 3 \ell} \partial_z \hat{u}_\ell \right\}\biggr|_{z=0} = 0\,.
\end{equation}

In the case of isotropic media, the stiffness tensor $\bC$ takes the form~\cite{AkiRic}
\begin{equation}\label{isotropic_stiffness}
    C_{i j k \ell} = \lambda \delta_i^j \delta_k^\ell + \mu \left( \delta_i^k \delta_j^\ell + \delta_i^\ell \delta_j^k \right) \, ,
\end{equation}
where $\lambda$ and $\mu$ are the Lam{\'e} parameters. Injecting~\eqref{isotropic_stiffness} in~\eqref{div_equation_C_symm_only_z_ith_component_Fourier}--\eqref{div_Neumann_boundary_C_symm_only_z_ith_component_Fourier}, we obtain
\[
    \left\{%
    \begin{aligned}
        0 &= \left( ( \lambda + \mu ) \xi_1^2 + \mu |\bxi|^2 - \left( \partial_z \mu \partial_z + \rho \omega^2 \right) \right) \hat{u}_1 + ( \lambda + \mu ) \xi_1 \xi_2 \hat{u}_2 - \ii \left( \lambda \partial_z + \partial_z \mu \right) \xi_1 \hat{u}_3\,, \\
        0 &= ( \lambda + \mu ) \xi_1 \xi_2 \hat{u}_1 + \left( ( \lambda + \mu ) \xi_2^2 + \mu |\bxi|^2 - \left( \partial_z \mu \partial_z + \rho \omega^2 \right) \right) \hat{u}_2 - \ii \left( \lambda \partial_z + \partial_z \mu \right) \xi_2 \hat{u}_3\,, \\
        0 &= - \ii \left( \mu \partial_z + \partial_z \lambda \right) \xi_1 \hat{u}_1 - \ii \left( \mu \partial_z + \partial_z \lambda \right) \xi_2 \hat{u}_2 + \left( \mu |\bxi|^2 - \left( \partial_z \left( \lambda + 2 \mu \right) \partial_z + \rho \omega^2 \right) \right) \hat{u}_3\,,
    \end{aligned}
    \right.
\]
that is,\\[4pt]
\makebox[\textwidth]{%
    \resizebox{0.95\textwidth}{!}{%
    $\begin{pmatrix}
        ( \lambda + \mu ) \xi_1^2 + \mu |\bxi|^2 - \left( \partial_z \mu \partial_z + \rho \omega^2 \right) & ( \lambda + \mu ) \xi_1 \xi_2 & - \ii \left( \lambda \partial_z + \partial_z \mu \right) \xi_1 \\
        ( \lambda + \mu ) \xi_1 \xi_2 & ( \lambda + \mu ) \xi_2^2 + \mu |\bxi|^2 - \left( \partial_z \mu \partial_z + \rho \omega^2 \right) & - \ii \left( \lambda \partial_z + \partial_z \mu \right) \xi_2 \\
        - \ii \left( \mu \partial_z + \partial_z \lambda \right) \xi_1 & - \ii \left( \mu \partial_z + \partial_z \lambda \right) \xi_2 & \mu |\bxi|^2 - \left( \partial_z \left( \lambda + 2 \mu \right) \partial_z + \rho \omega^2 \right)
    \end{pmatrix}
    \begin{pmatrix}
        \hat{u}_1 \\
        \hat{u}_2 \\
        \hat{u}_3
    \end{pmatrix}
    =
    \begin{pmatrix}
        0 \\
        0 \\
        0
    \end{pmatrix}$%
    }%
}\\[4pt]
and
\[
    \left\{
        \begin{aligned}
            \ii \xi_1 \hat{u}_3(0) + \partial_z \hat{u}_1(0) &= 0\,, \\
            \ii \xi_2 \hat{u}_3(0) + \partial_z \hat{u}_2(0) &= 0\,, \\
            \ii \lambda(0) \left( \xi_1 \hat{u}_1(0) + \xi_2 \hat{u}_2(0) \right) + \left( \lambda(0)+ 2 \mu(0) \right) \partial_z \hat{u}_3(0) &= 0\,.
        \end{aligned}
    \right.
\]

Introducing the orthogonal matrix
\[
    P(\bxi) :=
    \begin{pmatrix}
        \xi_2/|\bxi| & -\xi_1/|\bxi| & 0 \\
        \xi_1/|\bxi| & \xi_2/|\bxi| & 0 \\
        0 & 0 & 1
    \end{pmatrix}
\]
and $\phi := (\phi_1, \phi_2, \phi_3)\transp := P(\bxi) (\hat{u}_1, \hat{u}_2, \hat{u}_3)\transp$, we obtain the equation
\[%
    \makebox[\textwidth]{%
    \resizebox{0.95\textwidth}{!}{%
    $\begin{pmatrix}
        -\partial_z \mu \partial_z + \mu |\bxi|^2 - \rho \omega^2  & 0 & 0 \\
        0 & -\partial_z \mu \partial_z + (\lambda + 2\mu) |\bxi|^2 - \rho \omega^2 & -\ii |\bxi| \left( \lambda \partial_z + \partial_z \mu \right) \\
        0 & -\ii |\bxi| \left( \mu \partial_z + \partial_z \lambda \right) & -\partial_z (\lambda + 2\mu) \partial_z + \mu |\bxi|^2 - \rho \omega^2
    \end{pmatrix}
    \begin{pmatrix}
        \phi_1 \\
        \phi_2 \\
        \phi_3
    \end{pmatrix}
    =
    \begin{pmatrix}
        0 \\
        0 \\
        0
    \end{pmatrix}$%
    }%
    }%
\]
with the boundary condition
\[
    \left\{
        \begin{aligned}
            \partial_z \phi_1(0) &= 0\,, \\
            \ii |\bxi| \phi_3(0) + \partial_z \phi_2(0) &= 0\,, \\
            \ii \lambda(0) |\bxi| \phi_2(0) + \left( \lambda(0)+ 2 \mu(0) \right) \partial_z \phi_3(0) &= 0\,.
        \end{aligned}
    \right.
\]
In this decoupled system, the component $\phi_1$ corresponds to Love waves and $(\phi_2, \phi_3)$ to Rayleigh waves. Defining the wavenumber $k:=|\bxi|$, we have derived the boundary value problem~\eqref{Problem_equations_intro}.

\begin{remark*}
    In~\cite{deHIanNakZha-17}, the equation ---(5.2) in the paper---
    \[
        -\partial_z \hat\mu \partial_z \phi  + \hat\mu |\bxi|^2 \phi = \Lambda \phi
    \]
    is obtained for Love waves in an isotropic medium, where $\hat\mu=\mu/\rho$ and $\Lambda=\omega^2$. Our equation in~\eqref{Problem_equations_intro} differs by the presence of $\rho$ multiplying $\Lambda=\omega^2$ because we started from the true linear elastic wave equation while an approximated version of it (but equivalent from the semiclassical point of view) is considered in~\cite[Sect. 2]{deHIanNakZha-17}.
\end{remark*}

Finally, the continuity condition on the solution means that $\hat{\bu}$, hence $\phi$, is $z$-continuous, while the continuity condition on the stress components $xz$ and $yz$ means here that 
\[
    \btau_{1 3}(\bu) = \mu \left( \ii \xi_1 \hat{u}_3 + \partial_z \hat{u}_1 \right) \qquad \text{ and } \qquad \btau_{2 3}(\bu) = \mu \left( \ii \xi_2 \hat{u}_3 + \partial_z \hat{u}_2 \right)
\]
are continuous and, consequently, that $\phi_1 = \xi_2 \hat{u}_1 - \xi_1 \hat{u}_2$ satisfies $\mu \partial_z\phi_1$ continuous.

\subsection{Postponed proofs in the general case}\label{Appendix_nplus1_layers}
\begin{proof}[Proof of Lemma~\ref{Continuity_tilde_functions_nPlus1_layers}]
    We prove it for $\tilde{P}_m$ and $\tilde{Q}_m$, then the result for $\tilde{f}_n$ is an immediate consequence.
    
    First, $\tilde{P}_0 \equiv 1$ and $\tilde{Q}_0 \equiv 0$ on $[0,+\infty)\times[1/C_\infty, 1/C_0)$, and they are obviously continuous.
    
    Now, with the $\bar{M}_m$'s defined in~\eqref{Def_M_bar}, for which on $[0,+\infty)\times[1/C_\infty, 1/C_0)$ we have
    \[
        \begin{pmatrix} \tilde{P}_m(\omega,y) \\ \tilde{Q}_m(\omega,y) \end{pmatrix} =
        \bar{M}_m(\omega,y)
        \begin{pmatrix} \tilde{P}_{m-1}(\omega,y) \\ \tilde{Q}_{m-1}(\omega,y) \end{pmatrix}, \qquad \forall\, m\in\llbracket1,n\rrbracket\,,
    \]
    we see that the $\bar{M}_m$'s are continuous on $(0,+\infty)\times[1/C_\infty, 1/C_0)$, as $\tilde{P}_0$ and $\tilde{Q}_0$ are, and an induction immediately gives the claimed continuity on this domain. We are left with proving the continuity at $(0,y_0)$, $y_0\in[1/C_\infty, 1/C_0)$, which we also do by induction.
    
    The result holds for $m=0$ as explained earlier. Assume now that $\tilde{P}_{m-1}$ and $\tilde{Q}_{m-1}$, $m\in\llbracket1,n\rrbracket$, are continuous at $(0,y_0)$. 
    Then, $(\tilde{P}_m(0,y_0),\tilde{Q}_m(0,y_0))=(1,0)$ and, writing
    \begin{align*}
        \begin{pmatrix} \tilde{P}_m(\omega,y) \\ \tilde{Q}_m(\omega,y) \end{pmatrix} - \begin{pmatrix} 1 \\ 0 \end{pmatrix} =
        \bar{M}_m(\omega,y)
        \left( \begin{pmatrix} \tilde{P}_{m-1}(\omega,y) \\ \tilde{Q}_{m-1}(\omega,y) \end{pmatrix}  - \begin{pmatrix} 1 \\ 0 \end{pmatrix} \right) + \left(\bar{M}_m(\omega,y) - I_2 \right) \begin{pmatrix} 1 \\ 0 \end{pmatrix},
    \end{align*}
    we obtain the wanted result as $(\omega,y)\to(0,y_0)$, since
    \begin{multline*}
        \norm{\begin{pmatrix} \tilde{P}_m(\omega,y) - 1 \\ \tilde{Q}_m(\omega,y) \end{pmatrix}}_\infty \leq
        \norm{\bar{M}_m(\omega,y)}_\infty
        \norm{\begin{pmatrix} \tilde{P}_{m-1}(\omega,y) - 1 \\ \tilde{Q}_{m-1}(\omega,y) \end{pmatrix}}_\infty \\
        + \norm{\begin{pmatrix} \cosh[\omega\bar{\nu}_m(y) T_m] - 1 \\ \mu_m \bar{\nu}_m(y) \sinh[\omega\bar{\nu}_m(y) T_m] \end{pmatrix}}_\infty
    \end{multline*}
    and $\norm{\bar{M}_m(\omega,y)}_\infty$ is uniformly bounded on any neighborhood of $(0,y_0)$.
\end{proof}
\begin{remark*}
    Notice that even though the $\bar{M}_m$'s are not continuous since
    \[
        \bar{M}_m(0,1/C_m) = I_2 \neq
        \begin{pmatrix}
            1 & T_m / \mu_m \\
            0 & 1
        \end{pmatrix} = \bar{M}_m(\omega,1/C_m)\,, \qquad \forall\,\omega>0\,,
    \]
    the $\tilde{P}_m$'s and the $\tilde{Q}_m$'s are continuous.
\end{remark*}

\begin{proof}[Proof of Lemma~\ref{Equivalence_zeros_f_n_and_tilde_f_n}]
    Since the zeros of $f_n$ are in $(0,+\infty)\times[\omega / C_\infty, \omega / C_0)$ by Lemma~\ref{Alt_Def_kell_nPlus1_layers} and $\tilde{f}_n(\omega, y) = \omega^{-1} f_n(\omega, \omega y)$ on $(0,+\infty)\times[ 1 / C_\infty, 1 / C_0)$ by definition of $\tilde{f}_n$, we have
    \[
        \left\{ (\omega,y): (\omega,\omega y) \in \ker f_n \right\} = \ker \tilde{f}_n \cap (0,+\infty)\times[ 1 / C_\infty, 1 / C_0)\,.
    \]
    
    Since $\ker \tilde{f}_n \subset [0,+\infty)\times[ 1 / C_\infty, 1 / C_0)$ by definition of $\tilde{f}_n$, we are left with proving that
    \[
        \ker \tilde{f}_n \cap \{0\}\times[ 1 / C_\infty, 1 / C_0) = \{(0,1/C_\infty)\} \,.
    \]
    By definitions, $\tilde{P}_n(0,y)=1$ and $\tilde{Q}_n(0,y)=0$ hence $\tilde{f}_n(0,y)=\mu_\infty \bar{\nu}_\infty(y) = 0$. Hence, $\tilde{f}_n(0,y) > 0$ if $y>1/C_\infty$ and $\tilde{f}_n(0,1/C_\infty) = 0$, since $\bar{\nu}_\infty(y) > 0$ if $y>1/C_\infty$ and $\bar{\nu}_\infty(1/C_\infty)=0$. This concludes the proof.
\end{proof}

\subsection{Supplementary results for the simple square well}\label{Appendix_1plus1_layers}

We recall that $\nu_1=\nu_0$ and $\nu_2=\nu_\infty$ are defined in~\eqref{Def_nu}.
More precisely, $\nu_1(\omega, k)\in \R_+ \cup i\R_-$ and, by Section~\ref{Section_introduction}, for a Love wave to exist at~$(\omega, k)$ we must have~$\nu_2(\omega, k)>0$.

\subsubsection{Alternative formula giving \texorpdfstring{$H$}{H}}
\begin{proposition}
    Let $n=1$ and $\ell\geq1$. The function
    \begin{align*}
        \left( \omega_\ell, +\infty \right) &\to \left( (\ell-1)\frac{\pi}{H}, (\ell-1)\frac{\pi}{H} + \frac{\pi}{2H} \right) \\
        \omega &\mapsto \omega Y_\ell(\omega) = |\nu_1(\omega, k_\ell(\omega))| = \sqrt{\frac{\omega^2}{C_1^2} - k_\ell(\omega)^2}\,.
    \end{align*}
    is smooth, bijective, increasing. Consequently,
    \[
        \forall\, \ell\,,\, H = \lim\limits_{\omega\to+\infty} \frac{\pi}{ \sqrt{\frac{\omega^2}{C_1^2} - k_{\ell+1}(\omega)^2} - \sqrt{\frac{\omega^2}{C_1^2} - k_\ell(\omega)^2} }
        \,.
    \]
\end{proposition}
\begin{proof}
    Since $\sin[ \omega Y_\ell(\omega) H ] \cos[ \omega Y_\ell(\omega) H ] > 0$ by~\eqref{nPlus1_layers_relation_defining_Love_waves}, which reads
    \[
        \mu_1 |\nu_1(\omega, k)| \sin[ |\nu_1(\omega, k)| H ] = \mu_2 \nu_2(\omega, k) \cos[ |\nu_1(\omega, k)| H ] \quad \text{ and } \quad \omega/C_2 < k < \omega/C_1
    \]
    for $n=1$, and since $\omega\mapsto\omega Y_\ell(\omega)$ is continuous by Proposition~\ref{Proposition_1plus1layer}, for any $\ell$ there exists a $p$ such that $H \omega Y_\ell(\omega)\in (p\pi, p\pi + \pi/2)$ for $\omega>\omega_\ell$.
    From the formula of $\omega_\ell$ and the fact that $\lim_{\omega\searrow\omega_\ell} k_\ell(\omega)/\omega = 1/C_2$ by Proposition~\ref{Proposition_1plus1layer}, we obtain $\lim_{\omega\searrow\omega_\ell} \omega Y_\ell(\omega) H = (\ell-1)\pi$, hence $p=\ell-1$ and
    \begin{equation}\label{1Plus1_layers_proof_altern_formula_H_interval_var_of_trigo}
        \forall\, \ell\,, \quad \forall\, \omega>\omega_\ell\,, \quad \omega Y_\ell(\omega) H \in \left( (\ell-1)\pi, (\ell-1)\pi + \frac{\pi}{2} \right).
    \end{equation}
    We are therefore left with proving that the function is strictly increasing, which we do again by the IFT. Let us define, as before for shortness,
    \[
        r(\omega,Z) := \sqrt{\frac{C_2^2-C_1^2}{C_1^2 C_2^2} \omega^2 H^2 - Z^2} >0
    \]
    on $(\omega_\ell,+\infty) \times \left( (\ell-1)\pi, (\ell-1)\pi + \pi/2 \right)$, and $g_\ell:(\omega_\ell,+\infty) \times \left( (\ell-1)\pi, (\ell-1)\pi + \pi/2 \right)$ by
    \[
        g_\ell(\omega,Z) = \mu_1 Z \sin Z - \mu_2 r(\omega,Z) \cos Z\,.
    \]
    For any $\omega_\star>\omega_\ell$, $g_\ell(\omega_\star,\omega_\star Y_\ell(\omega_\star) H)=0$ and we have
    \[
        \frac{\di g_\ell}{\di Y}(\omega,Z) = \left[ \mu_1 + \mu_2 r(\omega,Z) \right] \sin Z + \left[ \mu_1 + \frac{\mu_2}{r(\omega,Z)} \right] \cos Z\,,
    \]
    hence $\frac{\di g_\ell}{\di Y}(\omega_\star,\omega_\star Y_\ell(\omega_\star) H) \neq 0$ again since $\sin[ \omega_\star Y_\ell(\omega_\star) H ] \cos[ \omega_\star Y_\ell(\omega_\star) H ] > 0$.
    
    Therefore, by the IFT there exists a neighborhood $U$ of $\{\omega_\star\}$ s.t.\ there exists a unique $h\in C^1(U,\R)$ with $h(\omega_\star)=\omega_\star Y_\ell(\omega_\star) H$ and $g_\ell(\omega,h(\omega))=0$ on $U$. This $h$ is (on $U$) exactly $\omega\mapsto\omega Y_\ell(\omega) H$ by definition of~$k_\ell$. Moreover, on $U$,
    \begin{align*}
        h'(\omega) &= - \left(\frac{\di g_\ell}{\di Y}(\omega,h(\omega))\right)^{-1} \frac{\di g_\ell}{\di \omega}(\omega,h(\omega)) \\
        &= \frac{ \rho_2 (C_2^2-C_1^2)  H^2 \omega }{ C_1^2 r(\omega,h(\omega)) \left( \left[ \mu_1 + \frac{\mu_2}{r(\omega,h(\omega))} \right] + \left[ \mu_1 + \mu_2 r(\omega,h(\omega)) \right] \tan [ h(\omega) ] \right) }\,,
    \end{align*}
    hence $h'(\omega_\star) > 0$, again since $\tan [\omega_\star Y_\ell(\omega_\star) H]>0$, concluding the proof of the first claim.
    Moreover, by a bootstrapping argument, we obtain that $h$ is smooth.
    
    The second claim is a direct consequence of the first result.
\end{proof}

\subsubsection{On the behaviour of the $k_\ell(\omega)$'s when $\omega\searrow\omega_\ell$}
\begin{proposition}
    Let $n=1$, the $k_\ell(\omega)$'s be defined in Definition~\ref{Def_kell_nPlus1_layers}, and the $\omega_\ell$'s as in Proposition~\ref{Proposition_1plus1layer}. Define $y_\ell:(\omega_\ell,+\infty) \to (1/C_2, 1/C_1)$ by $y_\ell(\omega) = k_\ell(\omega)/\omega$. Then,
    \[
        \lim\limits_{\omega\searrow\omega_\ell} y_\ell'(\omega) = 0\,.
    \]
\end{proposition}
\begin{proof}
    By definition, $\tilde{f}_1(\omega,y_\ell(\omega))=0$. Differentiating $\omega \mapsto\tilde{f}_1(\omega,y_\ell(\omega))$, we obtain
    \begin{multline*}
        \begin{multlined}[t][0.9\textwidth]
            y_\ell'(\omega) y_\ell(\omega) \left[ \left( \mu_2 H \omega \bar\nu_2(y_\ell(\omega)) - \mu_1 \right)\frac{\sin[ \omega |\bar\nu_1(y_\ell(\omega))| H ]}{|\bar{\nu}_1(y_\ell(\omega))|} \right.\\
            \left. + \left(\frac{\mu_2}{\bar\nu_2(y_\ell(\omega))} + \mu_1 H \omega \right) \cos[ \omega |\bar\nu_1(y_\ell(\omega))| H ] \right]
        \end{multlined} \\
        = \mu_2 H \bar\nu_2(y_\ell(\omega)) |\bar\nu_1(y_\ell(\omega))|\sin[\omega |\bar\nu_1(y_\ell(\omega))| H ] + \mu_1 H |\bar\nu_1(y_\ell(\omega))|^2 \cos[ \omega |\bar\nu_1(y_\ell(\omega))| H ]\,,
    \end{multline*}
    where $|\bar\nu_1(y_\ell(\omega))|, \bar\nu_2(y_\ell(\omega))>0$ since $|\bar\nu_1(y_\ell(\omega))|<y_\ell(\omega)<1/C_1$. Multiplying both sides by $\bar\nu_2(y_\ell(\omega))>0$ then taking the limit $\omega\searrow\omega_\ell = \frac{\ell-1}{\sqrt{1/C_1^2 - 1/C_2^2}} \frac{\pi}{H}$, we obtain
    \[
        (-1)^{\ell-1} \frac{\mu_2}{C_2} \lim_{\omega\searrow\omega_\ell} y_\ell'(\omega) = 0\,,
    \]
    where we used that $\lim_{\omega_\ell} y_\ell = 1/C_2$, hence $\lim_{\omega_\ell} \bar\nu_2\circ y_\ell = 0$, $\lim_{\omega_\ell} |\bar\nu_1|\circ y_\ell = \sqrt{1/C_1^2-1/C_2^2}$, and $\lim_{\omega\searrow\omega_\ell} \omega |\bar\nu_1(y_\ell(\omega))| H = (\ell-1) \pi$.
\end{proof}


\begin{thebibliography}{10}
	\bibitem{AkiRic}
	{\sc K.~Aki and P.~G. Richards}, {\em Quantitative seismology}, University Science Books, 2nd~ed., 2002.

	\bibitem{BosDzi-99}
	{\sc L.~Boschi and A.~M. Dziewonski}, {\em High- and low-resolution images of the {E}arth's mantle: Implications of different approaches to tomographic modeling}, J.~Geophys. Res. Solid Earth, 104 (1999), pp.~25567--25594.

	\bibitem{BucBen-96}
	{\sc P.~W. Buchen and R.~Ben-Hador}, {\em Free-mode surface-wave computations}, Geophys. J.~Int., 124 (1996), pp.~869--887.

	\bibitem{deHIanNakZha-17}
	{\sc M.~de~Hoop, A.~Iantchenko, G.~Nakamura, and J.~Zhai}, {\em Semiclassical analysis of elastic surface waves}.
\newblock \href{https://arxiv.org/abs/1709.06521}{arXiv:1709.06521}, 2017.

	\bibitem{deHGarSol-22}
	{\sc M.~V. de~Hoop, J.~Garnier, and K.~S{\o}lna}, {\em System of radiative transfer equations for coupled surface and body waves}, Z.~Angew. Math. Phys., 73 (2022), p.~177.

	\bibitem{deHIanvdHZha-20a}
	{\sc M.~V. de~Hoop, A.~Iantchenko, R.~D. van~der Hilst, and J.~Zhai}, {\em Semiclassical inverse spectral problem for seismic surface waves in isotropic media: part {I}. {L}ove waves}, Inverse Probl., 36 (2020), p.~075015.

	\bibitem{EksTroLar-97}
	{\sc G.~Ekstr{\"o}m, J.~Tromp, and E.~W.~F. Larson}, {\em Measurements and global models of surface wave propagation}, J.~Geophys. Res. Solid Earth, 102 (1997), pp.~8137--8157.

	\bibitem{GarPap}
	{\sc J.~Garnier and G.~Papanicolaou}, {\em Passive imaging with ambient noise}, Cambridge University Press, 2016.

	\bibitem{GilBac-66}
	{\sc F.~Gilbert and G.~E. Backus}, {\em Propagator matrices in elastic wave and vibration problems}, Geophysics, 31 (1966), pp.~326--332.

	\bibitem{Haskell-53}
	{\sc N.~A. Haskell}, {\em The dispersion of surface waves on multilayered media}, Bull. Seismol. Soc. Am., 43 (1953), pp.~17--34.

	\bibitem{KasElSBosMeiRosBelCriWei-18}
	{\sc E.~D. K{\"a}stle, A.~El-Sharkawy, L.~Boschi, T.~Meier, C.~Rosenberg, N.~Bellahsen, L.~Cristiano, and C.~Weidle}, {\em Surface wave tomography of the {A}lps using ambient-noise and earthquake phase velocity measurements}, J.~Geophys. Res. Solid Earth, 123 (2018), pp.~1770--1792.

	\bibitem{Kato}
	{\sc T.~Kato}, {\em Perturbation Theory for Linear Operators}, Class. Math., Springer Berlin, Heidelberg, reprint of the corr. print. of the 2nd ed. 1980~ed., 1995.

	\bibitem{Kennett-83}
	{\sc B.~L.~N. Kennett}, {\em Seismic Wave Propagation in Stratified Media}, Cambridge University Press, {F}irst~ed., 1983.

	\bibitem{Kennett-83-Edition2009}
\leavevmode\vrule height 2pt depth -1.6pt width 23pt, {\em Seismic Wave Propagation in Stratified Media}, ANU Press, {N}ew~ed., 2009.

	\bibitem{Knopoff-64}
	{\sc L.~Knopoff}, {\em A matrix method for elastic wave problems}, Bull. Seismol. Soc. Am., 54 (1964), pp.~431--438.

	\bibitem{LinMosRit-08}
	{\sc F.-C. Lin, M.~P. Moschetti, and M.~H. Ritzwoller}, {\em Surface wave tomography of the western {U}nited {S}tates from ambient seismic noise: {R}ayleigh and {L}ove wave phase velocity maps}, Geophys. J.~Int., 173 (2008), pp.~281--298.

	\bibitem{MosTar-95}
	{\sc K.~Mosegaard and A.~Tarantola}, {\em {M}onte {C}arlo sampling of solutions to inverse problems}, J.~Geophys. Res. Solid Earth, 100 (1995), pp.~12431--12447.

	\bibitem{NisMonKaw-09}
	{\sc K.~Nishida, J.-P. Montagner, and H.~Kawakatsu}, {\em Global surface wave tomography using seismic hum}, Science, 326 (2009), p.~112.

	\bibitem{ReeSim4}
	{\sc M.~Reed and B.~Simon}, {\em Analysis of operators}, vol.~IV of Methods of modern mathematical physics, Academic Press, 1~ed., 1978.

	\bibitem{RhiRom-04}
	{\sc J.~Rhie and B.~Romanowicz}, {\em Excitation of {E}arth's continuous free oscillations by atmosphere--ocean--seafloor coupling}, Nature, 431 (2004), pp.~552--556.

	\bibitem{ShaCam-04}
	{\sc N.~M. Shapiro and M.~Campillo}, {\em Emergence of broadband {R}ayleigh waves from correlations of the ambient seismic noise}, Geophys. Res. Lett., 31 (2004).

	\bibitem{ShaCamSteRit-05}
	{\sc N.~M. Shapiro, M.~Campillo, L.~Stehly, and M.~H. Ritzwoller}, {\em High-resolution surface-wave tomography from ambient seismic noise}, Science, 307 (2005), pp.~1615--1618.

	\bibitem{ShaRit-02}
	{\sc N.~M. Shapiro and M.~H. Ritzwoller}, {\em {M}onte-{C}arlo inversion for a global shear-velocity model of the crust and upper mantle}, Geophys. J.~Int., 151 (2002), pp.~88--105.

	\bibitem{Tarantola-05}
	{\sc A.~Tarantola}, {\em Inverse problem theory and methods for model parameter estimation}, Other Titles in Appl. Math., SIAM, 2005.

	\bibitem{Thomson-50}
	{\sc W.~T. Thomson}, {\em Transmission of elastic waves through a stratified solid medium}, J.~Appl. Phys., 21 (1950), pp.~89--93.

	\bibitem{TraWoo-96}
	{\sc J.~Trampert and J.~H. Woodhouse}, {\em High resolution global phase velocity distributions}, Geophys. Res. Lett., 23 (1996), pp.~21--24.
\end{thebibliography}
\end{document}